\documentclass{article}
\usepackage[utf8]{inputenc}
\usepackage{amsfonts}
\usepackage{amsmath}
\usepackage{amssymb}
\usepackage{amsthm}
\usepackage{physics}
\usepackage{diagbox}
\usepackage[
    top    = 0.75in,
    bottom = 0.75in,
    left   = 0.75in,
    right  = 0.75in]{geometry}
\usepackage{xcolor}
\usepackage{hyperref}
\usepackage{nccmath}
\usepackage{multicol}
\usepackage{mathtools}
\usepackage{graphicx}
\usepackage{float}
\usepackage{subcaption}
\usepackage{tikz}
\usepackage{pgfplots}
\usepackage{mathrsfs}
\usepackage{diagbox}
\usepackage{faktor}
\usepackage{lipsum}
\usepackage{longtable}
\graphicspath{ {./images/} }

\def\code#1{\color{black}\texttt{#1}\color{black}}
\newcommand\numberthis{\addtocounter{equation}{1}\tag{\theequation}}

\newtheorem{theorem}{Theorem}[section]
\newtheorem{conjecture}[theorem]{Conjecture}
\newtheorem{corollary}[theorem]{Corollary}
\newtheorem{definition}[theorem]{Definition}
\newtheorem{lemma}[theorem]{Lemma}

\newtheorem*{remark}{Remark}
\theoremstyle{definition}

\title{Higher Weight Generalized Dedekind Sums}
\author{\vspace{-1mm}\small Preston Tranbarger \\ \vspace{-2mm}\scriptsize\href{mailto:preston.tranbarger@rutgers.edu}{preston.tranbarger@rutgers.edu} \\ \vspace{-2mm}\scriptsize Department of Mathematics, Rutgers University \\ \vspace{-2mm}\scriptsize 110 Frelinghuysen Rd. Piscataway, NJ 08854-8019}
\date{}

\begin{document}

\maketitle

\begin{abstract}
Building upon the work of Stucker, Vennos, and Young we derive generalized Dedekind sums arising from period integrals applied to holomorphic Eisenstein series attached to pairs of primitive non-trivial Dirichlet characters. Furthermore, we explore a variety of properties of these generalized Dedekind sums: we develop a finite sum formula, demonstrate their behavior as quantum modular forms, provide a Fricke reciprocity law, and characterize analytic and arithmetic aspects of their image. Particularly, for the arithmetic aspect of the image, we generalize an existing conjecture to the higher weight case and provide significant computational evidence to support this generalized conjecture.
\end{abstract}

\section{Introduction and Basic Definitions}
    The classical Dedekind sum is well studied due to its wide range of applications within mathematics and even specific subfields of physics. For more background on classical Dedekind sums we refer the reader to \cite{rademacher-grosswald}, and for more background on how the Dedekind sum appears in other areas of study we refer the reader to \cite{atiyah}.

    Let $h$ and $k$ be coprime integers with $k>0$. The classical Dedekind sum is defined as
    $$s(h,k)=\sum_{n=1}^kB_1\left(\frac{n}{k}\right)B_1\left(\frac{hn}{k}\right),$$
    where $B_1$ is the first periodic Bernoulli polynomial which is defined as follows.
    \begin{definition}
        The periodic Bernoulli polynomials for integer $k\geq 1$ are given by the formula
        $$B_k(x)=\begin{dcases}
            \sum_{m=0}^k\sum_{n=0}^m(-1)^n\binom{m}{n}\frac{(\{x\}+n)^k}{m+1} & x\in\mathbb{R}\setminus\mathbb{Z}, \\
            0 & x\in\mathbb{Z}.
        \end{dcases}$$
    \end{definition}
    There are many generalizations of the classical Dedekind sum, and we will provide more discussion on relevant generalizations towards the end of this section; but, for the purposes of this paper we build upon the work of \cite{SVY}. As stated in Theorem $1.2$ of \cite{SVY}, given primitive non-trivial Dirichlet characters $\chi_1$ and $\chi_2$ with conductors $q_1$ and $q_2$ respectively such that $\chi_1\chi_2(-1)=(-1)^k$ and $\begin{psmallmatrix}
        a & b \\
        c & d
    \end{psmallmatrix}\in\Gamma_0(q_1q_2)$, we have
    $$S_{\chi_1,\chi_2,2}(\gamma)=\mathop{\sum\sum}_{\substack{j\bmod c \\ n\bmod q_1}}\overline{\chi_1}(n)\overline{\chi_2}(j)B_1\left(\frac{j}{c}\right)B_1\left(\frac{aj}{c}+\frac{n}{q_1}\right).$$
    Conveniently, $S_{\chi_1,\chi_2,2}$ satisfies an elegant crossed homomorphism relation as shown in Lemma $2.2$ of \cite{SVY}. Namely, for $\gamma_1,\gamma_2\in\Gamma_0(q_1q_2)$ we have that
    \begin{equation}\label{crossHomInit}S_{\chi_1,\chi_2,2}(\gamma_1\gamma_2)=S_{\chi_1,\chi_2,2}(\gamma_1)+\psi(\gamma_1)\,S_{\chi_1,\chi_2,2}(\gamma_2)\end{equation}
    where $\psi(\gamma)=\chi_1\overline{\chi_2}(d_\gamma)$ is the central character of the weight $k=2$ holomorphic Eisenstein series defined below.
    \begin{definition}\label{youngEisensteinDef}
        We define our holomorphic weight $k\geq 3$ Eisenstein series attached to primitive non-trivial Dirichlet characters such that $\chi_1\chi_2(-1)=(-1)^k$ as
        $$E_{\chi_1,\chi_2,k}(z)=\frac{(q_2/(2\pi))^k\,L(k,\chi_1\chi_2)\,\Gamma(k)}{i^{-k}\,\tau(\chi_2)}\sum_{(c,d)=1}\frac{\chi_1(c)\chi_2(d)}{(cq_2z+d)^k}.
        $$
        The central character of this series is $\psi=\chi_1\overline{\chi_2}$. For $k=2$, this sum doesn't converge absolutely, so we take the Fourier expansion in Lemma \ref{eisensteinSeries} to be our definition in this case.
    \end{definition}
    \begin{remark}
        Given a matrix $\gamma\in\Gamma_0(q_1q_2)$ this Eisenstein series has the relation under Mobius transformations given by $E_{\chi_1,\chi_2,k}(\gamma z)=\psi(\gamma)\,j(\gamma,z)^k\,E_{\chi_1,\chi_2,k}(z)$.
    \end{remark}
    % \cite{youngExplicitCalc}, Equation (3.9) and Section $4.4$
    \noindent This particular normalization of our Eisenstein series makes the Fourier expansion have simple coefficients.
    \begin{lemma}\label{eisensteinSeries}
        We have the following Fourier series for $E_{\chi_1,\chi_2,k}$:
        $$E_{\chi_1,\chi_2,k}(z)=2\sum_{1\leq N}\sum_{A\mid N}\chi_1(A)\overline{\chi_2}(N/A)(N/A)^{k-1}e(Nz)\quad\text{where}\quad e(z)=\exp(2\pi iz).$$
    \end{lemma}
    \begin{remark}
        These definitions of the holomorphic Eisenstein series and its Fourier expansion are standard; however, many authors use different normalizations and so care is required in checking the above is consistent. Some sources include \cite[Theorem $7.1.3$]{Miyake}, \cite[Theorem $4.5.1$]{diamondShurman}, and \cite[(3.9) and Section $4.4$]{youngExplicitCalc}.
    \end{remark}
    % \cite{youngExplicitCalc}, Section $4.4$ and \cite{diamondShurman}, Theorem $4.5.1$

    The approach of this paper differs significantly from that of \cite{SVY} in the manner in which we use Eisenstein series attached to characters. Specifically, \cite{SVY} chooses to use the completed Eisenstein series attached to characters, then by separating out the holomorphic and anti-holomorphic parts they construct $S_{\chi_1,\chi_2,2}$ from the holomorphic part in a manner analogous to the Kronecker limit formula; however, in the higher weight setting it does not seem that this procedure works. As such, the focus of this paper builds upon a small but critical observation contained within Section $5$ of \cite{SVY}: this generalization of the classical Dedekind sum is exactly a period integral applied to the weight $k=2$ holomorphic Eisenstein series attached to characters. Building upon this, the main goal of this work is to evaluate this period integral applied to weight $k\geq 3$ holomorphic Eisenstein series attached to characters.
    
    To adequately define our higher weight period integrals we provide some notation. Let $V_{k-2}(\mathbb{C})$ be the vector space of degree $k-2$ homogenous polynomials in two variables having complex coefficients. 
    \begin{definition}
        For $z\in\mathcal{H}$ and $X,Y\in\mathbb{C}$, let us write the polynomial
        $$P_{k-2}(z;X,Y)=(Xz+Y)^{k-2}=\left(\begin{pmatrix}
            X & Y
        \end{pmatrix}\begin{pmatrix}
            z \\
            1
        \end{pmatrix}\right)^{k-2}.$$
        
    \end{definition}
    \noindent Note that for fixed $z$, this polynomial $P_{k-2}(z;X,Y)$ is a member of $V_{k-2}(\mathbb{C})$. Now we define our period integral in terms of our Eisenstein series and polynomials.
    \begin{definition}\label{eichlerIntegral}
        Given weight $k\geq 2$, primitive non-trivial Dirichlet characters $\chi_1$ and $\chi_2$ with conductors $q_1$ and $q_2$ respectively such that $\chi_1\chi_2(-1)=(-1)^k$, $\gamma=\begin{psmallmatrix}
            a & b \\
            c & d
        \end{psmallmatrix}\in\Gamma_0(q_1q_2)$, we define $\phi_{\chi_1,\chi_2,k}(\gamma,X,Y)$ as the period integral of $E_{\chi_1,\chi_2,k}$ against the polynomial $P_{k-2}(z;X,Y)$ with base point at $\infty$. That is,
        $$\phi_{\chi_1,\chi_2,k}(\gamma,X,Y)=\int_\infty^{\gamma\infty}E_{\chi_1,\chi_2,k}(z)\,P_{k-2}(z;X,Y)\,dz.$$
    \end{definition}
    \noindent Note that the above integral converges due to the exponential decay of $E_{\chi_1,\chi_2,k}$ at the endpoints (see Lemma \ref{eisensteinSeries}).
\subsection{Finite Sum Formula}
    From here on we implicitly impose the same restrictions on $k$, $\chi_1$, $\chi_2$, $q_1$, $q_2$, $a$, $c$ as in Definition \ref{eichlerIntegral} unless otherwise stated. Usually we can assume that $k\geq 2$, but in Section \ref{finiteSumSec} which contains the proof of Theorem \ref{mainThm}, we assume $k\geq 3$ for much of the section to get the appropriate convergence in our analysis. However, this is not a significant drawback because the $k=2$ case of Theorem \ref{mainThm} is covered in \cite{SVY} Theorem $1.2$, and we incorporate this fact into Theorem \ref{mainThm} below.
    
    \begin{definition}\label{Sdef}
        We define $S_{\chi_1,\chi_2,k}$ as:
        $$S_{\chi_1,\chi_2,k}(\gamma)=(-1)^k\tau(\overline{\chi_1})(k-1)\,\phi_{\chi_1,\chi_2,k}(\gamma,1,-a/c).$$
    \end{definition}
    
    \noindent This paper proves the following finite sum formula which generalizes Theorem $1.2$ of \cite{SVY}.
    \begin{theorem}\label{mainThm}
        For $k\geq 2$, given $\gamma=\begin{psmallmatrix}
            a & b \\
            c & d
        \end{psmallmatrix}\in\Gamma_0(q_1q_2)$ with $c>0$ and primitive non-trivial Dirichlet characters $\chi_1$ and $\chi_2$ with conductors $q_1$ and $q_2$ respectively such that $\chi_1\chi_2(-1)=(-1)^k$, we have
        $$S_{\chi_1,\chi_2,k}(\gamma)=\mathop{\sum\sum}_{\substack{j\bmod c \\ n\bmod q_1}}\overline{\chi_1}(n)\overline{\chi_2}(j)B_1\left(\frac{j}{c}\right)B_{k-1}\left(\frac{aj}{c}+\frac{n}{q_1}\right).$$
    \end{theorem}
    \noindent We will prove Theorem \ref{mainThm} in Section \ref{mainThmProof}.

    Given that Theorem \ref{mainThm} generalizes \cite{SVY} Theorem $1.2$, our higher weight Dedekind sums occupy a similar space in the wider body of literature to the Dedekind sums presented in \cite{SVY}. We give a brief contextualization of our work in the broader literature. As we mentioned earlier the central procedure of \cite{SVY} is applying the Kronecker limit formula to the holomorphic part of the completed Eisenstein series attached to characters (which is a newform); this idea was a natural outgrowth of the work of Goldstein in \cite{Goldstein}. Goldstein essentially did the same thing as \cite{SVY}; however, he applied the Kronecker limit formula to the Eisenstein series attached to a cusp (which is an oldform). Apart from \cite{SVY}, Nagasaka \cite{nagasaka} presents a special case of our Dedekind sums on $\Gamma_0^0(q_1q_2)$ where $\chi=\chi_1=\chi_2$. Additionally, Dağlı and Can used a similar integral-based technique to construct a Dedekind sum in two characters \cite{DC15}; however, their integral was not a period integral and instead comprised a product of Bernoulli polynomials. Perhaps most similar in spirit is the work of \cite{modularSymbolsEisenstein} which implicitly constructs our $\phi_{\chi_1,\chi_2,k}$ function; however, this paper is not focused on the topic of Dedekind sums but rather on modular symbols. Aside from these papers, a large part of the literature on Dedekind sums is based on the work of Berndt \cite{berndt73} which uses a different Eisenstein-type series from the $E_{\chi_1,\chi_2,k}$ defined in Definiton \ref{youngEisensteinDef}. This area of the literature includes works such as \cite{CCK,meyer,sekine}.

    \subsection{Quantum Modularity}
    There is significant precedent for the idea that our higher weight Dedekind sum should behave like a quantum modular form. In the paper \cite{Zagier} where Zagier defines quantum modular forms, the first example he uses is the classical Dedekind sum $s(h,k)$. This point of course comes with the caveat that the classical Dedekind is actually not a quantum modular form, but rather serves as an illustrative example for the structure demonstrated by these quantum modular forms. Additionally, the third example of a quantum modular form that Zagier provides in \cite{Zagier} is a period integral of a cusp form; this is very similar to our generalized Dedekind sums, which are similarly defined as period integrals of Eisenstein series.

    Toward proving that $S_{\chi_1,\chi_2,k}$ is a quantum modular form, we recount the definition of a quantum modular form as provided by Zagier but specialized for our use. 

    \begin{definition}\label{slashDef}
        Given a function $f:\mathcal{H}\to\mathbb{C}$ and a matrix $\gamma\in\text{SL}_2(\mathbb{R})$, we define the weight $\kappa$ slash operator as
        $$f|_\kappa\gamma(z)=j(\gamma,z)^{-\kappa}f(\gamma z).$$
        Additionally, suppose $\mathcal{C}\subseteq\mathbb{P}^1(\mathbb{Q})$ is a set of cusps closed under the action of a congruence subgroup $\Gamma$. Then, given a function $g:\mathcal{C}\to\mathbb{C}$ and a matrix $\gamma\in\Gamma$, we define the weight $\kappa\leq 0$ slash operator as
        $$g|_\kappa\gamma(z)=j(\gamma,z)^{-\kappa}g(\gamma z).$$
    \end{definition}
    \begin{definition}[\cite{Zagier}]
        Let $\Gamma$ be group commensurable with $\text{SL}_2(\mathbb{Z})$ and let $\Gamma(\infty)\subseteq\mathbb{P}^1(\mathbb{Q})$ be the set of cusps which are $\Gamma$-equivalent to $\infty$. We say that a function $f:\Gamma(\infty)\to\mathbb{C}$ is a weight $\kappa\leq 0$ quantum modular form on $\Gamma$ if the function
        $$h_\gamma:\Gamma(\infty)\setminus\{\infty\}\to\mathbb{C},\quad \mathfrak{a}\mapsto f(\mathfrak{a})-f|_\kappa\gamma(\mathfrak{a})$$
        is continuous with respect to the real topology for every $\gamma\in\Gamma$.
    \end{definition}
    \noindent Now we define a function which gives $S_{\chi_1,\chi_2,k}$ as a function on cusps in $(\Gamma_0(q_1q_2))(\infty)$. 
    \begin{definition}\label{shatDef}
        Let $(\Gamma_0(q_1q_2))(\infty)\subseteq\mathbb{P}^1(\mathbb{Q})$ be the set of cusps which are $\Gamma_0(q_1q_2)$-equivalent to $\infty$. We define
        $$\widehat{S}_{\chi_1,\chi_2,k}:(\Gamma_0(q_1q_2))(\infty)\to\mathbb{C},\quad \mathfrak{a}\mapsto(-1)^k\tau(\overline{\chi_1})(k-1)\int_\infty^\mathfrak{a}E_{\chi_1,\chi_2,k}(z)P_{k-2}(z;1,-\mathfrak{a})\,dz$$
        when $\mathfrak{a}\neq\infty$. Otherwise, we let $\widehat{S}_{\chi_1,\chi_2,k}(\infty)=0$.
    \end{definition}
    \begin{remark}
        Note that for all $\gamma\in\Gamma_0(q_1q_2)$ we have that $\widehat{S}_{\chi_1,\chi_2,k}(\gamma\infty)=S_{\chi_1,\chi_2,k}(\gamma).$
    \end{remark}
    
    Now that we have properly defined our objects, we state the following theorem on quantum modularity.
    \begin{theorem}\label{qmfThrm}
        We have that $\widehat{S}_{\chi_1,\chi_2,k}$ is a weight $2-k$ quantum modular form on $\Gamma_1(q_1q_2)$.
    \end{theorem}
    \noindent Theorem \ref{qmfThrm} will be proven in Section \ref{qmfSec}. Quantum modular forms tend to be visually appealing---as such, we have some graphs of $\widehat{S}_{\chi_1,\chi_2,k}$ in Section \ref{vizDiscussion}. Additionally, we have developed tools to produce similar graphs; these tools are available in the Github repository in Section \ref{codeSec}.

    \subsection{Arithmetic Aspects of $S_{\chi_1,\chi_2,k}(\Gamma_1(q_1q_2))$}
    Section \ref{arithmeticProps} details some computations on arithmetic properties of the image $S_{\chi_1,\chi_2,k}(\Gamma_1(q_1q_2))$ which build upon the results of Majure \cite{majorMajure} and an unpublished conjecture of De Leon and McCormick. Firstly, Majure gives us the following beautiful result.
    \begin{lemma}[\cite{majorMajure}, Theorem $1.5$]\label{majImage}
        Let $F_{\chi_1,\chi_2}$ be the smallest number field in which $\chi_1$ and $\chi_2$ take values. Then the image $S_{\chi_1,\chi_2,2}(\Gamma_1(q_1q_2))$ is a full rank lattice in $F_{\chi_1,\chi_2}$. That is,
        $$S_{\chi_1,\chi_2,2}(\Gamma_1(q_1q_2))=\bigoplus_{i=1}^n\alpha_i\,\mathbb{Z}$$
        where $\alpha_i\in F_{\chi_1,\chi_2}$ are non-zero and $n=[F_{\chi_1,\chi_2}:\mathbb{Q}]$.
    \end{lemma}
    \noindent Note that in the case when $\chi_1$  and $\chi_2$ are quadratic, we have that $[F_{\chi_1,\chi_2}:\mathbb{Q}]=1$ and thus $S_{\chi_1,\chi_2,2}(\Gamma_1(q_1q_2))=\alpha\mathbb{Z}$ for some $\alpha\in\mathbb{Q}_{>0}$. Using this theorem of Majure alongside computational evidence, De Leon and McCormick conjectured the following.
    \begin{conjecture}[De Leon, McCormick]\label{wadeConj}
        Given quadratic primitive characters $\chi_1,\chi_2$ with $\chi_1\chi_2(-1)=1$ we have
        $$S_{\chi_1,\chi_2,2}(\Gamma_1(q_1q_2))=2\mathbb{Z}.$$
    \end{conjecture}

    \noindent De Leon and McCormick provided significant computational evidence for this conjecture. Specifically, due to the crossed homomorphism property of (\ref{crossHomInit}), it suffices to compute $S_{\chi_1,\chi_2,2}$ on a generating set of $\Gamma_1(q_1q_2)$ to completely characterize the image $S_{\chi_1,\chi_2,2}(\Gamma_1(q_1q_2))$. Noting that $\Gamma_1(q_1q_2)$ is a finite index subgroup of $\text{SL}_2(\mathbb{Z})$ and that $\text{SL}_2(\mathbb{Z})$ is finitely generated, it follows that $\Gamma_1(q_1q_2)$ is finitely generated as well. Thus we only need to compute $S_{\chi_1,\chi_2,2}$ on a finite generating set to characterize the image $S_{\chi_1,\chi_2,2}(\Gamma_1(q_1q_2))$ (see \cite{majorMajure} Section $3$ for a detailed explanation). Using the method of Tranbarger and Wang in \cite{TW24}, computing $S_{\chi_1,\chi_2,2}$ on a finite generating set becomes trivial. This approach enabled De Leon and McCormick to prove Conjecture \ref{wadeConj} for all pairs of Dirichlet characters with prime $q_1<30$ and prime $q_2<30$.
    
    Naturally, we extend their work by computing a small subset of the image of $S_{\chi_1,\chi_2,k}$ for higher weight to see if it has nice arithmetic properties as well. Indeed, with the proper normalization, we find more elaborate arithmetic properties in the higher weight setting. To state the results of these computations we first define the following auxiliary function.
    % Note that in the higher weight generalization we do not yet have an analogue of Lemma \ref{majImage} which provides a descriptive structure on the image $\widetilde{S}_{\chi_1,\chi_2,k}(\Gamma_1(q_1q_2))$. Despite this, based on the significant amount of computation required to construct the tables in Section \ref{computation}, the statements postulated in Conjecture \ref{conj611} seem reasonable.
    
    \begin{definition}\label{stildeDef}
        For $\chi_1,\,\chi_2$ quadratic characters with $\gamma=\begin{psmallmatrix}
            a & b \\
            c & d
        \end{psmallmatrix}\in\Gamma_0(q_1q_2)$, let $\widetilde{S}_{\chi_1,\chi_2,k}(\gamma)=c^{k-2}\,S_{\chi_1,\chi_2,k}(\gamma).$
    \end{definition}
    \begin{remark}
        It is not obvious that including a $c^{k-2}$ factor is an appropriate normalization for revealing arithmetic properties of $S_{\chi_1,\chi_2,k}$; however, we have a few hints as to why this is this choice is correct. Note that from Theorem \ref{mainThm}, when $\chi_1$ and $\chi_2$ are quadratic we expect that $S_{\chi_1,\chi_2,k}\in\mathbb{Q}$ will have a $c^{-k}$ factor in the denominator; however, in the weight $2$ case, the conjecture of De Leon and McCormick suggests we get $c^2$ cancellation in the numerator. This led us to guess that $c^{k-2}$ was the normalizing factor in the general case, and indeed, computation confirmed this to be the case.
    \end{remark}
       
    In Section \ref{computation} we provide the full results of the following computation. Let $G_j(N)=\{\begin{psmallmatrix}
        a & b \\ c & d
    \end{psmallmatrix}\in\Gamma_1(N):1\leq a<jN,\,N\leq c<jN\}$. For each pair of primitive quadratic characters $(\chi_1,\chi_2)$ satisfying $\chi_1\chi_2(-1)=(-1)^k$ and $q_1q_2\leq 32$ with weight $2\leq k\leq 9$, we compute the largest value of $r\in\mathbb{Q}_{\geq 0}$ such that $\widetilde{S}_{\chi_1,\chi_2,k}(G_{50}(q_1q_2))\subseteq r\mathbb{Z}$. In the table below we provide some selected data from this large computation. Here we define $\chi_3,\chi_4,$ and $\chi_7$ to be the unique primitive quadratic characters modulo $3$, $4$, and $7$ respectively. We have the following values of $r$ for each pair of $(\chi_1,\chi_2)$ and $k$:
    \begin{center}
    \begin{tabular}{c|c c c c c c}
       \backslashbox{$k$}{$(\chi_1,\chi_2)$} & $(\chi_3,\chi_3)$ & $(\chi_3,\chi_4)$ & $(\chi_4,\chi_3)$ & $(\chi_4,\chi_4)$ & $(\chi_3,\chi_7)$ & $(\chi_7,\chi_3)$ \\
       \hline
       $2$ & $2$ & $2$ & $2$ & $2$ & $2$ & $2$ \\
       $4$ & $2$ & $2/3$ & $6/4$ & $6$ & $2/3$ & $6/7$ \\
       $6$ & $10/3$ & $10$ & $10/4$ & $10$ & $10$ & $10/7$ \\
       $8$ & $14$ & $14/3$ & $14/4$ & $14$ & $14/3$ & $2$
    \end{tabular}
    \end{center}
    Note that some values are not expressed in reduced form, and this is a stylistic choice to highlight the pattern outlined in Conjecture \ref{conj611} to follow. Given the large amount of data generated from the computations in Section \ref{computation}, we make the following conjecture of our own, generalizing the work of De Leon and McCormick to the higher weight setting.
    \begin{conjecture}\label{conj611}
        We conjecture the following points regarding the image $\widetilde{S}_{\chi_1,\chi_2,k}(\Gamma_1(q_1q_2))$ for pairs of quadratic characters $\chi_1$ and $\chi_2$.
        \begin{itemize}
            \item The tables in Section \ref{computation} extend from $G_{50}(q_1q_2)$ to $\Gamma_1(q_1q_2)$ with equality.
            \item For some $d$ dividing $2k-2$, either $\widetilde{S}_{\chi_1,\chi_2,k}(\Gamma_1(q_1q_2))=d\mathbb{Z}$ or $\widetilde{S}_{\chi_1,\chi_2,k}(\Gamma_1(q_1q_2))=d\mathbb{Z}/q_1$.
        \end{itemize}
    \end{conjecture}
    \noindent We present some partial progress towards this conjecture in Section \ref{sec623}. Specifically,  we illustrate a sufficient condition for a containment implied by the first point of Conjecture \ref{conj611}.

    \begin{definition}\label{weirdPolynomialSpace}
        Let us define a subset of the polynomials $\mathbb{Q}[x]$
        $$\mathcal{P}(k;m,q)=\left\{\sum_{n=0}^{k-2}a_nx^{k-n-2}\in\mathbb{Q}[x]:q^{n+1}\,a_n\in m\mathbb{Z}\right\}.$$
    \end{definition}
    \begin{theorem}\label{containmentTheorem}
        Let $G$ be a finite generating set of $\Gamma_1(q_1q_2)$. If $h_{\gamma,\chi_1,\chi_2,k}\in\mathcal{P}(k;m,q_1)$ for all $\gamma\in G$ then we have that $\widetilde{S}_{\chi_1,\chi_2,k}(\Gamma_1(q_1q_2))\subseteq m\mathbb{Z}/q_1$.
    \end{theorem}
    \noindent Later in Section \ref{sec623}, this theorem (with the prerequisite computation) will allow us to prove that this containment implied by the first point of Conjecture \ref{conj611} holds. The procedure for this is roughly as follows. First we compute a finite generating set $G$ of $\Gamma_1(q_1q_2)$. Then for every $\gamma\in G$ we compute $h_{\gamma,\chi_1,\chi_2,k}$ over various $k$ and $(\chi_1,\chi_2)$ via Lagrange interpolation and tabulate the results. Then for fixed $k$ and $(\chi_1,\chi_2)$ verify there exists $m$ such that $h_{\gamma,\chi_1,\chi_2,k}\in\mathcal{P}(k;m,q_1)$ for all $\gamma\in G$. Theorem \ref{containmentTheorem} then shows that $\widetilde{S}_{\chi_1,\chi_2,k}(\Gamma_1(q_1q_2))\subseteq m\mathbb{Z}/q_1$. In Section \ref{codeSec} at the end of this paper, one can find the computational evidence that this procedure has been applied for every relevant combination of $k$ and $(\chi_1,\chi_2)$ exhibited in the tables of Section \ref{computation}. We provide an extension of this discussion with an explicit example towards the end of Section \ref{sec623}.

    \subsection{Fricke Reciprocity}
    Some progress on the conjecture of De Leon and McCormick has recently been made by \cite{KMS(y)}. Specifically, the authors show that when $\chi_1$ and $\chi_2$ are primitive quadratic characters with $\chi_1\chi_2(-1)=1$ one has
    $$S_{\chi_1,\chi_2,2}(\gamma)\in\left(\frac{1}{(q_1,q_2)}\right)\,  \mathbb{Z}$$
    for all $\gamma\in\Gamma_1(q_1q_2)$. One of the key results which allowed them to make this advancement was the use of Fricke reciprocity. The Fricke involution is defined as follows
    \begin{definition}\label{frickeDef}
        We define the Fricke involution $\omega_N\in\mathrm{SL}_2(\mathbb{R})$ as
        $$\omega_N=\begin{pmatrix}
            0 & -1/\sqrt{N} \\
            \sqrt{N} & 0
        \end{pmatrix}.$$
    \end{definition}
    \noindent Note that for the rest of the paper, we will write $\omega$ instead of $\omega_{q_1q_2}$ since $q_1q_2$ is implicitly described by $\chi_1$ and $\chi_2$. Stucker, Vennos, and Young showed how the Fricke involution interacts with $S_{\chi_1,\chi_2,2}$. Specifically, for $\gamma=\begin{psmallmatrix}
        a & b \\ cq_1q_2 & d
    \end{psmallmatrix}\in\Gamma_0(q_1q_2)$ and $\gamma'=\begin{psmallmatrix}
        d & -c \\ -bq_1q_2 & a
    \end{psmallmatrix}\in\Gamma_0(q_1q_2)$, per \cite{SVY} Theorem $1.3$ in combination with the functional equation of $L$-functions, one has
    \begin{equation}\label{equFrickeRep}S_{\chi_1,\chi_2,2}(\gamma)=\chi_1(-1)\,S_{\chi_2,\chi_1,2}(\gamma')+(1-\psi(\gamma))\left(\frac{\tau(\overline{\chi_1})}{\pi i}\right)L(1,\chi_1)\,L(0,\overline{\chi_2}).\end{equation}
    Analogously, in an effort to enable further research on Conjecture \ref{conj611}, we develop some reciprocity formulae involving the Fricke involution. We extend $\widehat{S}_{\chi_1,\chi_2,k}$, $h_{\gamma,\chi_1,\chi_2,k}$, and $\phi_{\chi_1,\chi_2,k}$ from $(\Gamma_0(q_1q_2))(\infty)$ to $(\Gamma_0(q_1q_2))(\infty)\cup(\omega\Gamma_0(q_1q_2))(\infty)$ in the natural way. We detail this process in the beginning of Section \ref{frickeSec}, ultimately resulting in a reciprocity formula on $\widehat{S}_{\chi_1,\chi_2,k}$.
    \begin{theorem}\label{reciprocity}
        For $\gamma=\begin{psmallmatrix}
            a & b \\ cq_1q_2 & d
        \end{psmallmatrix}\in\Gamma_0(q_1q_2)$ and $\gamma'=\begin{psmallmatrix}
            d & -c \\ -bq_1q_2 & a
        \end{psmallmatrix}\in\Gamma_0(q_1q_2)$ and $\mathfrak{a}\in((\Gamma_0(q_1q_2))(\infty)\cup(\omega\Gamma_0(q_1q_2))(\infty))\setminus\{\infty\}$ one has
        \begin{align*}
        &\widehat{S}_{\chi_1,\chi_2,k}|_{2-k}\gamma'(\mathfrak{a})-\psi(\gamma')\,\widehat{S}_{\chi_1,\chi_2,k}(\mathfrak{a}) \\
        &\qquad+(-1)^k\tau(\overline{\chi_1})(k-1)\left(\psi(\gamma')\,\phi_{\chi_1,\chi_2,k}(\gamma'^{-1},1,-\mathfrak{a})-\psi(\gamma)\,\phi_{\chi_1,\chi_2,k}|_{2-k}\omega(\gamma^{-1},1,-\mathfrak{a})\right) \\
        &=R_{\chi_1,\chi_2,k}\cdot(\tau(\overline{\chi_1})/\tau(\overline{\chi_2}))\left(\widehat{S}_{\chi_2,\chi_1,k}|_{2-k}\gamma'(\mathfrak{a})-\widehat{S}_{\chi_2,\chi_1,k}(\mathfrak{a})\right) \\
        &\qquad+(-1)^k\tau(\overline{\chi_1})(k-1)\,R_{\chi_1,\chi_2,k}\cdot\left(\phi_{\chi_2,\chi_1,k}(\omega^{-1},1,-\mathfrak{a})-\phi_{\chi_2,\chi_1,k}|_{2-k}\gamma'(\omega^{-1},1,-\mathfrak{a})\right) \\
        &\qquad+(1-\psi(\gamma))\,\widehat{S}_{\chi_1,\chi_2,k}|_{2-k}\omega(\mathfrak{a}).
    \end{align*}
    where $R_{\chi_1,\chi_2,k}=\chi_1(-1)(\tau(\chi_1)/\tau(\chi_2))(q_2/q_1)^{k/2}$.
    \end{theorem}
    \noindent In the weight $k=2$ case, since $f|_{2-k}\gamma(\mathfrak{a})=f(\gamma\mathfrak{a})$, we can extend $h_{\gamma,\chi_1,\chi_2,k}$ and $h_{\omega\gamma,\chi_1,\chi_2,k}$ to all of $(\Gamma_0(q_1q_2))(\infty)\cup(\omega\Gamma_0(q_1q_2))(\infty)$. By extension we find the above formula holds on all of $(\Gamma_0(q_1q_2))(\infty)\cup(\omega\Gamma_0(q_1q_2))(\infty)$ when $k=2$. So, in looking at the behavior at $\mathfrak{a}=\infty$ when $k=2$, we recover the reciprocity formula on $S_{\chi_1,\chi_2,2}(\gamma)$ and $S_{\chi_2,\chi_1,2}(\gamma')$ as in (\ref{equFrickeRep}). The details of this procedure are in the proof of Corollary \ref{SVYFrickeRecovery}.  

    \subsection{Analytic Aspects of $S_{\chi_1,\chi_2,k}(\Gamma_0(q_1q_2))$}

    Although this work serves as a generalization of \cite{SVY}, the use of period integrals as the definition of these generalized Dedekind sums represents a departure from the approach in \cite{SVY} as we alluded to earlier. Indeed, \cite{SVY} implicitly define $S_{\chi_1,\chi_2,2}$ as the sum
    \begin{equation}\label{basicIndepPointFormula}
        S_{\chi_1,\chi_2,2}(\gamma)=\frac{\tau(\overline{\chi_1})}{\pi i}\left(\sum_{1\leq A}\frac{\chi_1(A)}{A}\sum_{1\leq B}\overline{\chi_2}(B)\,e(AB\gamma z_1)-\psi(\gamma)\sum_{1\leq A}\frac{\chi_1(A)}{A}\sum_{1\leq B}\overline{\chi_2}(B)\,e(ABz_1)\right)
    \end{equation}
    which turns out to be independent of the choice of $z_1\in\mathcal{H}$ per Lemma 2.1 of \cite{SVY}. We recover this result from the new framework of period integrals yielding the more general independent point formula of Lemma \ref{cobformula} in Section \ref{cobproof}. Following a process similar to that outlined in \cite{cy24}, this lemma proves useful in that we can leverage this independent point formula with a specific choice of $z_1\in\mathcal{H}$ to derive the non-trivial bound on the magnitude of our generalized Dedekind sums as in (\ref{nonTrivialBoundEqu}).

    Before we state this non-trivial upper bound on $\abs{S_{\chi_1,\chi_2,k}}$ we provide a trivial bound for comparison. First, note that Lehmer gave an upper bound on the size of periodic Bernoulli polynomials in his paper \cite{lehmer}.
    $$\abs{B_k(x)}\leq\begin{cases}
        1/2 & k=1 \\
        2k!\,\zeta(k)/(2\pi)^k & k\equiv 2\bmod 4 \\
        2k!/(2\pi)^k & \text{otherwise.}
    \end{cases}$$
    We can unify these bounds with $\abs{B_k(x)}\leq (\pi^2/3)(k!/(2\pi)^k)$ for $k\geq 1$. Thus using Theorem \ref{mainThm}, by the triangle inequality and the above bound, we immediately have that
    $$\abs{S_{\chi_1,\chi_2,k}(\gamma)}\leq\sum_{\substack{j\bmod c \\ n\bmod q_1}}\abs{B_1\left(\frac{j}{c}\right)}\cdot\abs{B_{k-1}\left(\frac{aj}{c}+\frac{n}{q_1}\right)}\leq cq_1\cdot\frac{1}{2}\cdot\frac{\pi^2}{3}\left(\frac{(k-1)!}{(2\pi)^{k-1}}\right)\leq\frac{cq_1\pi^2}{6}\cdot\frac{(k-1)!}{(2\pi)^{k-1}}.$$
    Note the $O(c)$ growth of this upper bound; our new independent point formula will allow us to derive the following bound, which is generally better than this trivial bound.
     \begin{definition}\label{partialQuoDef}
        For $q\in\mathbb{Q}$, the continued fraction expansion of $q$ can be written in two ways
        $$q=[a_0;a_1,\ldots,a_n]=[a_0;a_1,\ldots a_n-1,1]$$
        where $a_0,\ldots,a_n$ depend on $q$. We let $M(q)=\max\{a_1,\ldots,a_n\}$.
    \end{definition}
    \begin{theorem}\label{nontrivialbound}
        Let $c'=c/q_2$ with $c\geq 1$. We have that
        \begin{equation}\label{nonTrivialBoundEqu}
            \abs{S_{\chi_1,\chi_2,k}(\gamma)}\ll_{k,\,q_1,\,q_2}M(a/c')\log^2c'
        \end{equation}
        where $M(q)$ is defined in Definition \ref{partialQuoDef} and $\mathcal{C}$ is a constant solely depending on $q_1$, $q_2$, and $k$. Furthermore, let
        $$L_{\chi_1,\chi_2,k}(\alpha,C)=\abs{\{(a,c):1\leq a<c\leq C,\,(a,c)=1,\,q_1q_2\mid c,\,\abs{S_{\chi_1,\chi_2,k}(\gamma)}>\alpha\log^3C\}}.$$
        We then have a bound on the number of exceptionally large values of $S_{\chi_1,\chi_2,k}$ as
        $$L_{\chi_1,\chi_2,k}(\alpha,C)\ll_{\chi_1,\chi_2,k}\frac{C^2}{\alpha}+C^2\,\frac{\log\log C}{\log C}.$$
    \end{theorem}
    \noindent We will prove Theorem \ref{nontrivialbound} in Section \ref{nontrivialboundproof}. Note that both the proof of Lemma \ref{cobformula} and Theorem \ref{nontrivialbound} roughly follow the process outlined in \cite{cy24}, but in this case we have adapted for the higher weight setting.

\section{$S_{\chi_1,\chi_2,k}$ as a Finite Sum}\label{finiteSumSec}
    % Before we provide our necessary analysis, we first collect some results from the literature which we will use in the proof of Theorem \ref{mainThm}.
\subsection{Results from the Literature}
    Using the work of Berndt we define character analogues of Bernoulli polynomials.
    \begin{definition}[\cite{berndt}, Definition $1$]\label{berndt1}
        Given a primitive Dirichlet character $\chi$ with modulus $m$, for integer $k\geq 2$ we define $B_{k-1,\chi}(x)$ using the expression
        $$B_{k-1,\chi}(x)=\frac{(-i)^k\tau(\overline{\chi})(k-1)!}{im(2\pi /m)^{k-1}}\sum_{n\neq 0}\frac{\chi(n)}{n^{k-1}}e_m(nx).$$
    \end{definition}
    \noindent Furthermore, we have an expression relating Bernoulli polynomials to character analogues of Bernoulli polynomials.
    \begin{lemma}[\cite{berndt}, Theorem $3.1$]\label{berndt31}
        Given a primitive Dirichlet character $\chi$ with modulus $m$, for integer $k\geq 1$ we have the following expression
        $$B_{k,\chi}(x)=m^{k-1}\sum_{n\bmod m}\overline{\chi}(n)B_k\left(\frac{x+n}{m}\right).$$
    \end{lemma}
    Berndt also provides us with the following character analogue of Poisson summation.
    \begin{theorem}[\cite{berndt}, Theorem $2.3$]\label{berndt23}
        Given a primitive Dirichlet character $\chi$ with conductor $q$ and a function $f$ of bounded variation on $[0,\infty)$ we have that
        $$\sum_{1\leq n}\chi(n)f(n)=\frac{2\tau(\chi)}{q}\sum_{v\in\mathbb{Z}}\overline{\chi}(v)\int_0^\infty f(t)\,e^{2\pi ivt/q}\,dt.$$
    \end{theorem}
    
    Finally, we use the work of Stucker, Vennos, and Young to help simplify a particular sum in Section \ref{analysisSection}.
    \begin{lemma}[\cite{SVY}, Lemma $3.2$]\label{svy32}
        We have that
        $$\sum_{1\leq B}\overline{\chi_2}(B)e^{2\pi AB(ia/c-u)}=\sum_{j\bmod c}\overline{\chi_2}(j)e_c(Aaj)\left(\frac{e(Aiuj)-1}{1-e(Aiuc)}\right)\quad\text{where}\quad e_c(z)=e(z/c).$$
    \end{lemma}
    \begin{lemma}[\cite{SVY}, Lemma $3.3$]\label{svy33}
        We have that
        $$\lim_{u\to 0^+}\sum_{1\leq B}\overline{\chi_2}(B)e^{2\pi AB(ia/c-u)}=-\sum_{j\bmod c}\overline{\chi_2}(j)B_1\left(\frac{j}{c}\right)e_c(Aaj).$$
    \end{lemma}
        
\subsection{Analysis Preliminaries}\label{analysisSection}
    First we develop a twisted Poisson summation identity which will give us the proper convergence in future lemmas.
    \begin{lemma}\label{poissonSummation}
        For an integer $K\geq 1$ and $z\in\mathcal{H}$ we have
        $$\sum_{1\leq B}\overline{\chi_2}(B)B^Ke^{2\pi ABiz}=2q_2^K\left(\frac{\tau(\overline{\chi_2})\,K!}{(-2\pi i)^{K+1}}\right)\sum_{v\in\mathbb{Z}}\frac{\chi_2(v)}{(Aq_2z+v)^{K+1}}.$$
    \end{lemma}
    \begin{proof}
        Applying the formula given in Theorem \ref{berndt23} gives us
        $$\sum_{1\leq B}\overline{\chi_2}(B)B^Ke^{2\pi ABiz}=\frac{2\tau(\overline{\chi_2})}{q_2}\sum_{v\in\mathbb{Z}}\chi_2(v)\int_0^\infty t^Ke^{(2\pi it/q_2)(Aq_2z+v)}\,dt.$$
        Evaluating this integral completes the proof (see \cite{GRTables} 8.312.2)
    \end{proof}
    We would like to show that the sum on the right hand side of Lemma \ref{poissonSummation} converges absolutely.
    \begin{lemma}\label{hurwitzTypeSum}
        For a real number $\sigma>0$ and $(a,Q)=1$, when $Q\nmid A$ we have that
        $$\sum_{v\in\mathbb{Z}}\frac{1}{\abs{Aa/Q+v}^{\sigma+1}}\ll_{\sigma,Q}1.$$
    \end{lemma}
    \begin{proof}
        Note that when $Q\nmid A$ then
        \begin{align*}
            \sum_{v\in\mathbb{Z}}\frac{1}{\abs{Aa/Q+v}^{\sigma+1}}&=\sum_{v=\lceil Aa/Q\rceil}^\infty\frac{1}{(v-Aa/Q)^{\sigma+1}}+\sum_{v=-\lfloor Aa/Q\rfloor}^\infty\frac{1}{(Aa/Q+v)^{\sigma+1}} \\
            &=\zeta(\sigma+1,\lceil Aa/Q\rceil-Aa/Q)+\zeta(\sigma+1,Aa/Q-\lfloor Aa/Q\rfloor)
        \end{align*}
        where $\zeta(s,a)=\sum_{0\leq n}(n+a)^{-s}$ is the Hurwitz zeta function which converges absolutely for complex $\mathfrak{Re}(s)>1$ and real $0<a<1$. Noting $\zeta(\sigma+1,\lceil Aa/Q\rceil-Aa/Q)\leq\zeta(\sigma+1,1/Q)$ and $\zeta(\sigma+1,Aa/Q-\lfloor Aa/Q\rfloor)\leq\zeta(\sigma+1,1/Q)$ gives us the desired upper bound on our sum.
    \end{proof}
    We now provide a uniform bound on a particular summation.
    \begin{lemma}\label{outerSumLimit}
        For integers $k\geq 3$ and $0\leq n<k-2$, we have the following bound
        $$\abs{\sum_{1\leq A}\frac{\chi_1(A)}{A^{n+1}}\sum_{1\leq B}\overline{\chi_2}(B)B^{k-n-2}e^{2\pi AB(ia/c-u)}}\ll_{k,\,n,\,c,\,q_2}\frac{1}{\sqrt{u}}.$$
    \end{lemma}
    \begin{proof}
        We use Lemma \ref{poissonSummation} to give an alternate formula for the inner sum over $B$; thus we have
        \begin{align*}
            &\abs{\sum_{1\leq A}\frac{\chi_1(A)}{A^{n+1}}\sum_{1\leq B}\overline{\chi_2}(B)B^{k-n-2}e^{2\pi AB(ia/c-u)}} \\
            &\qquad\qquad\leq 2q_2^{k-n-2}\left(\frac{\sqrt{q_2}\,(k-n-2)!}{(2\pi)^{k-n-1}}\right)\sum_{\substack{1\leq A \\ (A,q_1)=1}}\abs{\frac{\chi_1(A)}{A^{n+1}}\sum_{v\in\mathbb{Z}}\frac{\chi_2(v)}{(Aq_2(a/c+iu)+v)^{k-n-1}}}.
        \end{align*}
        % Now we will bound the summand. Note that
        % $$\abs{\frac{\chi_1(A)}{A^{n+1}}\sum_{v\in\mathbb{Z}}\frac{\chi_2(v)}{(Aq_2(a/c+iu)+v)^{k-n-1}}}\leq\frac{1}{A^{n+1}}\sum_{\substack{v\in\mathbb{Z} \\ (v,q_2)=1}}\frac{1}{\abs{Aq_2(a/c+iu)+v}^{k-n-1}}.$$
        Examining the above, note it is sufficient to prove that
        $$\abs{\sum_{v\in\mathbb{Z}}\frac{\chi_2(v)}{(Aq_2(a/c+iu)+v)^{k-n-1}}}\ll_{k,n,c,q_2}\frac{1}{\sqrt{Aq_2u}}$$
        for all $(A,q_1)=1$, because this gives us an upper bound of $\zeta(n+3/2)$ on the sum over $A$. So to prove this, first note the following inequality where we borrow $\sqrt{Aq_2u}$ factor from the denominator.
        $$\frac{1}{\abs{Aq_2(a/c+iu)+v}^{k-n-1}}\leq \frac{1}{(Aq_2u)^{1/2}\abs{Aa/(cq_2^{-1})+v}^{k-n-3/2}}.$$
        Now since $(A,q_1)=1$, we know $q_1\nmid A$; so by Lemma \ref{hurwitzTypeSum} we have
        $$\sum_{\substack{v\in\mathbb{Z} \\ (v,q_2)=1}}\frac{1}{\abs{Aa/(cq_2^{-1})+v}^{k-n-3/2}}\leq\sum_{v\in\mathbb{Z}}\frac{1}{\abs{Aa/(cq_2^{-1})+v}^{k-n-3/2}}\ll_{k,n,c,q_2}1.$$
        % Thus, using the triangle inequality in combination with the two above, we have
        % This bounds our summand; hence
        % $$\abs{\sum_{1\leq A}\frac{\chi_1(A)}{A^{n+1}}\sum_{1\leq B}\overline{\chi_2}(B)B^{k-n-2}e^{2\pi AB(ia/c-u)}}\ll_{k,n,c,q_2}\frac{2q_2^{k-n-2}}{\sqrt{u}}\left(\frac{(k-n-2)!}{(2\pi)^{k-n-1}}\right)\sum_{\substack{1\leq A \\ (A,q_1)=1}}\frac{1}{A^{n+3/2}}.$$
        % Noting that this sum converges completes our proof.
        Thus by the triangle inequality, applying the two inequalities above, we have that
        \begin{align*}
            \abs{\sum_{v\in\mathbb{Z}}\frac{\chi_2(v)}{(Aq_2(a/c+iu)+v)^{k-n-1}}}&\leq\sum_{\substack{v\in\mathbb{Z} \\ (v,q_2)=1}}\frac{1}{\abs{Aq_2(a/c+iu)+v}^{k-n-1}} \\
            &\qquad\leq\frac{1}{\sqrt{Aq_2u}}\sum_{\substack{v\in\mathbb{Z} \\ (v,q_2)=1}}\frac{1}{\abs{Aa/(cq_2^{-1})+v}^{k-n-3/2}}\ll_{k,n,c,q_2}\frac{1}{\sqrt{Aq_2u}}
        \end{align*}
        completing our proof by satisfying the sufficient condition.
    \end{proof}
    \noindent From Lemma \ref{outerSumLimit} it immediately follows that the following limit vanishes.
    \begin{corollary}\label{outerSumLimitCor}
        For integers $k\geq 3$ and $0\leq n<k-2$ we have
        $$\lim_{u\to 0^+}u^{k-n-2}\sum_{1\leq A}\frac{\chi_1(A)}{A^{n+1}}\sum_{1\leq B}\overline{\chi_2}(B)B^{k-n-2}e^{2\pi AB(ia/c-u)}=0.$$
    \end{corollary}
    And lastly we explicitly evaluate a particular limit.
    \begin{lemma}\label{limitInterchange}
        For $k\geq 3$, we have the limit
        $$\lim_{u\to 0^+}\sum_{1\leq A}\frac{\chi_1(A)}{A^{k-1}}\sum_{1\leq B}\overline{\chi_2}(B)e^{2\pi AB(ia/c-u)}=-\sum_{1\leq A}\sum_{j\bmod c}\frac{\chi_1(A)\overline{\chi_2}(j)}{A^{k-1}}B_1\left(\frac{j}{c}\right)e_c(Aaj).$$
    \end{lemma}
    \begin{proof}
        First we seek to interchange this limit and the sum over $A$ via dominated convergence, so we will bound the modulus of the summand independently of $u$. Applying Lemma \ref{svy32} to the inner sum and noting, by standard calculus exercise that $\abs{(e(Aiuj)-1)/(1-e(Aiuc))}\leq 1$, gives us the bound
        $$\abs{\frac{\chi_1(A)}{A^{k-1}}\sum_{1\leq B}\overline{\chi_2}(B)e^{2\pi AB(ia/c-u)}}\leq\frac{1}{A^{k-1}}\sum_{j\bmod c}\abs{\overline{\chi_2}(j)e_c(Aaj)\left(\frac{e(Aiuj)-1}{1-e(Aiuc)}\right)}\leq\frac{c}{A^{k-1}}.$$
        Thus we can interchange limits via dominated convergence. So it follows that
        $$\lim_{u\to 0^+}\sum_{1\leq A}\frac{\chi_1(A)}{A^{k-1}}\sum_{1\leq B}\overline{\chi_2}(B)e^{2\pi AB(ia/c-u)}=\sum_{1\leq A}\frac{\chi_1(A)}{A^{k-1}}\lim_{u\to 0^+}\sum_{1\leq B}\overline{\chi_2}(B)e^{2\pi AB(ia/c-u)}.$$
        Applying Lemma \ref{svy33} to the limit on the right hand side completes the proof.
    \end{proof}
    
\subsection{Proof of Theorem \ref{mainThm}}\label{mainThmProof}
    We first prove a lemma on integrals similar to those of the form in $\phi_{\chi_1,\chi_2,k}$.
    \begin{lemma}\label{generalIntegralFormula}
        For $s,s'\in\mathcal{H}$ we have that
        $$\int_{s}^{s'}E_{\chi_1,\chi_2,k}(z)\,P_{k-2}(z;X,Y)\,dz=2\sum_{1\leq N}\sum_{A\mid N}\sum_{n=0}^{k-2}\left(\frac{\chi_1(A)\overline{\chi_2}(N/A)N^{k-n-2}}{-A^{k-1}(-2\pi i)^{n+1}}\left(\frac{d^n}{dz^n}(Xz+Y)^{k-2}\right)e(Nz)\right)\Bigg|_{s}^{s'}.$$
    \end{lemma}
    \begin{proof}
        We begin by substituting Lemma \ref{eisensteinSeries} into Definition \ref{eichlerIntegral}; doing so gives that
        $$\int_{s}^{s'}E_{\chi_1,\chi_2,k}(z)\,P_{k-2}(z;X,Y)\,dz=2\int_s^{s'}\left((Xz+Y)^{k-2}\sum_{1\leq N}\sum_{A\mid N}\chi_1(A)\overline{\chi_2}(N/A)(N/A)^{k-1}e(Nz)\right)\,dz.$$
        We may interchange the sums and integral due to the rapid decay of $e(Nz)$ when $z\in\mathcal{H}$; thus
        $$\int_{s}^{s'}E_{\chi_1,\chi_2,k}(z)\,P_{k-2}(z;X,Y)\,dz=2\sum_{1\leq N}\sum_{A\mid N}\int_s^{s'}(Xz+Y)^{k-2}\chi_1(A)\overline{\chi_2}(N/A)(N/A)^{k-1}\,e(Nz)\,dz.$$
        Using repeated integration by parts (differentiating $(Xz+Y)^{k-2}$ and integrating $e(Nz)$), we have
        $$\int_{s}^{s'}E_{\chi_1,\chi_2,k}(z)\,P_{k-2}(z;X,Y)\,dz=2\sum_{1\leq N}\sum_{A\mid N}\sum_{n=0}^{k-2}\left(\frac{\chi_1(A)\overline{\chi_2}(N/A)N^{k-n-2}}{-A^{k-1}(-2\pi i)^{n+1}}\left(\frac{d^n}{dz^n}(Xz+Y)^{k-2}\right)e(Nz)\right)\Bigg|_s^{s'}$$
        exactly as desired.
    \end{proof}
    
    \begin{remark}
        For the purposes of computing $\phi_{\chi_1,\chi_2,k}(\gamma,X,Y)$ we let $z_0=-d/c+i/(c^2u)$ such that $\gamma z_0=a/c+iu$; then
        $$\phi_{\chi_1,\chi_2,k}(\gamma,X,Y)=\lim_{u\to 0^+}\int_{z_0}^{\gamma z_0}E_{\chi_1,\chi_2,k}(z)P_{k-2}(z;X,Y)\,dz.$$
    \end{remark}
    \begin{proof}
    Let $k\geq 3$. By Lemma \ref{generalIntegralFormula} we have
    $$\phi_{\chi_1,\chi_2,k}(\gamma,X,Y)=2\lim_{u\to 0^+}\sum_{1\leq N}\sum_{A\mid N}\sum_{n=0}^{k-2}\left(\frac{\chi_1(A)\overline{\chi_2}(N/A)N^{k-n-2}}{-A^{k-1}(-2\pi i)^{n+1}}\left(\frac{d^n}{dz^n}(Xz+Y)^{k-2}\right)e(Nz)\right)\Bigg|_{z_0}^{\gamma z_0}.$$
    Note that as $u\to 0^+$, we have polynomial growth of $(d^n/dz^n)(Xz_0+Y)^{k-2}$ and exponential decay of $e(Nz_0)$. So every term in this sum vanishes in the limit $u\to 0^+$ at $z_0$, thus we are left with
    $$\phi_{\chi_1,\chi_2,k}(\gamma,X,Y)=2\lim_{u\to 0^+}\sum_{1\leq N}\sum_{A\mid N}\sum_{n=0}^{k-2}\left(\frac{\chi_1(A)\overline{\chi_2}(N/A)N^{k-n-2}}{-A^{k-1}(-2\pi i)^{n+1}}\right)\left(\frac{X^n(k-2)!}{(k-n-2)!}(X\gamma z_0+Y)^{k-n-2}\,e(N\gamma z_0)\right).$$
    Interchanging the summations we write $\phi_{\chi_1,\chi_2,k}(\gamma,X,Y)$ in the form
    $$=\sum_{n=0}^{k-2}\left(-\frac{2X^n(k-2)!}{(-2\pi i)^{n+1}(k-n-2)!}\right)\left(\lim_{u\to 0^+}(X\gamma z_0+Y)^{k-n-2}\sum_{1\leq A}\frac{\chi_1(A)}{A^{n+1}}\sum_{1\leq B}\overline{\chi_2}(B)B^{k-n-2}e(AB\gamma z_0)\right).$$
    Recall from Definition \ref{Sdef} that $S_{\chi_1,\chi_2,k}$ requires the specialization $X=1$ and $Y=-a/c$; doing so gives the formula
    $$\phi_{\chi_1,\chi_2,k}(\gamma,1,-a/c)=\sum_{n=0}^{k-2}\left(-\frac{2i^{k-n-2}(k-2)!}{(-2\pi i)^{n+1}(k-n-2)!}\right)\left(\lim_{u\to 0^+}u^{k-n-2}\sum_{1\leq A}\frac{\chi_1(A)}{A^{n+1}}\sum_{1\leq B}\overline{\chi_2}(B)B^{k-n-2}e^{2\pi AB(ia/c-u)}\right).$$
    Corollary \ref{outerSumLimitCor} implies that when $k\geq 3$ and $0\leq n<k-2$, this limit vanishes. Thus,
    $$\phi_{\chi_1,\chi_2,k}(\gamma,1,-a/c)=\frac{(k-2)!}{\pi i}\left(-\frac{1}{2\pi i}\right)^{k-2}\lim_{u\to 0^+}\sum_{1\leq A}\frac{\chi_1(A)}{A^{k-1}}\sum_{1\leq B}\overline{\chi_2}(B)e^{2\pi AB(ia/c-u)}.$$
    Applying Lemma \ref{limitInterchange}, we have that
    % $$\frac{(-1)^kS_{\chi_1,\chi_2,k}(\gamma)}{\tau(\overline{\chi_1})(k-1)}=\frac{(k-2)!}{\pi i}\left(-\frac{1}{2\pi i}\right)^{k-2}\sum_{1\leq A}\frac{\chi_1(A)}{A^{k-1}}\lim_{u\to 0^+}\sum_{1\leq B}\overline{\chi_2}(B)e^{2\pi AB(ia/c-u)}.$$
    % \cite{SVY} Corollary $3.3$ states that
    % $$\lim_{u\to 0^+}\sum_{1\leq B}\overline{\chi_2}(B)e^{2\pi AB(ia/c-u)}=-\sum_{0\leq j<c}\overline{\chi_2}(j)B_1\left(\frac{j}{c}\right)e_c(Aaj).$$
    % Thus we have that
    \begin{equation}\label{pickupPoint}
        \phi_{\chi_1,\chi_2,k}(\gamma,1,-a/c)=-\frac{(k-2)!}{\pi i}\left(-\frac{1}{2\pi i}\right)^{k-2}\sum_{1\leq A}\,\sum_{j\bmod c}\left(\frac{\chi_1(A)\overline{\chi_2}(j)}{A^{k-1}}\right)B_1\left(\frac{j}{c}\right)e_c(Aaj).
    \end{equation}
    Following the proof of Theorem $1.2$ in \cite{SVY}, we would like to use some type of symmetry relation to reindex our sum with $A$ running over non-zero integers. Specifically, $(-1)^k\,\phi_{\chi_1,\chi_2,k}(\gamma,1,-a/c)=\chi_2(-1)\,\overline{\phi}_{\overline{\chi_1},\overline{\chi_2},k}(\gamma,1,-a/c)$. To do this, we pause the proof of Theorem \ref{mainThm} and
    % we pause the proof of Theorem \ref{mainThm} to show that
    % $$S_{\chi_1,\chi_2,k}(\gamma)=\frac{1}{2}\left(S_{\chi_1,\chi_2,k}(\gamma)+\chi_1(-1)\overline{S}_{\overline{\chi_1},\overline{\chi_2},k}(\gamma)\right)$$
    note that
    $$\chi_2(-1)\,\overline{\phi}_{\overline{\chi_1},\overline{\chi_2},k}(\gamma,1,-a/c)=\chi_2(-1)\left(\frac{(k-2)!}{\pi i}\right)\left(\frac{1}{2\pi i}\right)^{k-2}\sum_{1\leq A}\,\sum_{j\bmod c}\left(\frac{\chi_1(A)\overline{\chi_2}(j)}{A^{k-1}}\right)B_1\left(\frac{j}{c}\right)\overline{e_c}(Aaj).$$
    Observe that $\chi_2(-1)=\overline{\chi_2}(-1)$, $B_1(-x)=-B_1(x)$, and $\overline{e_c}(Aaj)=e_c(-Aaj)$. So,
    $$\chi_2(-1)\,\overline{\phi}_{\overline{\chi_1},\overline{\chi_2},k}(\gamma,1,-a/c)=-\frac{(k-2)!}{\pi i}\left(\frac{1}{2\pi i}\right)^{k-2}\sum_{1\leq A}\,\sum_{j\bmod c}\left(\frac{\chi_1(A)\overline{\chi_2}(-j)}{A^{k-1}}\right)B_1\left(-\frac{j}{c}\right)e_c(-Aaj).$$
    By sending $j\mapsto -j$, we re-index our summation
    $$\chi_2(-1)\,\overline{\phi}_{\overline{\chi_1},\overline{\chi_2},k}(\gamma,1,-a/c)=-\frac{(k-2)!}{\pi i}\left(\frac{1}{2\pi i}\right)^{k-2}\sum_{1\leq A}\,\sum_{j\bmod c}\left(\frac{\chi_1(A)\overline{\chi_2}(j)}{A^{k-1}}\right)B_1\left(\frac{j}{c}\right)e_c(Aaj).$$
    % Now note that because $\gamma\in\Gamma_0(q_1q_2)$ it follows that $q_2\mid c$. Observe the periodicity results $\overline{\chi_2}(j)=\overline{\chi_2}(j+c)$, $B_1(j/c)=B_1((j+c)/c)$, and $e_c(Aaj)=e_c(Aa(j+c))$. By sending $j\mapsto j+c$ and applying these periodicity results (expect when $j=0$), we re-index our summation
    % $$\chi_2(-1)\left(\frac{(-1)^k\overline{S}_{\overline{\chi_1},\overline{\chi_2},k}(\gamma)}{\tau(\overline{\chi_1})(k-1)}\right)=-\frac{(k-2)!}{\pi i}\left(\frac{1}{2\pi i}\right)^{k-2}\sum_{1\leq A}\,\sum_{0\leq j<c}\left(\frac{\chi_1(A)\overline{\chi_2}(j)}{A^{k-1}}\right)B_1\left(\frac{j}{c}\right)e_c(Aaj)=\frac{S_{\chi_1,\chi_2,k}(\gamma)}{\tau(\overline{\chi_1})(k-1)}.$$
    Simplifying the right hand side, the desired identity follows. Using this identity and recalling that $\chi_1\chi_2(-1)=(-1)^k$ we have the symmetric formula
    $$\phi_{\chi_1,\chi_2,k}(\gamma,1,-a/c)=\frac{1}{2}\left(\phi_{\chi_1,\chi_2,k}(\gamma,1,-a/c)+\chi_1(-1)\,\overline{\phi}_{\overline{\chi_1},\overline{\chi_2},k}(\gamma,1,-a/c)\right).$$
    
    Now we return to the proof of Theorem \ref{mainThm}. Using our new symmetry relation we can rewrite (\ref{pickupPoint}) as
    \begin{align*}
        \phi_{\chi_1,\chi_2,k}(\gamma,1,-a/c)=\,&-\left(\frac{(k-2)!}{2\pi i}\right)\left(-\frac{1}{2\pi i}\right)^{k-2}\sum_{1\leq A}\,\sum_{j\bmod c}\left(\frac{\chi_1(A)\overline{\chi_2}(j)}{A^{k-1}}\right)B_1\left(\frac{j}{c}\right)e_c(Aaj) \\
        &+\chi_1(-1)\left(\frac{(k-2)!}{2\pi i}\right)\left(\frac{1}{2\pi i}\right)^{k-2}\sum_{1\leq A}\,\sum_{j\bmod c}\left(\frac{\chi_1(A)\overline{\chi_2}(j)}{A^{k-1}}\right)B_1\left(\frac{j}{c}\right)e_c(-Aaj).
    \end{align*}
    Noting that $A^{-k+1}=(-1)^{-k+1}(-A)^{-k+1}$, we get
    \begin{align*}
        \phi_{\chi_1,\chi_2,k}(\gamma,1,-a/c)=\,&-\left(\frac{(k-2)!}{2\pi i}\right)\left(-\frac{1}{2\pi i}\right)^{k-2}\sum_{1\leq A}\,\sum_{j\bmod c}\left(\frac{\chi_1(A)\overline{\chi_2}(j)}{A^{k-1}}\right)B_1\left(\frac{j}{c}\right)e_c(Aaj) \\
        &-\left(\frac{(k-2)!}{2\pi i}\right)\left(-\frac{1}{2\pi i}\right)^{k-2}\sum_{1\leq A}\,\sum_{j\bmod c}\left(\frac{\chi_1(-A)\overline{\chi_2}(j)}{(-A)^{k-1}}\right)B_1\left(\frac{j}{c}\right)e_c(-Aaj)
    \end{align*}
    Re-indexing and swapping the order of summation we have
    $$\phi_{\chi_1,\chi_2,k}(\gamma,1,-a/c)=-\left(\frac{(k-2)!}{2\pi i}\right)\left(-\frac{1}{2\pi i}\right)^{k-2}\sum_{j\bmod c}\overline{\chi_2}(j)B_1\left(\frac{j}{c}\right)\sum_{A\neq 0}\frac{\chi_1(A)}{A^{k-1}}e_c(Aaj).$$
    Recognizing the inner sum from Definition \ref{berndt1} we rewrite our expression as
    $$\phi_{\chi_1,\chi_2,k}(\gamma,1,-a/c)=-\left(\frac{(k-2)!}{2\pi i}\right)\left(-\frac{1}{2\pi i}\right)^{k-2}\left(\frac{iq_1(2\pi/q_1)^{k-1}}{(-i)^k\tau(\overline{\chi_1})(k-1)!}\right)\sum_{j\bmod c}\overline{\chi_2}(j)B_1\left(\frac{j}{c}\right)B_{k-1,\chi_1}\left(\frac{ajq_1}{c}\right).$$
    Now applying Lemma \ref{berndt31} gives us the equivalent expression
    $$=-\left(\frac{(k-2)!}{2\pi i}\right)\left(-\frac{1}{2\pi i}\right)^{k-2}\left(\frac{iq_1(2\pi/q_1)^{k-1}}{(-i)^k\tau(\overline{\chi_1})(k-1)!}\right)q_1^{k-2}\mathop{\sum\sum}_{\substack{j\bmod c \\ n\bmod q_1}}\overline{\chi_1}(n)\overline{\chi_2}(j)B_1\left(\frac{j}{c}\right)B_{k-1}\left(\frac{aj}{c}+\frac{n}{q_1}\right).$$
    Simplifying the leading coefficient gives
    $$\phi_{\chi_1,\chi_2,k}(\gamma,1,-a/c)=\frac{1}{(-1)^k\tau(\overline{\chi_1})(k-1)}\mathop{\sum\sum}_{\substack{j\bmod c \\ n\bmod q_1}}\overline{\chi_1}(n)\overline{\chi_2}(j)B_1\left(\frac{j}{c}\right)B_{k-1}\left(\frac{aj}{c}+\frac{n}{q_1}\right).$$
    Removing the leading scalar factor from the right hand side of the equation completes the proof for $k\geq 3$. Referencing Theorem $1.2$ of \cite{SVY} completes the proof for the $k=2$ case.
    \end{proof}
    
\section{Cohomological Properties of $\phi_{\chi_1,\chi_2,k}$ and $S_{\chi_1,\chi_2,2}$}\label{cohomology}
    We seek to recover analogues of previous cohomological results involving $S_{\chi_1,\chi_2,2}$ in an effort to better understand how $\phi_{\chi_1,\chi_2,k}$ and $S _{\chi_1,\chi_2,k}$ behave when $k\geq 3$. Recall the definition of the slash operator in Definition \ref{slashDef}, as we will use this to demonstrate that $\phi_{\chi_1,\chi_2,k}$ exhibits a crossed homomorphism relation similar to Lemma $2.2$ of \cite{SVY}.
    
    We endow $V_{k-2}(\mathbb{C})$ with a group action as follows:
    \begin{definition}
        Recall that $V_{k-2}(\mathbb{C})$ is the vector space of degree $k-2$ homogeneous polynomials in two variables having complex coefficients. Let $Q(X,Y)$ be a polynomial in $V_{k-2}(\mathbb{C})$; given a Dirichlet character $\chi$ modulo $N$, we define $V_{k-2}^\chi(\mathbb{C})$ as the same vector space of polynomials with the right group action $V_{k-2}(\mathbb{C})\times\Gamma_0(N)\to V_{k-2}(\mathbb{C})$ via the map
        $$\left(Q\cdot\begin{pmatrix}
            a & b \\
            c & d
        \end{pmatrix}\right)\begin{pmatrix}
            X & Y
        \end{pmatrix}\mapsto\chi(d)\,Q\left(\begin{pmatrix}
            X & Y
        \end{pmatrix}\begin{pmatrix}
            a & b \\
            c & d
        \end{pmatrix}\right).$$
    \end{definition}
    \begin{remark}
        For fixed $z$ note that $P_{k-2}(z;X,Y)\in V_{k-2}^\chi(\mathbb{C})$, so for $\gamma\in\Gamma_0(N)$ we have $(P_{k-2}\cdot\gamma)(z;X,Y)\in V_{k-2}^\chi(\mathbb{C})$. Additionally, if we fix $X$ and $Y$ then we can view $P_{k-2}(z;X,Y)$ as a function of $z$. With this in mind, one has the relation
        $$(P_{k-2}\cdot\gamma)(z;X,Y)=\chi(\gamma)\,P_{k-2}|_{2-k}\gamma(z;X,Y).$$
        Here we view the left hand side as fixing $z$, applying the group action, and then evaluating at $X$ and $Y$. For the right hand side, we view this as fixing $X$ and $Y$, slashing with respect to the $z$ variable, and then evaluting at the same $z$.
    \end{remark}
    \noindent We now prove some general change of variable formulae on integrals that we will make use of in future sections.
    \begin{lemma}\label{integralCOV1}
        Let $\gamma_1,\gamma_2\in\Gamma_0(q_1q_2)$. We have that
        $$\int_{\gamma_1\infty}^{\gamma_1\gamma_2\infty}E_{\chi_1,\chi_2,k}(z)\,P_{k-2}(z;X,Y)\,dz=\int_\infty^{\gamma_2\infty}E_{\chi_1,\chi_2,k}(z)(P_{k-2}\cdot\gamma_1)(z;X,Y)\,dz.$$
    \end{lemma}
    \begin{proof}
        We can change variables using $\gamma_1u=z$ and $j(\gamma_1,u)^{-2}\,du=dz$; doing so gives
        \begin{align*}
            \int_{\gamma_1\infty}^{\gamma_1\gamma_2\infty}E_{\chi_1,\chi_2,k}(z)P_{k-2}(z;X,Y)\,dz&=\int_\infty^{\gamma_2\infty}j(\gamma_1,u)^kE_{\chi_1,\chi_2,k}|_k\gamma_1(u)\left(\frac{P_{k-2}|_{2-k}\gamma_1(u;X,Y)}{j(\gamma_1,u)^{k-2}}\right)\left(\frac{du}{j(\gamma_1,u)^2}\right) \\
            &=\int_\infty^{\gamma_2\infty}E_{\chi_1,\chi_2,k}|_k\gamma_1(u)\,P_{k-2}|_{2-k}\gamma_1(u;X,Y)\,du \\
            &=\int_\infty^{\gamma_2\infty}\psi(\gamma_1)E_{\chi_1,\chi_2,k}(u)\,P_{k-2}|_{2-k}\gamma_1(u;X,Y)\,du \\
            &=\int_\infty^{\gamma_2\infty}E_{\chi_1,\chi_2,k}(u)(P_{k-2}\cdot\gamma_1)(u;X,Y)\,du
        \end{align*}
        as desired.
    \end{proof}
    \begin{lemma}\label{integralCOV2}
        For $\mathfrak{a},\mathfrak{b}\in\mathbb{P}^1(\mathbb{Q})$ and $\gamma\in\mathrm{SL}_2(\mathbb{R})$ one has
        $$\int_\mathfrak{b}^{\gamma\mathfrak{a}}E_{\chi_1,\chi_2,k}(z)\,P_{k-2}(z;1,-\gamma\mathfrak{a})\,dz=j(\gamma,\mathfrak{a})^{2-k}\int_{\gamma^{-1}\mathfrak{b}}^\mathfrak{a}E_{\chi_1,\chi_2,k}|_k\gamma(z)\,P_{k-2}(z;1,-\mathfrak{a})\,dz.$$
    \end{lemma}
    \begin{proof}
        We can change variables using $\gamma u=z$ and $j(\gamma,u)^{-2}\,du=dz$; doing so gives
        $$\int_\mathfrak{b}^{\gamma\mathfrak{a}}E_{\chi_1,\chi_2,k}(z)\,P_{k-2}(z;1,-\gamma\mathfrak{a})\,dz=\int_{\gamma^{-1}\mathfrak{b}}^\mathfrak{a}j(\gamma,u)^k\,E_{\chi_1,\chi_2,k}|_k\gamma(u)\,P_{k-2}(\gamma u;1,-\gamma\mathfrak{a})\left(\frac{du}{j(\gamma,u)^2}\right).$$
        Noting that
        $$P_{k-2}(\gamma u;1,-\gamma\mathfrak{a})=\frac{P_{k-2}(u;1,-\mathfrak{a})}{j(\gamma,u)^{k-2}\,j(\gamma,\mathfrak{a})^{k-2}}$$
        and substituting this into the above completes the proof.
    \end{proof}
    \noindent Now we prove that $\phi_{\chi_1,\chi_2,k}$ is a crossed homomorphism.
    \begin{remark}
        The below proof is standard in the literature for the Eichler-Shimura isomorphism; however, much of the literature focuses on the application to cusp forms. For the context relevant to this work, a good reference on the Eichler-Shimura isomorphism applied in the Eisenstein part can be found in \cite{modularSymbolsEisenstein}. Seen as there is little overlap between the literature on the Eichler-Shimura isomorphism and the literature on Dedekind sums, the inclusion of this proof may help those who are unfamiliar with the Eichler-Shimura isomorphism.
    \end{remark}
    \begin{lemma}\label{crossHom}
         We have that $\phi_{\chi_1,\chi_2,k}$ is a crossed homomorphism; specifically, for $\gamma_1,\gamma_2\in\Gamma_0(q_1q_2)$ with $\gamma_1=\begin{psmallmatrix}
            a & b \\
            c & d
        \end{psmallmatrix}$,
        $$\phi_{\chi_1,\chi_2,k}(\gamma_1\gamma_2,X,Y)=\phi_{\chi_1,\chi_2,k}(\gamma_1,X,Y)+\psi(\gamma_1)\,\phi_{\chi_1,\chi_2,k}(\gamma_2,aX+cY,bX+dY).$$
    \end{lemma}
    \begin{proof}
        Noting that $E_{\chi_1,\chi_2,k}(z)P_{k-2}(z;X,Y)$ is holomorphic, by path independence we have that
        \begin{align*}
            \phi_{\chi_1,\chi_2,k}(\gamma_1\gamma_2,X,Y)&=\int_\infty^{\gamma_1\gamma_2\infty}E_{\chi_1,\chi_2,k}(z)P_{k-2}(z;X,Y)\,dz \\
            &=\int_\infty^{\gamma_1\infty}E_{\chi_1,\chi_2,k}(z)P_{k-2}(z;X,Y)\,dz+\int_{\gamma_1\infty}^{\gamma_1\gamma_2\infty}E_{\chi_1,\chi_2,k}(z)P_{k-2}(z;X,Y)\,dz.
        \end{align*}
        Applying Lemma \ref{integralCOV1} we have
        \begin{align*}
            \phi_{\chi_1,\chi_2,k}(\gamma_1\gamma_2,X,Y)&=\int_\infty^{\gamma_1\infty}E_{\chi_1,\chi_2,k}(z)P_{k-2}(z;X,Y)\,dz+\int_\infty^{\gamma_2\infty}E_{\chi_1,\chi_2,k}(z)(P_{k-2}\cdot\gamma_1)(z;X,Y)\,dz \\
            &=\phi_{\chi_1,\chi_2,k}(\gamma_1,X,Y)+\psi(\gamma_1)\,\phi_{\chi_1,\chi_2,k}(\gamma_2,aX+cY,bX+dY)
        \end{align*}
        completing our proof as desired.
    \end{proof}
    \begin{remark}
        From Lemma \ref{crossHom}, $\phi_{\chi_1,\chi_2,k}$ can be viewed as an element of the cohomology group $H^1(\Gamma_0(q_1q_2),V_{k-2}^\psi(\mathbb{C}))$. For further commentary on this aspect of the cohomological properties of Dedekind sums, particularly as it relates to $S_{\chi_1,\chi_2,2}$, see Section $4$ of \cite{majorMajure}.
    \end{remark}
    Noting that $\phi_{\chi_1,\chi_2,2}$ is independent of the choice of $X$ and $Y$ allows us to recover the crossed homomorphism relation for $S_{\chi_1,\chi_2,2}$ from a different approach.
    \begin{corollary}[\cite{SVY}, Lemma $2.2$]\label{crossHomSpecial}
        We can view $S_{\chi_1,\chi_2,2}$ as an element of the space $\text{Hom}(\Gamma_1(q_1q_2),\mathbb{C})$ and we have a crossed homomorphism relation for $S_{\chi_1,\chi_2,2}$ given by
        $$S_{\chi_1,\chi_2,2}(\gamma_1\gamma_2)=S_{\chi_1,\chi_2,2}(\gamma_1)+\psi(\gamma_1)\,S_{\chi_1,\chi_2,2}(\gamma_2).$$
    \end{corollary}
    \begin{proof}
        Note that the period integral is independent of $X$ and $Y$ when $k=2$. So it follows that given $\gamma_1,\gamma_2\in\Gamma_0(q_1q_2)$ with $\gamma_1=\begin{psmallmatrix}
            a & b \\
            c & d
        \end{psmallmatrix}$ then we have that 
        $$\phi_{\chi_1,\chi_2,2}(\gamma_1,X,Y)=\int_\infty^{\gamma_1\infty}E_{\chi_1,\chi_2,2}(z)P_0(z;X,Y)\,dz=\int_\infty^{\gamma_1\infty}E_{\chi_1,\chi_2,2}(z)P_0(z;1,-a/c)\,dz=\frac{S_{\chi_1,\chi_2,2}(\gamma_1)}{(-1)^k\,\tau(\overline{\chi_1})(k-1)}.$$
        Using this independence in combination with Lemma \ref{crossHom} we have
        \begin{align*}
            \phi_{\chi_1,\chi_2,2}(\gamma_1\gamma_2,X,Y)&=\phi_{\chi_1,\chi_2,2}(\gamma_1,X,Y)+\psi(\gamma_1)\,\phi_{\chi_1,\chi_2,2}(\gamma_2,aX+cY,bX+dY) \\
            &=\phi_{\chi_1,\chi_2,2}(\gamma_1,X,Y)+\psi(\gamma_1)\int_\infty^{\gamma_2\infty}E_{\chi_1,\chi_2,2}(z)P_0(z;aX+cY,bX+dY)\,dz \\
            &=\phi_{\chi_1,\chi_2,2}(\gamma_1,X,Y)+\psi(\gamma_1)\int_\infty^{\gamma_2\infty}E_{\chi_1,\chi_2,2}(z)P_0(z;X,Y)\,dz \\
            &=\phi_{\chi_1,\chi_2,2}(\gamma_1,X,Y)+\psi(\gamma_1)\,\phi_{\chi_1,\chi_2,2}(\gamma_2,X,Y).
        \end{align*}
        This explicitly gives a crossed homomorphism relation for $S_{\chi_1,\chi_2,2}$ up to scalar multiple. So we have
        $$S_{\chi_1,\chi_2,2}(\gamma_1\gamma_2)=S_{\chi_1,\chi_2,2}(\gamma_1)+\psi(\gamma_1)\,S_{\chi_1,\chi_2,2}(\gamma_2).$$
        Noting that $\psi(\gamma_1)$ is trivial on $\Gamma_1(q_1q_2)$ we also have that $S_{\chi_1,\chi_2,2}\in\text{Hom}(\Gamma_1(q_1q_2),\mathbb{C})$. So we recover all of the results of Lemma $2.2$ in \cite{SVY} as desired.
    \end{proof}

\section{A Non-Trivial Upper Bound on $\abs{S_{\chi_1,\chi_2,k}}$}
    % Now to prove our non-trivial bound on $S_{\chi_1,\chi_2,k}$ we first derive an alternative independent point formula for $\phi_{\chi_1,\chi_2,k}(\gamma,1,-a/c)$ as we alluded to in the introduction.
    
    \subsection{Higher Weight Independent Point Formula}\label{cobproof}
    We now derive an analogue of the independent point formula for $\phi_{\chi_1,\chi_2,k}(\gamma,1,-a/c)$ as we alluded to in the introduction with (\ref{basicIndepPointFormula}).
    \begin{lemma}\label{cobformula}
        For all $z_1\in\mathcal{H}$ we have that
        \begin{align*}
            \phi_{\chi_1,\chi_2,k}(\gamma,1,-a/c)&=-\frac{\psi(d)}{\pi i}\left(-\frac{1}{c}\right)^{k-2}\sum_{1\leq A}\frac{\chi_1(A)}{A}\sum_{1\leq B}\overline{\chi_2}(B)B^{k-2}e(ABz_1) \\
            &\qquad-2\sum_{n=0}^{k-2}\frac{(\gamma z_1-a/c)^{k-n-2}\,(k-2)!}{(-2\pi i)^{n+1}\,(k-n-2)!}\sum_{1\leq A}\frac{\chi_1(A)}{A^{n+1}}\sum_{1\leq B}\overline{\chi_2}(B)B^{k-n-2}e(AB\gamma z_1).
        \end{align*}
        Note that the left hand side is independent of our choice of $z_1\in\mathcal{H}$.
    \end{lemma}
    \begin{remark}
        Recall that by definition $S_{\chi_1,\chi_2,k}(\gamma)$ and $\phi_{\chi_1,\chi_2,k}(\gamma,1,-a/c)$ differ by only a scalar factor, so the above serves as an independent point formula for $S_{\chi_1,\chi_2,k}(\gamma)$ as well.
    \end{remark}
    \begin{proof}
        Note that for all $z_0,z_1\in\mathcal{H}$ we have that
        \begin{align*}
            &\int_{z_0}^{\gamma z_0}E_{\chi_1,\chi_2,k}(z)P_{k-2}(z;X,Y)\,dz \\
            &\qquad\qquad=\int_{z_0}^{z_1}E_{\chi_1,\chi_2,k}(z)P_{k-2}(z;X,Y)\,dz+\int_{z_1}^{\gamma z_1}E_{\chi_1,\chi_2,k}(z)P_{k-2}(z;X,Y)\,dz+\int_{\gamma z_1}^{\gamma z_0}E_{\chi_1,\chi_2,k}(z)P_{k-2}(z;X,Y)\,dz
        \end{align*}
        by Cauchy's integral theorem. Now focusing on the last integral, by Lemma \ref{integralCOV1} one has that
        $$\int_{\gamma z_1}^{\gamma z_0}E_{\chi_1,\chi_2,k}(z)\,P_{k-2}(z;X,Y)\,dz=-\int_{z_0}^{z_1}E_{\chi_1,\chi_2,k}(z)(P_{k-2}\cdot\gamma)(z;X,Y)\,dz.$$
        Thus we have that
        \begin{align*}
            &\int_{z_0}^{\gamma z_0}E_{\chi_1,\chi_2,k}(z)P_{k-2}(z;X,Y)\,dz \\
            &\qquad\qquad=\int_{z_1}^{\gamma z_1}E_{\chi_1,\chi_2,k}(z)P_{k-2}(z;X,Y)\,dz+\int_{z_0}^{z_1}E_{\chi_1,\chi_2,k}(z)(P_{k-2}(z;X,Y)-(P_{k-2}\cdot\gamma)(z;X,Y))\,dz.
        \end{align*}
        Now by taking $z_0=-d/c+i/(c^2u)$, $\gamma z_0=a/c+iu$, and specializing $X=1$ and $Y=-a/c$, then in the limit note
        \begin{align*}
            &\phi_{\chi_1,\chi_2,k}(\gamma,1,-a/c) \\
            &\qquad\qquad=\int_{z_1}^{\gamma z_1}E_{\chi_1,\chi_2,k}(z)P_{k-2}(z;1,-a/c)\,dz+\lim_{u\to 0^+}\int_{z_0}^{z_1}E_{\chi_1,\chi_2,k}(z)(P_{k-2}(z;1,-a/c)-\psi(d)(-1/c)^{k-2})\,dz
        \end{align*}
        for all $z_1\in\mathcal{H}$ independent of choice. Applying linearity we have that
        \begin{equation}\label{returnToStep}
            \phi_{\chi_1,\chi_2,k}(\gamma,1,-a/c)=\lim_{u\to 0^+}\int_{z_0}^{\gamma z_1}E_{\chi_1,\chi_2,k}(z)P_{k-2}(z;1,-a/c)\,dz-\psi(d)\left(-\frac{1}{c}\right)^{k-2}\lim_{u\to 0^+}\int_{z_0}^{z_1}E_{\chi_1,\chi_2,k}(z)\,dz.
        \end{equation}
        
        Focusing on the first integral of (\ref{returnToStep}) and applying Lemma \ref{generalIntegralFormula} gives
        $$\int_{z_0}^{\gamma z_1}E_{\chi_1,\chi_2,k}(z)P_{k-2}(z;1,-a/c)\,dz=2\sum_{1\leq N}\sum_{A\mid N}\sum_{n=0}^{k-2}\left(\frac{\chi_1(A)\overline{\chi_2}(N/A)N^{k-n-2}}{-A^{k-1}(-2\pi i)^{n+1}}\left(\frac{d^n}{dz^n}(z-a/c)^{k-2}\right)e(Nz)\right)\Bigg|_{z_0}^{\gamma z_1}.$$
        Note that as $u\to 0^+$, we have polynomial growth of $(d^n/dz^n)(z_0-a/c)^{k-2}$ and exponential decay of $e(Nz_0)$. So every term in this sum vanishes in the limit $u\to 0^+$ at $z_0$, thus we are left with
        $$\lim_{u\to 0^+}\int_{z_0}^{\gamma z_1}E_{\chi_1,\chi_2,k}(z)P_{k-2}(z;1,-a/c)\,dz=2\sum_{1\leq N}\sum_{A\mid N}\sum_{n=0}^{k-2}\frac{\chi_1(A)\overline{\chi_2}(N/A)N^{k-n-2}\,(k-2)!}{-A^{k-1}(-2\pi i)^{n+1}\,(k-n-2)!}\,(\gamma z_1-a/c)^{k-n-2}\,e(N\gamma z_1).$$
        After re-indexing this yields the second main term in Theorem \ref{cobformula}, namely
        \begin{align*}
            &\lim_{u\to 0^+}\int_{z_0}^{\gamma z_1}E_{\chi_1,\chi_2,k}(z)P_{k-2}(z;1,-a/c)\,dz \\
            &\qquad\qquad=-2\sum_{n=0}^{k-2}\frac{(\gamma z_1-a/c)^{k-n-2}\,(k-2)!}{(-2\pi i)^{n+1}\,(k-n-2)!}\sum_{1\leq A}\frac{\chi_1(A)}{A^{n+1}}\sum_{1\leq B}\overline{\chi_2}(B)B^{k-n-2}e(AB\gamma z_1).
        \end{align*}

        Now focusing on the second integral of (\ref{returnToStep}), we note that this is just a special case of the above integral where $k=2$ and the endpoint is $z_1$ instead of $\gamma z_1$. So we have
        $$\lim_{u\to 0^+}\int_{z_0}^{z_1}E_{\chi_1,\chi_2,k}(z)\,dz=\frac{1}{\pi i}\sum_{1\leq A}\frac{\chi_1(A)}{A}\sum_{1\leq B}\overline{\chi_2}(B)B^{k-2}e(ABz_1).$$
        Substituting these two integrals into (\ref{returnToStep}) completes the proof.
    \end{proof}

    \subsection{Literature Results for the Proof of Theorem \ref{nontrivialbound}}
    In this section we note results from \cite{korobov}, \cite{hensley}, and \cite{cy24} that will be used to prove Theorem \ref{nontrivialbound}. First, recalling Definition \ref{partialQuoDef}, we present a result from \cite{korobov}.
    \begin{lemma}[\cite{korobov}, Lemma $3$ and ($19$)]\label{korobovBounds}
        Let $(a,c)=1$ and $c'=c/q_2$. Then we have that
        $$\sum_{A=1}^{c'-1}\norm{\frac{Aa}{c'}}^{-1}\leq 2c'\log c',\qquad\sum_{A=1}^{c'-1}\frac{1}{A\,\norm{Aa/c'}}\leq 18M(a/c')\,\log^2 c',\qquad\text{and}\qquad\sum_{A=1}^j\norm{\frac{Aa}{c'}}^{-1}\leq 16jM(a/c')\,\log c'$$
        where $j$ is arbitrary and $\norm{\cdot}$ is the distance to nearest integer function.
    \end{lemma}
    \noindent Now we use the bounds of Lemma \ref{korobovBounds} to deduce a more general bound on sums of similar type.
    \begin{corollary}\label{genKorobovBounds} Let $(a,c)=1$ and $c'=c/q_2$. For $K\geq 0$ we have that
        $$\sum_{A=1}^{c'-1}\frac{1}{A^{K+1}\,\norm{Aa/c'}}\leq 18M(a/c')\,\log^2c'.$$
    \end{corollary}
    \begin{proof}
        This immediately follows from the middle inequality of Lemma \ref{korobovBounds}.
    \end{proof}
    \begin{remark}
        One can do better than this bound for $K\geq 1$; however, doing so doesn't help our bound in (\ref{nonTrivialBoundEqu}) as we will always have a $K=0$ term which dominates.
        % It may seem reasonable to assume we could get a tighter bound by using similar techniques as in the proof of Lemma \ref{korobovBounds}; however, this turns out not to be the case. Following the method in \cite{korobov}, we apply Abel's summation formula for $K\geq 1$ to get
        % $$\sum_{A=1}^{c'-1}\frac{1}{A^{K+1}\,\norm{Aa/c'}}\leq 2M(a/c')\,(c'^{-K}+8\,\zeta(K+1))\log c'$$
        % One can use this tighter estimate to get a slightly better upper bound on the size of $\abs{S_{\chi_1,\chi_2,k}}$ but it does little improvement and makes the bound needlessly complex, as there will always be terms requiring the $K=0$ growth.
    \end{remark}
    
    Now we introduce a few results of \cite{hensley} and \cite{cy24} relating to partial quotients of continued fraction expansions.
    \begin{lemma}[\cite{cy24}, Corollary $2.4$]\label{partialQuotientDiff}
        Let $\gamma=\begin{psmallmatrix}
            a & b \\
            c & d
        \end{psmallmatrix}\in\Gamma_0(q_1q_2)$ with $c'=c/q_2$. Then $$M(a/c')=M(d/c')+\delta$$
        where $\abs{\delta}\leq 1$.
    \end{lemma}
    \noindent A result of \cite{hensley} allows us to describe how often $M(d/c')$ and $M(a/c')$ are less than size $\log c'$.
    \begin{lemma}[\cite{hensley}, Theorem $1$]\label{hensleyDensity}
        Define
        $$\Phi(\alpha,C)=\abs{\{(r,s):1<r<s\leq C,\,\gcd(r,s)=1,\,M(r/s)\leq\alpha\,\log C\}}.$$
        We have that
        $$\Phi(\alpha,C)=\frac{3}{\pi^2}C^2\exp\left(-\frac{12}{\alpha\pi^2}\right)\left(1+O\left(\left(1+\frac{1}{\alpha^2}\right)\exp\left(\frac{24}{\alpha\pi^2}\right)\frac{\log\log C}{\log C}\right)\right)$$
        uniformly in $\alpha>4/\log\log C$ as $C\to\infty$.
    \end{lemma}
    \noindent Finally, \cite{cy24} bounded the frequency of large values of $M(a/c')$ as a corollary to Lemma \ref{hensleyDensity}.
    \begin{corollary}[\cite{cy24}, Corollary $2.2$]\label{georgiaDensity}
        Define
        $$\widehat{\Phi}(\alpha,C)=\abs{\{(r,s):1<r<s\leq C,\,\gcd(r,s)=1,\,M(r/s)>\alpha\log C\}}.$$
        We have that
        $$\widehat{\Phi}(\alpha,C)\ll\frac{C^2}{\alpha}+C^2\,\frac{\log\log C}{\log C}$$
        uniformly in $\alpha>1$ as $C\to\infty$.
    \end{corollary}

    \subsection{Proof of Theorem \ref{nontrivialbound}}\label{nontrivialboundproof}
    Before we begin the proof of Theorem \ref{nontrivialbound} we need to introduce some preliminary results. We begin by bounding the inner-most sums of Theorem \ref{cobformula}. Then, a generalized result of Korobov will then allow us to bound the inner-most double sums of Theorem \ref{cobformula}.  Note that the method used for the proofs in this section roughly follow \cite{cy24}, though we of course have to deviate in some respects due to the higher weight setting.
    \begin{lemma}\label{goofyUpperBound}
        For $K\geq 0$ and $z\in\mathcal{H}$ with $\mathfrak{Im}(z)\leq 1$ we have the upper bound
        $$\abs{\sum_{1\leq B}\overline{\chi_2}(B)B^K\,e(ABz)}\ll_{K,\,A,\,q_2} \frac{\abs{e(Az)}}{\norm{\mathfrak{Re}(Aq_2z)}}\left(\frac{1}{\mathfrak{Im}(z)}\right)^K.$$
    \end{lemma}
    \begin{proof}
        We first break up our sum modulo $q_2$ where $B=j+mq_2$ with $1\leq j\leq q_2$ and $0\leq m$.
        $$\sum_{1\leq B}\overline{\chi_2}(B)B^K\,e(ABz)=\sum_{j=1}^{q_2}\overline{\chi_2}(j)\,e(Ajz)\sum_{0\leq m}(j+mq_2)^K\,e(Aq_2z)^m.$$
        Using binomial expansion we have that
        $$\sum_{1\leq B}\overline{\chi_2}(B)B^K\,e(ABz)=\sum_{j=1}^{q_2}\overline{\chi_2}(j)\,e(Ajz)\sum_{M=0}^K\binom{K}{M}j^{K-M}q_2^M\sum_{0\leq m}m^M\,e(Aq_2z)^m.$$
        Checking that $\abs{e(Aq_2z)}<1$ (since $z\in\mathcal{H}$) for convergence, the inner-most sum on the right hand side is a linear combination of the derivatives of the geometric sum $\sum_{0\leq m}x^m$ with respect to $x$ where $x=e(Aq_2z)$. So we have
        \begin{equation}\label{awfulFiniteSum}
            \sum_{1\leq B}\overline{\chi_2}(B)B^K\,e(ABz)=\sum_{j=1}^{q_2}\overline{\chi_2}(j)\,e(Ajz)\sum_{M=0}^K\binom{K}{M}j^{K-M}q_2^M\left(\frac{1}{1-e(Aq_2z)}\right)\sum_{N=0}^Ma_{M,\,N}\left(\frac{e(Aq_2z)}{1-e(Aq_2z)}\right)^N
        \end{equation}
        where $a_{M,\,N}$ are the appropriate coefficients. Now since $z\in\mathcal{H}$ we have the inequality
        \begin{equation}\label{inequality1}
            \frac{1}{\abs{1-e(Aq_2z)}}\leq\frac{1}{2\,\norm{\mathfrak{Re}(Aq_2z)}}.
        \end{equation}
        Additionally since $z\in\mathcal{H}$ we also have that
        $$\abs{\frac{e(Aq_2z)}{1-e(Aq_2z)}}\leq\frac{\exp(-2\pi Aq_2\cdot\mathfrak{Im}(z))}{1-\exp(-2\pi Aq_2\cdot\mathfrak{Im}(z))}.$$
        By a calculus exercise we can show that $e^{-y}/(1-e^{-y})\leq 1/y$. So we have $\abs{e(Aq_2z)/(1-e(Aq_2z)}\ll_{A,\,q_2}1/\mathfrak{Im}(z)$.
        % $$\abs{\frac{e(Aq_2z)}{1-e(Aq_2z)}}\leq\frac{1}{2\pi Aq_2\cdot\mathfrak{Im}(z)}\ll_{A,\,q_2}\frac{1}{\mathfrak{Im}(z)}.$$
        Now because $\mathfrak{Im}(z)\leq 1$, we have that
        \begin{equation}\label{inequality2}
            \abs{\frac{e(Aq_2z)}{1-e(Aq_2z)}}^N\ll_{A,\,q_2}\left(\frac{1}{\mathfrak{Im}(z)}\right)^N\leq\left(\frac{1}{\mathfrak{Im}(z)}\right)^K.
        \end{equation}
        Applying inequalities (\ref{inequality1}) and (\ref{inequality2}) to (\ref{awfulFiniteSum}) and noting that $\abs{e(Ajz)}\leq\abs{e(Az)}$ (since $z\in\mathcal{H}$), by the triangle inequality the desired upper bound immediately follows.
        % $$\abs{\sum_{1\leq B}\overline{\chi_2}(B)B^K\,e(ABz)}\ll_{A,\,q_2}\frac{\abs{e(Az)}}{\norm{\mathfrak{Re}(Aq_2z)}}\left(\frac{1}{\mathfrak{Im}(z)}\right)^K\sum_{j=1}^{q_2}\sum_{M=0}^K\binom{K}{M}j^{K-M}q_2^M\sum_{N=0}^M\abs{a_{M,\,N}}$$
        % The triple sum on the right is some constant, so
        % The inner-most double summation is exactly $\mathcal{C}(K,j)$ where $a_{M,N}=b_{M,N}$. So by Lemma \ref{randomConstantNonDecreasing} we have that
        % \begin{align*}
        %     \abs{\sum_{1\leq B}\overline{\chi_2}(B)B^K\,e(ABz)}&\ll\frac{\abs{e(Az)}}{\norm{\mathfrak{Re}(Aq_2z)}}\left(\frac{1}{\mathfrak{Im}(z)}\right)^K\sum_{j=1}^{q_2}\mathcal{C}(K,j)\ll\frac{\mathcal{C}(K,q_2)\,\abs{e(Az)}}{\norm{\mathfrak{Re}(Aq_2z)}}\left(\frac{1}{\mathfrak{Im}(z)}\right)^K.
        % \end{align*}
        % Which is exactly the desired upper-bound.
    \end{proof}
    We now have the necessary machinery to bound the infinite double sums featured in Theorem \ref{cobformula}.
    \begin{lemma}\label{cobDoubleSumBound}
        Let $\gamma=\begin{psmallmatrix}
            a & b \\ c & d
        \end{psmallmatrix}\in\Gamma_0(q_1q_2)$ with $z_1=(i-d)/c$ and $\gamma z_1=(i+a)/c$ with $c\geq 1$. We have that
        $$\abs{\sum_{1\leq A}\frac{\chi_1(A)}{A^{n+1}}\sum_{1\leq B}\overline{\chi_2}(B)B^{k-n-2}\,e(ABz_1)}\ll M(d/c')\,c^{k-n-2}\log^2c'$$
        and likewise
        $$\abs{\sum_{1\leq A}\frac{\chi_1(A)}{A^{n+1}}\sum_{1\leq B}\overline{\chi_2}(B)B^{k-n-2}\,e(AB\gamma z_1)}\ll M(a/c')\,c^{k-n-2}\log^2c'$$
        where $c'=c/q_2$.
    \end{lemma}
    \begin{proof}
        Note that $\chi_1(kc')=0$ for all $k\in\mathbb{Z}$ because $q_1\mid c'$. Additionally we have that $\mathfrak{Im}(z_1)\leq 1$ since $c\geq 1$, so by Lemma \ref{goofyUpperBound} and the triangle inequality we have that
        \begin{equation}\label{coro54pickup}
            \abs{\sum_{1\leq A}\frac{\chi_1(A)}{A^{n+1}}\sum_{1\leq B}\overline{\chi_2}(B)B^{k-n-2}\,e(ABz_1)}\ll_{k,\,n,\,q_2}\frac{1}{(\mathfrak{Im}(z_1))^{k-n-2}}\sum_{\substack{1\leq A \\ c'\nmid A}}\frac{\abs{e(Az_1)}}{A^{n+1}\,\norm{-Ad/c'}}.
        \end{equation}
        Focusing on the sum and breaking it up modulo $c'$, and noting that $\abs{e(jz_1)}\leq 1$ (since $z_1\in\mathcal{H}$) we have that
        $$\sum_{\substack{1\leq A \\ c\nmid A}}\frac{\abs{e(Az_1)}}{A^{n+1}\,\norm{-Ad/c'}}=\sum_{0\leq m}\sum_{j=1}^{c'-1}\frac{\abs{e(jz_1)\,e(mc'z_1)}}{(j+mc')^{n+1}\,\norm{-jd/c'+md}}\leq\sum_{0\leq m}\exp(-2\pi m/q_2)\sum_{j=1}^{c'-1}\frac{1}{j^{n+1}\norm{jd/c'}}.$$
        But now applying Corollary \ref{genKorobovBounds} and noting that $\sum_{0\leq m}\exp(-2\pi m/q_2)$ converges, we have that
        $$\sum_{\substack{1\leq A \\ c\nmid A}}\frac{\abs{e(Az_1)}}{A^{n+1}\norm{-Ad/c'}}\ll_{n,q_2} M(d/c')\,\log^2c'.$$
        Substituting this back into (\ref{coro54pickup}) and combining it with the fact that $\mathfrak{Im}(z_1)=c^{-1}$ completes the proof.

        A similar process follows for the specialization at $\gamma z_1$.
    \end{proof}
    
    Now we prove Theorem \ref{nontrivialbound} using all the previous results in this section.
    \begin{proof}  
        Let $z_1=(i-d)/c$ and $\gamma z_1=(i+a)/c$ with $c'=c/q_2$. Using Theorem \ref{cobformula} and Lemma \ref{cobDoubleSumBound} we have that
        $$\abs{\phi_{\chi_1,\chi_2,k}(\gamma,1,-a/c)}\ll_{k,\,q_2}(M(d/c')+M(a/c'))\log^2c'.$$
        % \begin{align*}
        %     \abs{\phi_{\chi_1,\chi_2,k}(\gamma,1,-a/c)}&\leq\frac{q_2\,(9M(d/c')+q_2/(2\pi))\,\mathcal{C}(k-2,q_2)}{\pi^{k-1}}\cdot\log^2c' \\
        %     &\qquad+2\sum_{n=0}^{k-2}\frac{(k-2)!}{(2\pi)^{n+1}(k-n-2)!}\cdot\frac{q_2\,(9M(a/c')+q_2/(2\pi))\,\mathcal{C}(k-n-2,q_2)}{\pi^{k-n-2}}\cdot\log^2c'.
        % \end{align*}
        % Using Lemma \ref{randomConstantNonDecreasing} and some loose bounds we have
        % \begin{align*}
        %     \abs{\phi_{\chi_1,\chi_2,k}(\gamma,1,-a/c)}&\leq\frac{q_2\,(9M(d/c')+q_2/(2\pi))\,\mathcal{C}(k-2,q_2)}{\pi^{k-1}}\cdot\log^2c' \\
        %     &\qquad+\frac{q_2\,(9M(a/c')+q_2/(2\pi))\,\mathcal{C}(k-2,q_2)}{\pi^{k-1}}\cdot(k-2)!\cdot\log^2c'.
        % \end{align*}
        Recalling that $S_{\chi_1,\chi_2,k}(\gamma)=(-1)^k\tau(\overline{\chi_1})(k-1)\,\phi_{\chi_1,\chi_2,k}(\gamma,1,-a/c)$ and that $\abs{\tau(\overline{\chi_1})}=\sqrt{q_1}$ gives us that
        $$\abs{S_{\chi_1,\chi_2,k}(\gamma)}\ll_{k,\,q_1,\,q_2}(M(d/c')+M(a/c'))\log^2c'$$
        also. Applying Lemma \ref{partialQuotientDiff} gives us (\ref{nonTrivialBoundEqu}). Finally, by applying Lemma \ref{hensleyDensity} and Corollary \ref{georgiaDensity} we complete the proof.
    \end{proof}

% We choose to address the quantum modularity first, as the arithmetic properties remain conjecture. Furthermore, the ideas developed from quantum modularity help us to develop Theorem \ref{containmentTheorem} which partially addresses the contents of Conjecture \ref{conj611}.

\section{Quantum Modularity of $\widehat{S}_{\chi_1,\chi_2,k}$}\label{qmfSec}
\subsection{Proof of Theorem \ref{qmfThrm}}
We develop a preliminary result from which the quantum modularity of $\widehat{S}_{\chi_1,\chi_2,k}$ is immediately realized.
\begin{lemma}\label{qmfLemma}
    The function
    $$h_{\gamma,\chi_1,\chi_2,k}:(\Gamma_0(q_1q_2))(\infty)\setminus\{\infty\}\to\mathbb{C},\quad \mathfrak{a}\mapsto\widehat{S}_{\chi_1,\chi_2,k}(\mathfrak{a})-\widehat{S}_{\chi_1,\chi_2,k}|_{2-k}\gamma(\mathfrak{a})$$
    has the equivalent expression
    \begin{equation}\label{equqmfLemma}h_{\gamma,\chi_1,\chi_2,k}(\mathfrak{a})=(-1)^k\tau(\overline{\chi_1})(k-1)\,\psi(\gamma)\,\phi_{\chi_1,\chi_2,k}(\gamma^{-1},1,-\mathfrak{a})+(1-\psi(\gamma))\,\widehat{S}_{\chi_1,\chi_2,k}(\mathfrak{a}).\end{equation}
    When $\psi(\gamma)=1$ we have that $h_{\gamma,\chi_1,\chi_2,k}$ is a degree $k-2$ polynomial (in $\mathfrak{a}\in(\Gamma_0(q_1q_2))(\infty)\setminus\{\infty\}\subseteq\mathbb{Q}$).
\end{lemma}
\begin{proof}
    Note that
    $$\frac{h_{\gamma,\chi_1,\chi_2,k}(\mathfrak{a})}{(-1)^k\tau(\overline{\chi_1})(k-1)}=\int_\infty^\mathfrak{a}E_{\chi_1,\chi_2,k}(z)P_{k-2}(z;1,-\mathfrak{a})\,dz-j(\gamma,\mathfrak{a})^{k-2}\int_\infty^{\gamma\mathfrak{a}}E_{\chi_1,\chi_2,k}(z)\,P_{k-2}(z;1,-\gamma\mathfrak{a})\,dz.$$
    Applying Lemma \ref{integralCOV2} we have
    $$\frac{h_{\gamma,\chi_1,\chi_2,k}(\mathfrak{a})}{(-1)^k\tau(\overline{\chi_1})(k-1)}=\int_\infty^\mathfrak{a}E_{\chi_1,\chi_2,k}(z)P_{k-2}(z;1,-\mathfrak{a})\,dz-\int_{\gamma^{-1}\infty}^\mathfrak{a}E_{\chi_1,\chi_2,k}|_k\gamma(z)\,P_{k-2}(z;1,-\mathfrak{a})\,dz.$$
    Recalling that $E_{\chi_1,\chi_2,k}|_k\gamma(z)=\psi(\gamma)\,E_{\chi_1,\chi_2,k}(z)$ and substituting this above we get
    \begin{align*}
        \frac{h_{\gamma,\chi_1,\chi_2,k}(\mathfrak{a})}{(-1)^k\tau(\overline{\chi_1})(k-1)}&=\int_\infty^\mathfrak{a}E_{\chi_1,\chi_2,k}(z)P_{k-2}(z;1,-\mathfrak{a})\,dz+\psi(\gamma)\int_\mathfrak{a}^{\gamma^{-1}\infty}E_{\chi_1,\chi_2,k}(z)P_{k-2}(z;1,-\mathfrak{a})\,dz \\
        &=\psi(\gamma)\int_{\infty}^{\gamma^{-1}\infty}E_{\chi_1,\chi_2,k}(z)P_{k-2}(z;1,-\mathfrak{a})\,dz+(1-\psi(\gamma))\int_\infty^\mathfrak{a}E_{\chi_1,\chi_2,k}(z)P_{k-2}(z;1,-\mathfrak{a})\,dz.
    \end{align*}
    Substituting definitions and rescaling we have that
    $$h_{\gamma,\chi_1,\chi_2,k}(\mathfrak{a})=(-1)^k\tau(\overline{\chi_1})(k-1)\,\psi(\gamma)\,\phi_{\chi_1,\chi_2,k}(\gamma^{-1},1,-\mathfrak{a})+(1-\psi(\gamma))\,\widehat{S}_{\chi_1,\chi_2,k}(\mathfrak{a})$$
    as desired. Furthermore, note that when $\psi(\gamma)=1$ we have that $h_{\gamma,\chi_1,\chi_2,k}$ is a degree $k-2$ polynomial as the $(1-\psi(\gamma))\,\widehat{S}_{\chi_1,\chi_2,k}(\mathfrak{a})$ term vanishes.
\end{proof}
Now we prove Theorem \ref{qmfThrm} as follows.
\begin{proof}
    For all $\gamma\in\Gamma_1(q_1q_2)$, we have that $\psi(\gamma)=1$. So per Lemma \ref{qmfLemma}, $h_{\gamma,\chi_1,\chi_2,k}$ is a degree $k-2$ polynomial (in $\mathfrak{a}$) and is thus continuous with respect to the real topology. Restricting to the domain $(\Gamma_1(q_1q_2))(\infty)\setminus\{\infty\}\subseteq(\Gamma_0(q_1q_2))(\infty)\setminus\{\infty\}\subseteq\mathbb{Q}$, the set of cusps which are $\Gamma_1(q_1q_2)$-equivalent to $\infty$ completes the proof.
\end{proof}
% \begin{corollary}
%     It follows that $\widehat{S}_{\chi,\chi,k}$ is a weight $2-k$ quantum modular form on $\Gamma_0(q_1q_2)$.
% \end{corollary}
% \begin{proof}
%     For all $\gamma\in\Gamma_0(q_1q_2)$, we have that $\psi(\gamma)=\chi(\gamma)\,\overline{\chi}(\gamma)=\abs{\chi(\gamma)}^2=1$. So per Theorem $\ref{qmfLemma}$, $h_{\gamma,\chi,\chi,k}$ is a degree $k-2$ polynomial and is thus continuous with respect to the real topology which completes the proof.
% \end{proof}

Similar in spirit to the contents of Section \ref{cohomology}, Zagier notes that for a weight $k$ quantum modular form $f$, it follows that the corresponding $h_\gamma$ should satisfy a crossed homomorphism relation by construction (\cite{Zagier}, Example 2). Indeed, following the standard method of proof, we can verify this holds for $h_{\gamma,\chi_1,\chi_2,k}$.
\begin{lemma}\label{qmfCrossHom}
    For $\gamma_1,\gamma_2\in\Gamma_0(q_1q_2)$ we have the crossed homomorphism relation
    $$h_{\gamma_1\gamma_2,\chi_1,\chi_2,k}=h_{\gamma_1,\chi_1,\chi_2,k}|_{2-k}\gamma_2+h_{\gamma_2,\chi_1,\chi_2,k}$$
\end{lemma}
\begin{proof}
    Note that
    $$h_{\gamma_1,\chi_1,\chi_2,k}|_{2-k}\gamma_2=\widehat{S}_{\chi_1,\chi_2,k}|_{2-k}\gamma_2-\widehat{S}_{\chi_1,\chi_2,k}|_{2-k}\gamma_1|_{2-k}\gamma_2=\widehat{S}_{\chi_1,\chi_2,k}|_{2-k}\gamma_2-\widehat{S}_{\chi_1,\chi_2,k}|_{2-k}\gamma_1\gamma_2.$$
    So it then follows that
    \begin{align*}
        h_{\gamma_1,\chi_1,\chi_2,k}|_{2-k}\gamma_2+h_{\gamma_2,\chi_1,\chi_2,k}&=\widehat{S}_{\chi_1,\chi_2,k}|_{2-k}\gamma_2-\widehat{S}_{\chi_1,\chi_2,k}|_{2-k}\gamma_1\gamma_2+\widehat{S}_{\chi_1,\chi_2,k}-\widehat{S}_{\chi_1,\chi_2,k}|_{2-k}\gamma_2 \\
        &=\widehat{S}_{\chi_1,\chi_2,k}-\widehat{S}_{\chi_1,\chi_2,k}|_{2-k}\gamma_1\gamma_2=h_{\gamma_1\gamma_2,\chi_1,\chi_2,k}
    \end{align*}
    as desired.
\end{proof}
\begin{remark}
    Note that this relation is not a restatement of Lemma \ref{crossHom} or Lemma \ref{crossHomSpecial}. In the case of Lemma \ref{crossHom}, note that Lemma \ref{qmfCrossHom} implicitly depends on $S_{\chi_1,\chi_2,k}$, which is $\phi_{\chi_1,\chi_2,k}$ with specialized $X$ and $Y$ whereas Lemma \ref{crossHom} does not specialize $X$ or $Y$. Similarly, in the case of Lemma \ref{crossHomSpecial}, this Lemma is specialized to $k=2$ where as Lemma \ref{qmfCrossHom} does not specialize $k$.
\end{remark}
This result will be used later in Section \ref{frickeSec} to help us derive our Fricke reciprocity relation Theorem \ref{reciprocity}.

\subsection{Visualizations with Discussion}\label{vizDiscussion}
We prove that $\phi_{\chi_1,\chi_2,k}$ is invariant under multiplication by matrices in $\Gamma_\infty$, as this will be used later in this section to clarify some of our visualizations.
\begin{lemma}\label{Speriodicity}
    For matrices $\gamma\in\Gamma_\infty$ and $\gamma'\in\Gamma_0(q_1q_2)$, we have that
    $$S_{\chi_1,\chi_2,k}(\gamma\gamma')=S_{\chi_1,\chi_2,k}(\gamma').$$
\end{lemma}
\begin{proof}
    Recalling Definition \ref{eichlerIntegral}, then applying Lemma \ref{integralCOV2} and noting that $\gamma^{-1}\infty=\infty$ and $E_{\chi_1,\chi_2,k}|_k\gamma(z)=E_{\chi_1,\chi_2,k}(z)$ since $\gamma\in\Gamma_\infty$, we have
    \begin{align*}
        \phi_{\chi_1,\chi_2,k}(\gamma\gamma',1,-\gamma\gamma'\infty)&=\int_\infty^{\gamma\gamma'\infty}E_{\chi_1,\chi_2,k}(z)\,P_{k-2}(z;1,-\gamma\gamma'\infty)\,dz \\
        &=\int_{\gamma^{-1}\infty}^{\gamma'\infty}E_{\chi_1,\chi_2,k}|_k\gamma(z)\,P_{k-2}(z;1,-\gamma'\infty)\,dz \\
        &=\int_\infty^{\gamma'\infty}E_{\chi_1,\chi_2,k}(z)\,P_{k-2}(z;1,-\gamma'\infty)\,dz=\phi_{\chi_1,\chi_2,k}(\gamma',1,-\gamma'\infty).
    \end{align*}
    Thus by Definition \ref{Sdef} we have that
    \begin{align*}
        S_{\chi_1,\chi_2,k}(\gamma\gamma')&=(-1)^k\tau(\overline{\chi_1})(k-1)\,\phi_{\chi_1,\chi_2,k}(\gamma\gamma',1,-\gamma\gamma'\infty) \\
        &=(-1)^k\tau(\overline{\chi_1})(k-1)\,\phi_{\chi_1,\chi_2,k}(\gamma',1,-\gamma'\infty)=S_{\chi_1,\chi_2,k}(\gamma').
    \end{align*}
\end{proof}
\begin{corollary}\label{ShatPeriodicity}
    For all $n\in\mathbb{Z}$ and $\mathfrak{a}\in(\Gamma_0(q_1q_2))(\infty)$ we have that
    $$\widehat{S}_{\chi_1,\chi_2,k}(\mathfrak{a}+n)=\widehat{S}_{\chi_1,\chi_2,k}(\mathfrak{a}).$$
\end{corollary}
\begin{proof}
    Since $\mathfrak{a}\in(\Gamma_0(q_1q_2))(\infty)$ there exists $\gamma\in\Gamma_0(q_1q_2)$ such that $\mathfrak{a}=\gamma\infty$. Now let $T=\begin{psmallmatrix}
        1 & 1 \\ 0 & 1
    \end{psmallmatrix}$ and note that $T^n=\begin{psmallmatrix}
        1 & n \\ 0 & 1
    \end{psmallmatrix}\in\Gamma_\infty$. By a simple calculation $\mathfrak{a}+n=T^n\gamma\infty$; thus, by Definition \ref{shatDef} and Lemma \ref{Speriodicity} we have that
    $$\widehat{S}_{\chi_1,\chi_2,k}(\mathfrak{a}+n)=\widehat{S}_{\chi_1,\chi_2,k}(T^n\gamma\infty)=S_{\chi_1,\chi_2,k}(T^n\gamma)=S_{\chi_1,\chi_2,k}(\gamma)=\widehat{S}_{\chi_1,\chi_2,k}(\gamma\infty)=\widehat{S}_{\chi_1,\chi_2,k}(\mathfrak{a}).$$
\end{proof}
To highlight the distinctive visual style of quantum modular forms, on the next page we present some graphs of $\widehat{S}_{\chi_1,\chi_2,k}(\mathfrak{a})$ over a range of values $\mathfrak{a}\in(\Gamma_1(q_1q_2))(\infty)\subseteq\mathbb{P}^1(\mathbb{Q})$.
Let $\mathcal{C}_j(N)=\{\gamma\infty:\gamma=\begin{psmallmatrix}
    a & b \\ c & d
\end{psmallmatrix}\in\Gamma_1(N),\,1\leq a<jN,\,N\leq c<jN)\}$. Additionally, let $k=4$ and $\chi_1=\chi_2$ be the unique quadratic primitive character modulo $5$. First, we provide a scatter plot for $\widehat{S}_{\chi_1,\chi_2,k}(\mathcal{C}_{50}(q_1q_2))$.
\vspace*{-4mm}
\begin{figure}[htbp]
    \centering
    \includegraphics[width=.8\linewidth]{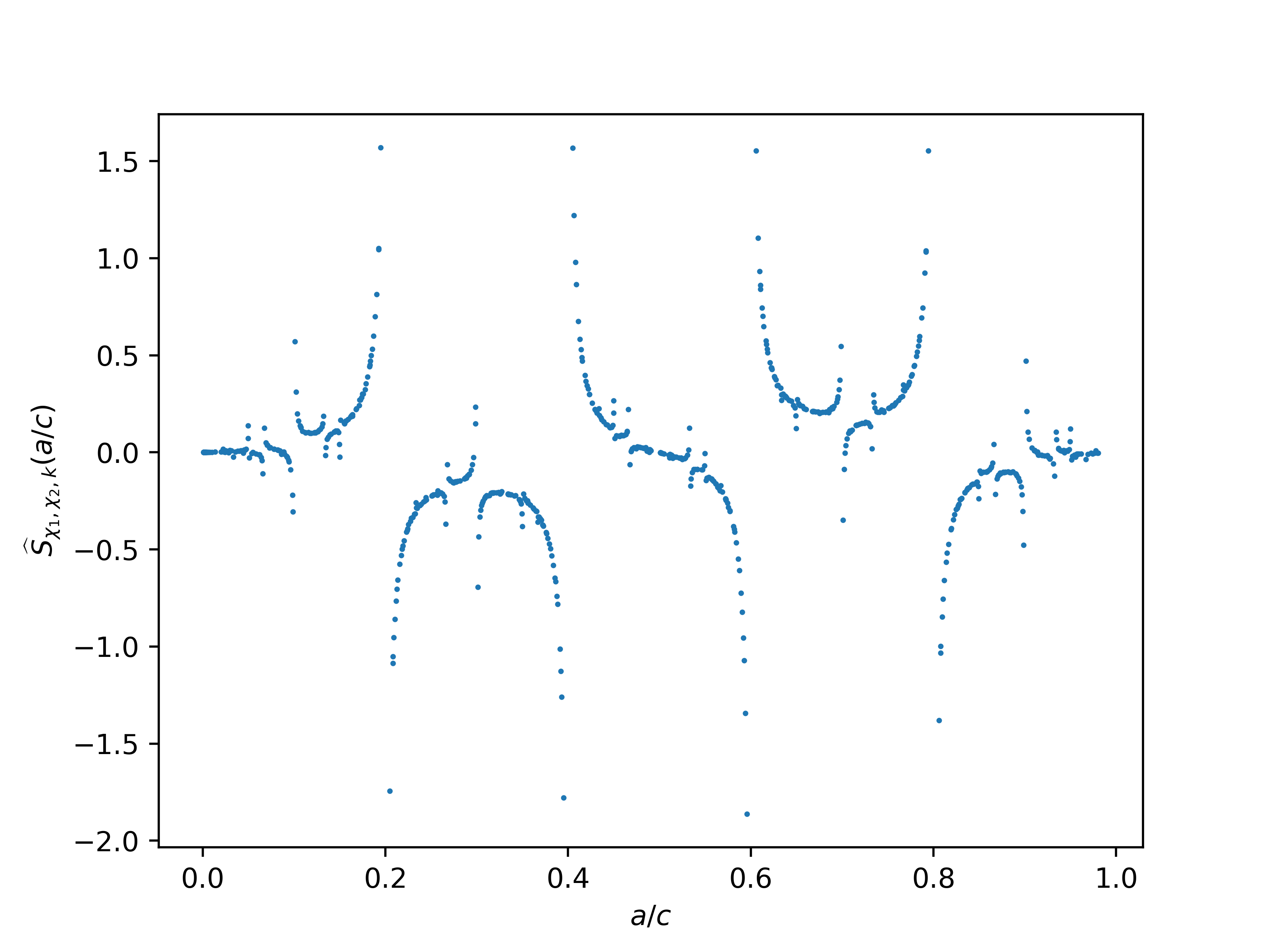}
    \caption{Scatter plot of $\widehat{S}_{\chi_1,\chi_2,k}(\mathcal{C}_{50}(q_1q_2))$}
    \label{fig1}
\end{figure}

\noindent Note that Figure \ref{fig1} is $1$-periodic in $\mathfrak{a}$ by Lemma \ref{ShatPeriodicity}.\newpage
\noindent Next, we provide a similar visualization for the associated $h_{\gamma,\chi_1,\chi_2,k}(\mathcal{C}_{15}(q_1q_2))$ for selected matrices in $\Gamma_1(q_1q_2)$, specifically $\gamma_1=\begin{psmallmatrix}
    26 & 1 \\ 25 & 1
\end{psmallmatrix}$ and $\gamma_2=\begin{psmallmatrix}
    51 & 104 \\ 25 & 51
\end{psmallmatrix}$.

\vspace*{-2mm}
\begin{figure*}[htbp]
    \subfloat[$h_{\gamma_1,\chi_1,\chi_2,k}(\mathcal{C}_{15}(q_1q_2))$ for $\gamma_1=\begin{psmallmatrix}
        26 & 1 \\ 25 & 1
    \end{psmallmatrix}$]{%
        \includegraphics[width=.47\linewidth]{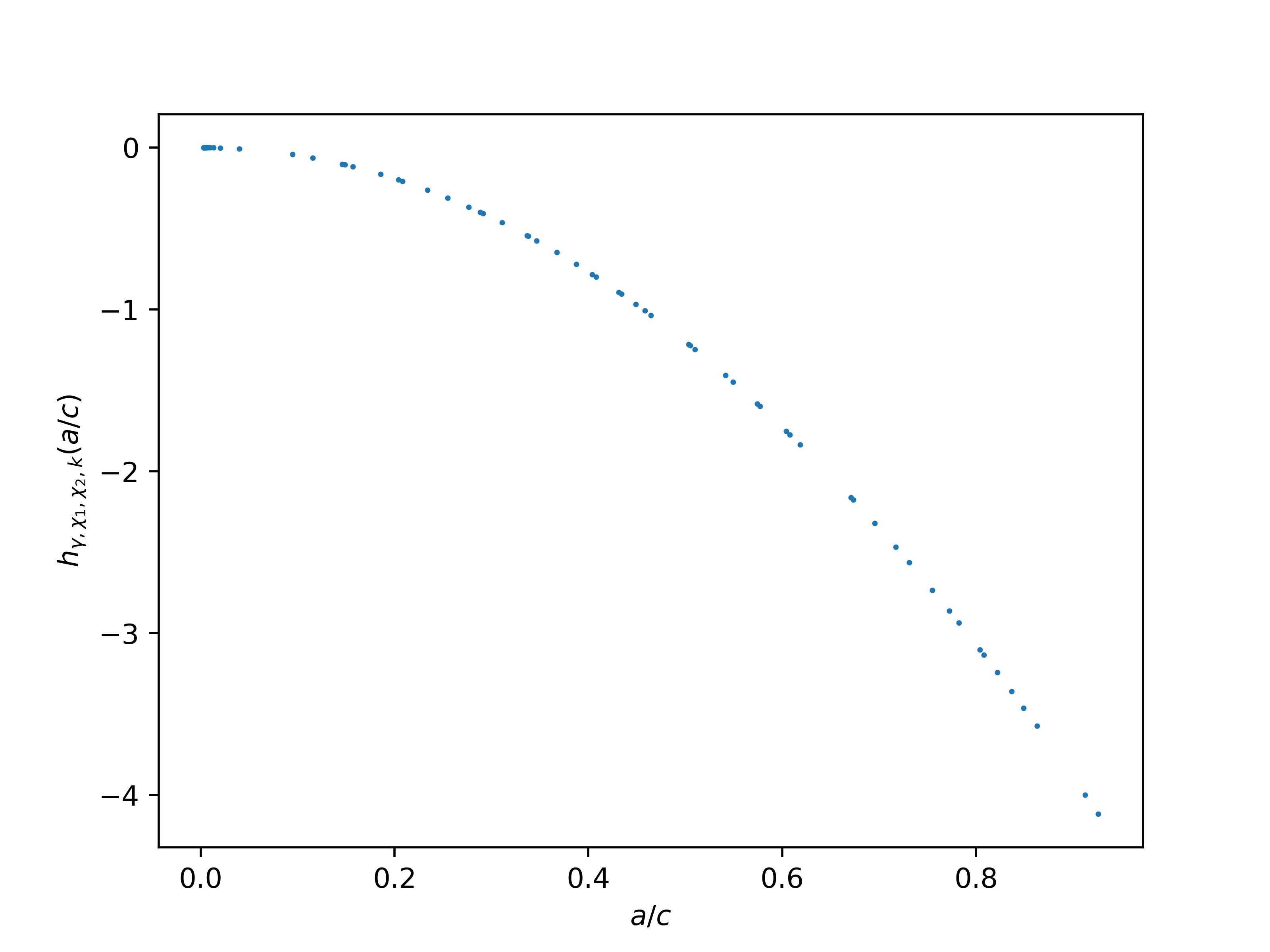}%
        \label{subfig:a}%
    }\hfill
    \subfloat[$h_{\gamma_2,\chi_1,\chi_2,k}(\mathcal{C}_{15}(q_1q_2))$ for $\gamma_2=\begin{psmallmatrix}
        51 & 104 \\ 25 & 51
    \end{psmallmatrix}$]{%
        \includegraphics[width=.47\linewidth]{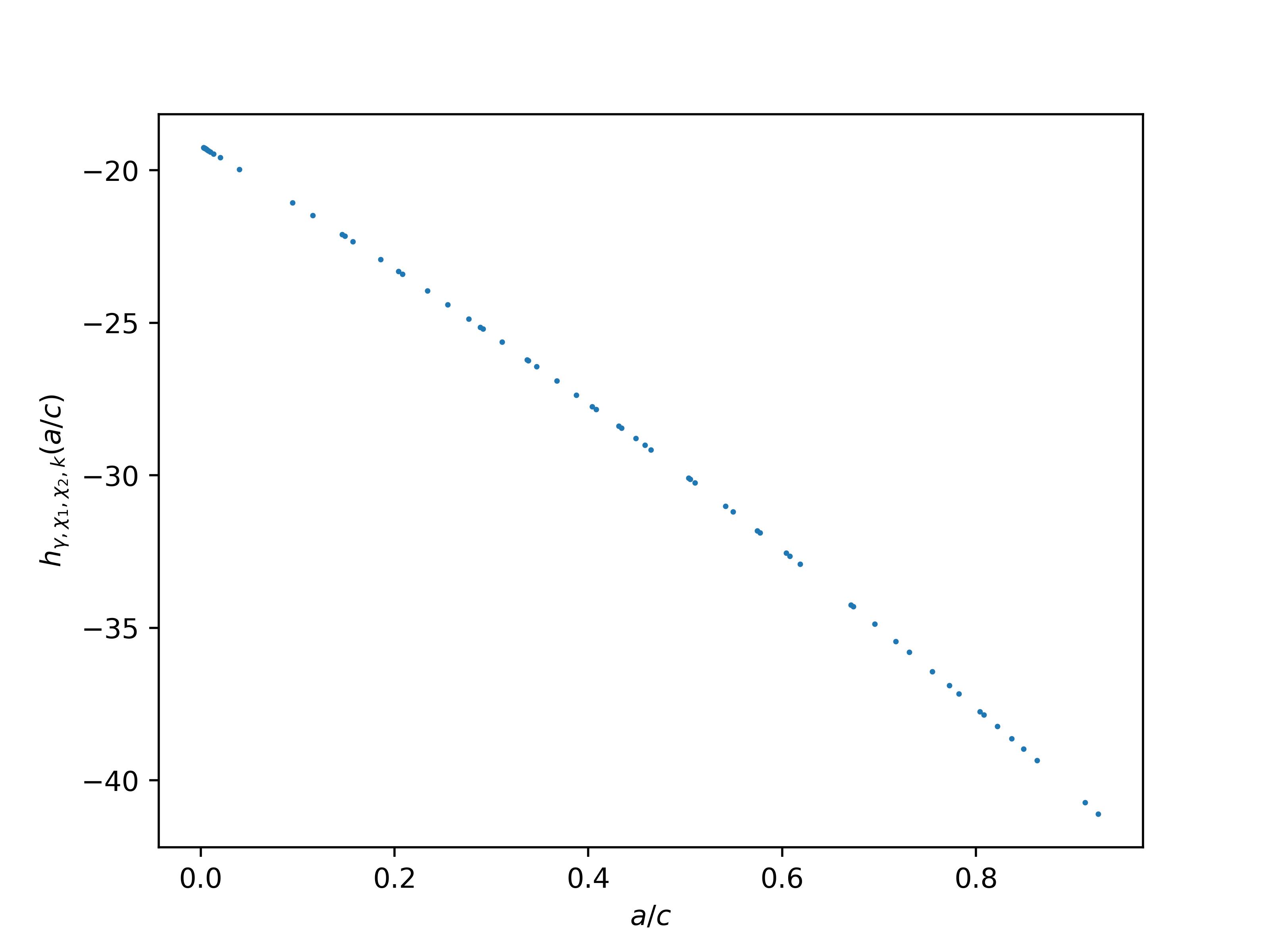}%
        \label{subfig:b}%
    }
    \caption{Scatter plots of $h_{\gamma,\chi_1,\chi_2,k}(\mathcal{C}_{15}(q_1q_2))$ for two selected values of $\gamma\in\Gamma_1(q_1q_2)$}
    \label{fig2}
\end{figure*}

\noindent Note that our $h_{\gamma_1,\chi_1,\chi_2,k}$ and $h_{\gamma_2,\chi_1,\chi_2,k}$ in Figure \ref{fig2} are not $1$-periodic, as $h_{\gamma_1,\chi_1,\chi_2,k}$ and $h_{\gamma_2,\chi_1,\chi_2,k}$ are exactly interpolated by quadratic polynomials as per Lemma \ref{qmfLemma}. More precisely, for $\mathfrak{a}\in(\Gamma_1(q_1q_2))(\infty)\setminus\{\infty\}$ we have that
$$h_{\gamma_1,\chi_1,\chi_2,k}(\mathfrak{a})=-\frac{24}{5}\mathfrak{a}^2\qquad\text{and}\qquad h_{\gamma_2,\chi_1,\chi_2,k}(\mathfrak{a})=-\frac{24}{5}\mathfrak{a}^2-\frac{96}{5}\mathfrak{a}-\frac{96}{5}.$$
Next we compute the associated $h_{\gamma,\chi_1,\chi_2,k}(\mathcal{C}_{15}(q_1q_2))$ for $\gamma=\gamma_1\gamma_2$ and $\gamma=\gamma_2\gamma_1$ to verify the crossed homomorphism relation of Lemma \ref{qmfCrossHom} holds on the above polynomials.
\vspace*{-2mm}
\begin{figure*}[htbp]
    \subfloat[$h_{\gamma_1\gamma_2,\chi_1,\chi_2,k}(\mathcal{C}_{15}(q_1q_2))$ for $\gamma_1\gamma_2=\begin{psmallmatrix}
        1351 & 2755 \\ 1300 & 2651
    \end{psmallmatrix}$]{%
        \includegraphics[width=.47\linewidth]{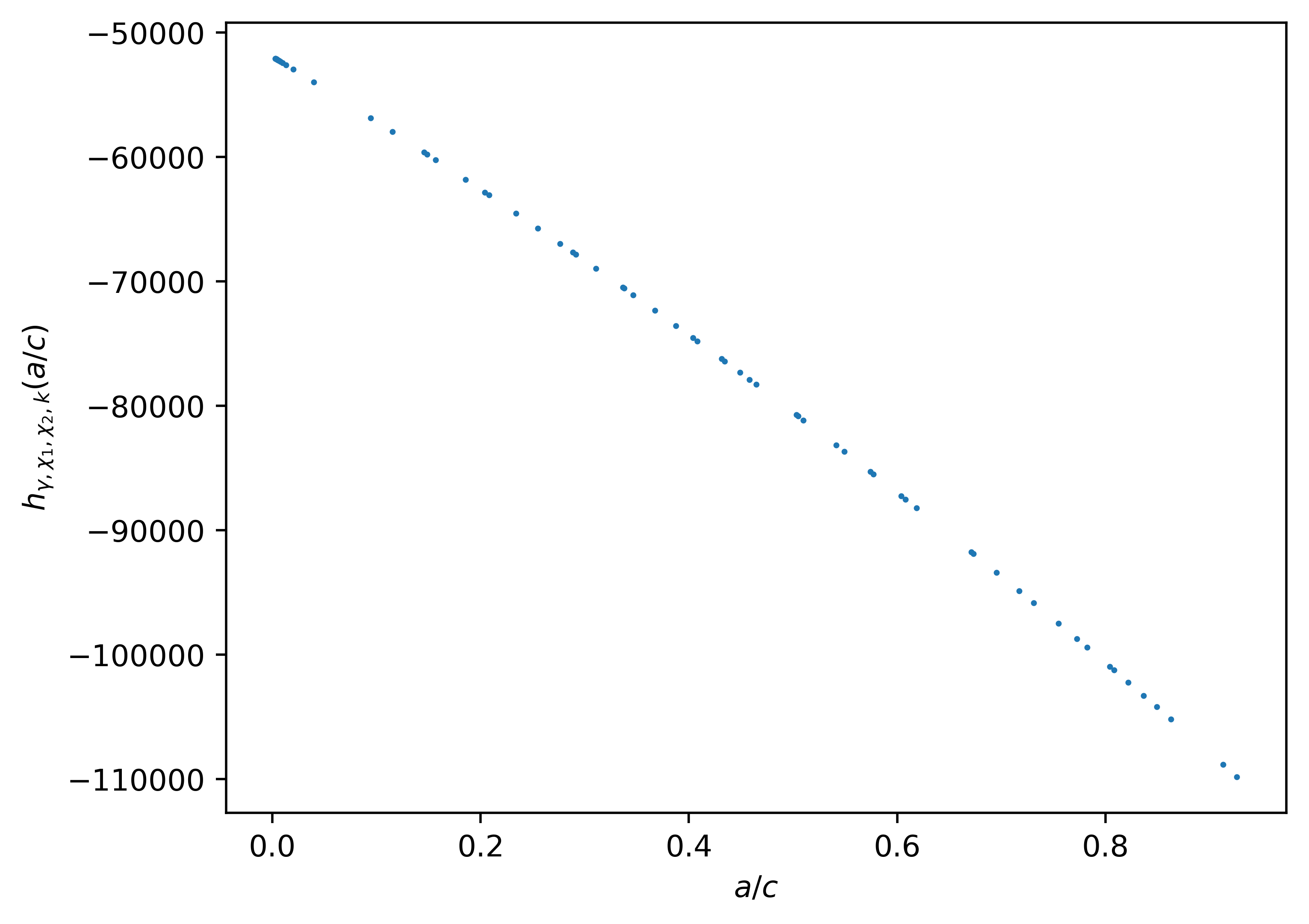}%
        \label{subfig:c}%
    }\hfill
    \subfloat[$h_{\gamma_2\gamma_1,\chi_1,\chi_2,k}(\mathcal{C}_{15}(q_1q_2))$ for $\gamma_2\gamma_1=\begin{psmallmatrix}
        3926 & 155 \\ 1925 & 76
    \end{psmallmatrix}$]{%
        \includegraphics[width=.47\linewidth]{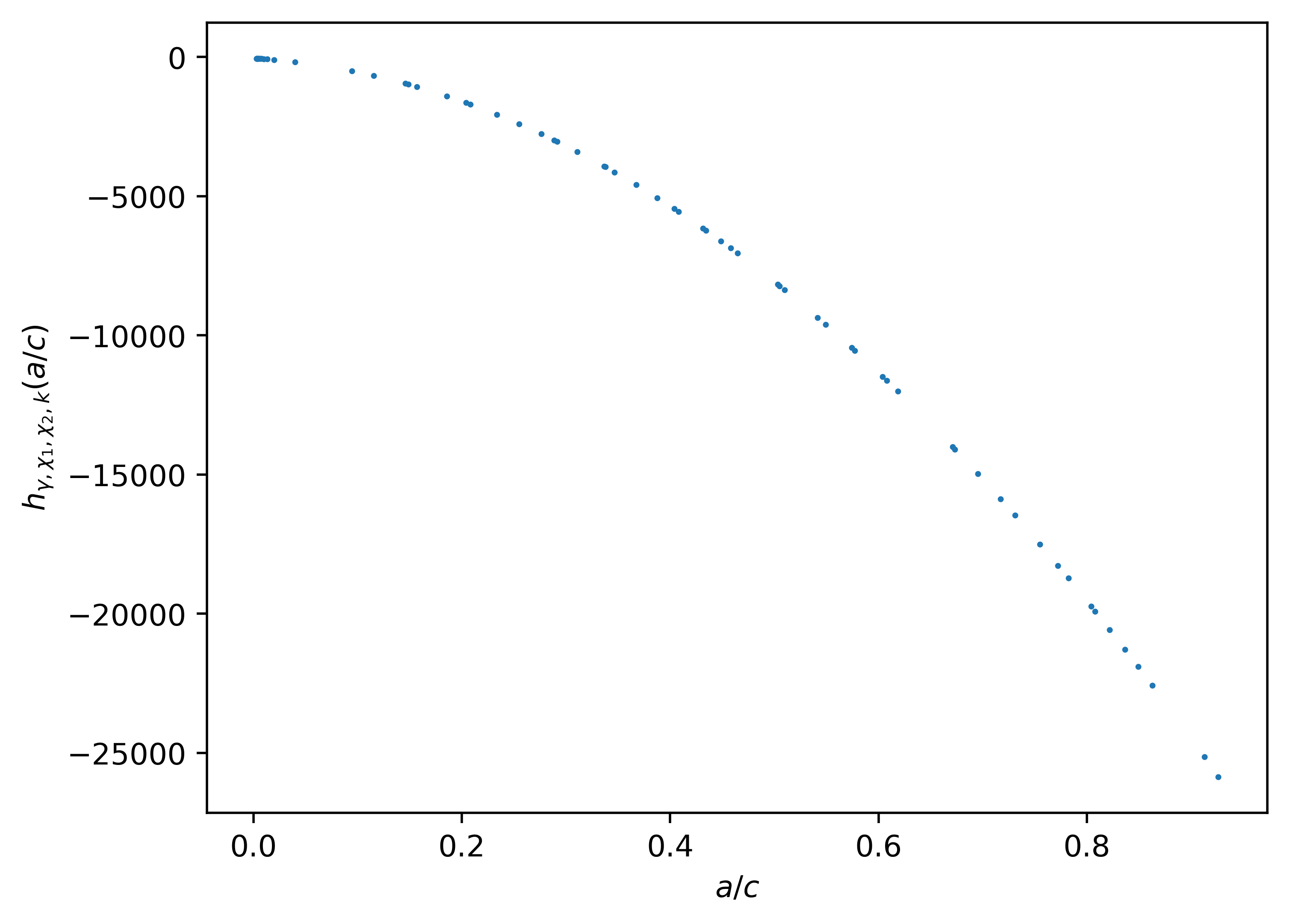}%
        \label{subfig:d}%
    }
    \caption{Scatter plots of $h_{\gamma,\chi_1,\chi_2,k}(\mathcal{C}_{15}(q_1q_2))$ for $\gamma=\gamma_1\gamma_2$ and $\gamma=\gamma_2\gamma_1$.}
    \label{fig3}
\end{figure*}

\noindent Again, note that $h_{\gamma_1\gamma_2,\chi_1,\chi_2,k}$ and $h_{\gamma_2\gamma_1,\chi_1,\chi_2,k}$ in Figure \ref{fig3} are not $1$-periodic, as $h_{\gamma_1\gamma_2,\chi_1,\chi_2,k}$ and $h_{\gamma_2\gamma_1,\chi_1,\chi_2,k}$ are exactly interpolated by quadratic polynomials as per Lemma \ref{qmfLemma}. More precisely, for $\mathfrak{a}\in(\Gamma_1(q_1q_2))(\infty)\setminus\{\infty\}$ we have that
$$h_{\gamma_1\gamma_2,\chi_1,\chi_2,k}(\mathfrak{a})=-\frac{62448}{5}\mathfrak{a}^2-\frac{254688}{5}\mathfrak{a}-51936\qquad\text{and}\qquad h_{\gamma_2\gamma_1,\chi_1,\chi_2,k}(\mathfrak{a})=-\frac{138648}{5}\mathfrak{a}^2-\frac{10944}{5}\mathfrak{a}-\frac{216}{5}.$$
Indeed, Lemma \ref{qmfCrossHom} can compute these polynomials given $h_{\gamma_1,\chi_1,\chi_2,k}$ and $h_{\gamma_2,\chi_1,\chi_2,k}$. For $h_{\gamma_1\gamma_2,\chi_1,\chi_2,k}$ note
\begin{align*}
    h_{\gamma_1,\chi_1,\chi_2,k}|_{2-k}\gamma_2(\mathfrak{a})+h_{\gamma_2,\chi_1,\chi_2,k}(\mathfrak{a})&=-\frac{24}{5}(51\mathfrak{a}+104)^2-\frac{24}{5}\mathfrak{a}^2-\frac{96}{5}\mathfrak{a}-\frac{96}{5} \\
    &=-\frac{62448}{5}\mathfrak{a}^2-\frac{254688}{5}\mathfrak{a}-51936 \\
    &=h_{\gamma_1\gamma_2,\chi_1,\chi_2,k}(\mathfrak{a}).
\end{align*}
And similarly for $h_{\gamma_2\gamma_1,\chi_1,\chi_2,k}$ note
\begin{align*}
    h_{\gamma_2,\chi_1,\chi_2,k}|_{2-k}\gamma_1(\mathfrak{a})+h_{\gamma_1,\chi_1,\chi_2,k}(\mathfrak{a})&=-\frac{24}{5}(26\mathfrak{a}+1)^2-\frac{96}{5}(26\mathfrak{a}+1)(25\mathfrak{a}+1)-\frac{96}{5}(25\mathfrak{a}+1)^2-\frac{24}{5}\mathfrak{a}^2 \\
    &=-\frac{138648}{5}\mathfrak{a}^2-\frac{10944}{5}\mathfrak{a}-\frac{216}{5} \\
    &=h_{\gamma_2\gamma_1,\chi_1,\chi_2,k}(\mathfrak{a}).
\end{align*}
This explicit example serves as a nice consistency check of the validity of Lemma \ref{qmfCrossHom}.

\section{Arithmetic Properties of $\widetilde{S}_{\chi_1,\chi_2,k}$}\label{arithmeticProps}
To highlight the arithmetic properties of the image $S_{\chi_1,\chi_2,k}(\Gamma_1(q_1q_2))$, we choose a different normalization for our Dedekind sums and define some general notation as in Definition \ref{stildeDef}. Now we define some further notation specific to the context of the following computation.
\begin{definition}
    We define the following Dirichlet characters:
    \begin{itemize}
        \item Let $\chi_3$, $\chi_4$, $\chi_5$, and $\chi_7$ be the unique quadratic primitive characters modulo $3$, $4$, $5$, and $7$ respectively.
        \item Let $\chi_{8a}$ and $\chi_{8b}$ be the unique primitive quadratic characters modulo $8$ such that $1=\chi_{8a}(7)$ and $-1=\chi_{8b}(7)$. 
        % \item Let $\chi_3$ be the unique quadratic character modulo $3$ such that $\chi_3(2)\mapsto -1$.
        % \item Let $\chi_4$ be the unique quadratic character modulo $4$ such that $\chi_4(3)\mapsto -1$.
        % \item Let $\chi_5$ be the unique quadratic character modulo $5$ such that $\chi_5(2)\mapsto -1$.
        % \item Let $\chi_7$ be the unique quadratic character modulo $7$ such that $\chi_7(3)\mapsto -1$.
        % \item Let $\chi_{8a}$ be the unique quadratic character modulo $8$ such that $\chi_{8a}(5)\mapsto -1$ and $\chi_{8a}(7)\mapsto 1$.
        % \item Let $\chi_{8b}$ be the unique quadratic character modulo $8$ such that $\chi_{8b}(5)\mapsto -1$ and $\chi_{8b}(7)\mapsto -1$.
    \end{itemize}
\end{definition}

\subsection{Computation}\label{computation}
Let $G_j(N)=\{\begin{psmallmatrix}
    a & b \\ c & d
\end{psmallmatrix}\in\Gamma_1(N):1\leq a<jN,\,N\leq c<jN\}$. For each pair of quadratic primitive characters $(\chi_1,\chi_2)$ such that $q_1q_2\leq 32$ and $\chi_1\chi_2(-1)=(-1)^k$ for $2\leq k\leq 9$, we compute the largest value of $r\in\mathbb{Q}_{\geq 0}$ such that $\widetilde{S}_{\chi_1,\chi_2,k}(G_{50}(q_1q_2))\subseteq r\mathbb{Z}$. The results are presented in three tables on the next page. Additionally, some values are not expressed in reduced form to highlight the pattern stated in Conjecture \ref{conj611}. We have the following values of $r$ for each pair of $(\chi_1,\chi_2)$ and $k$:
\begin{center}
\begin{tabular}{c|c c c c c c}
   \backslashbox{$k$}{$(\chi_1,\chi_2)$} & $(\chi_3,\chi_3)$ & $(\chi_3,\chi_4)$ & $(\chi_4,\chi_3)$ & $(\chi_4,\chi_4)$ & $(\chi_3,\chi_7)$ & $(\chi_7,\chi_3)$ \\
   \hline
   $2$ & $2$ & $2$ & $2$ & $2$ & $2$ & $2$ \\
   $4$ & $2$ & $2/3$ & $6/4$ & $6$ & $2/3$ & $6/7$ \\
   $6$ & $10/3$ & $10$ & $10/4$ & $10$ & $10$ & $10/7$ \\
   $8$ & $14$ & $14/3$ & $14/4$ & $14$ & $14/3$ & $2$
\end{tabular}
\end{center}
\begin{center}
\begin{tabular}{c|c c c c c c c}
    \backslashbox{$k$}{$(\chi_1,\chi_2)$} & $(\chi_3,\chi_{8b})$ & $(\chi_{8b},\chi_3)$ & $(\chi_5,\chi_5)$ & $(\chi_4,\chi_7)$ & $(\chi_7,\chi_4)$ & $(\chi_4,\chi_{8b})$ & $(\chi_{8b},\chi_4)$ \\
    \hline
    $2$ & $2$ & $2$ & $2$ & $2$ & $2$ & $2$ & $2$ \\
    $4$ & $2/3$ & $3$ & $6/5$ & $6/4$ & $6/7$ & $6$ & $6$ \\
    $6$ & $10/3$ & $5$ & $10$ & $10/4$ & $10/7$ & $10$ & $10$ \\
    $8$ & $14/3$ & $7$ & $14/5$ & $14/4$ & $2$ & $14$ & $14$
\end{tabular}
\end{center}
\begin{center}
\begin{tabular}{c|c c c c c c c c}
    \backslashbox{$k$}{$(\chi_1,\chi_2)$} & $(\chi_3,\chi_5)$ & $(\chi_5,\chi_3)$ & $(\chi_4,\chi_5)$ & $(\chi_5,\chi_4)$ & $(\chi_3,\chi_{8a})$ & $(\chi_{8a},\chi_3)$ & $(\chi_4,\chi_{8a})$ & $(\chi_{8a},\chi_4)$ \\
    \hline
    $3$ & $4$ & $4/5$ & $4$ & $4/5$ & $4$ & $2$ & $4$ & $4$ \\
    $5$ & $8/3$ & $8/5$ & $4$ & $8/5$ & $8/3$ & $4$ & $8$ & $8$ \\
    $7$ & $4$ & $12/5$ & $6$ & $12/5$ & $4/3$ & $6$ & $12$ & $12$ \\
    $9$ & $16$ & $16/5$ & $8$ & $16/5$ & $16$ & $8$ & $16$ & $16$
\end{tabular}
\end{center}

\subsection{Proof of Theorem \ref{containmentTheorem} and Extending the Computation from $G_{50}(q_1q_2)$ to $\Gamma_1(q_1q_2)$}\label{sec623}
Using the properties of quantum modularity developed in Section \ref{qmfSec}, we can show that the containment illustrated in the tables of Section \ref{computation} extend from $G_{50}(q_1q_2)$ to $\Gamma_1(q_1q_2)$; however, note that this does not constitute a complete proof of the first point of Conjecture \ref{conj611} as we do not demonstrate equality. Recalling Definition \ref{weirdPolynomialSpace}, we develop some preliminaries on $\mathcal{P}(k;m,q)$.
\begin{lemma}\label{lem613}
    For $\gamma_1,\gamma_2\in\Gamma_0(q_1q_2)$, if $\mathcal{P}_0\in\mathcal{P}(k;m,q_1)$ then $\mathcal{P}_0|_{2-k}\gamma_2\in\mathcal{P}(k;m,q_1)$ also.
\end{lemma}
\begin{proof}
    We first write $\mathcal{P}_0(x)=\sum_{n=0}^{k-2}a_nx^{k-n-2}$. Now let $\gamma_2=\begin{psmallmatrix}
        a & b \\ c & d
    \end{psmallmatrix}$. Thus,
    $$\mathcal{P}_0|_{2-k}\gamma_2(x)=\sum_{n=0}^{k-2}a_n(cx+d)^n(ax+b)^{k-n-2}.$$
    Now note that
    $$(cx+d)^n(ax+b)^{k-n-2}=\left(\sum_{i=0}^n\binom{n}{i}(cx)^{n-i}\,d^i\right)\left(\sum_{i=0}^{k-n-2}\binom{k-n-2}{i}(ax)^{k-n-2-i}\,b^i\right).$$
    We can re-index this as
    $$(cx+d)^n(ax+b)^{k-n-2}=\sum_{j=0}^{k-2}x^j\sum_{i=0}^j\binom{n}{i}\binom{k-n-2}{j-i}a^{k-n-2-j+i}\,b^{j-i}\,c^{n-i}\,d^i.$$
    Thus,
    $$\mathcal{P}_0|_{2-k}\gamma_2(x)=\sum_{j=0}^{k-2}x^j\left(\sum_{n=0}^{k-2}a_n\sum_{i=0}^j\binom{n}{i}\binom{k-n-2}{j-i}a^{k-n-2-j+i}\,b^{j-i}\,c^{n-i}\,d^i\right).$$
    So we need to verify that
    $$b_j=q_1^{j+1}\sum_{n=0}^{k-2}a_n\sum_{i=0}^j\binom{n}{i}\binom{k-n-2}{j-i}a^{k-n-2-j+i}\,b^{j-i}\,c^{n-i}\,d^i\in m\mathbb{Z},$$
    as this proves that $\mathcal{P}_0|_{2-k}\gamma_2\in\mathcal{P}(k;m,q_1)$. We have that
    $$b_j=\sum_{n=0}^{k-2}q_1^{n+1}a_n\sum_{i=0}^j\binom{n}{i}\binom{k-n-2}{j-i}q_1^{j-n}\,a^{k-n-2-j+i}\,b^{j-i}\,c^{n-i}\,d^i.$$
    But now note that $c^{n-i}=(q_1\,c')^{n-i}$ for some integer $c'$ since $q_1\mid c$. Thus,
    $$b_j=\sum_{n=0}^{k-2}q^{n+1}a_n\sum_{i=0}^j\binom{n}{i}\binom{k-n-2}{j-i}q_1^{j-i}\,a^{k-n-2-j+i}\,b^{j-i}\,c'^{n-i}\,d^i.$$
    Now note that $q_1^{j-i}b^{j-i}\in\mathbb{Z}$ since $i\leq j$, $\binom{k-n-2}{j-i}a^{k-n-2-j+i}\in\mathbb{Z}$, $\binom{n}{i}c'^{n-i}\in\mathbb{Z}$, and $d^i\in\mathbb{Z}$ since $0\leq i$. Additionally, $q_1^{n+1}a_n\in m\mathbb{Z}$ since $\mathcal{P}_0\in\mathcal{P}(k;m,q_1)$. Thus, $b_j\in m\mathbb{Z}$ as desired, so $\mathcal{P}_0|_{2-k}\gamma_2\in\mathcal{P}(k;m,q_1)$.
\end{proof}
\begin{lemma}\label{lem614}
    If $\gamma=\begin{psmallmatrix}
        a & b \\ c & d
    \end{psmallmatrix}\in\Gamma_1(q_1q_2)$ and $\mathcal{P}_0\in\mathcal{P}(k;m,q_1)$, then $c^{k-2}\,\mathcal{P}_0(\gamma\infty)\in m\mathbb{Z}/q_1$.
\end{lemma}
\begin{proof}
    We first write $\mathcal{P}_0(x)=\sum_{n=0}^{k-2}a_nx^{k-n-2}$. Now let $\gamma=\begin{psmallmatrix}
        a & b \\ c & d
    \end{psmallmatrix}$. Since $\gamma\in\Gamma_1(q_1q_2)$, there exists $c'\in\mathbb{Z}$ such that $c=c'q_1$. Similarly, since $\mathcal{P}_0\in\mathcal{P}(k;m,q_1)$ we have that there exists $a_n'\in m\mathbb{Z}$ such that $a_n=a_n'/q_1^{n+1}$. Thus
    $$c^{k-2}\,\mathcal{P}_0(\gamma\infty)=c^{k-2}\,\sum_{n=0}^{k-2}a_n\left(\frac{a}{c}\right)^{k-n-2}=\sum_{n=0}^{k-2}a_na^{k-n-2}c^n=\frac{1}{q_1}\sum_{n=0}^{k-2}a_n'a^{k-n-2}c'^n.$$
    Since $a_n'\in m\mathbb{Z}$ and $a,c'\in\mathbb{Z}$, the desired result follows immediately.
\end{proof}
Now we can put all of this together to prove Theorem \ref{containmentTheorem}, a sufficient condition for $\widetilde{S}_{\chi_1,\chi_2,k}(\Gamma_1(q_1q_2))\subseteq m\mathbb{Z}/q_1$.
\begin{proof}
    First note that for $\gamma^{-1}=\begin{psmallmatrix}
        a & b \\ c & d
    \end{psmallmatrix}\in\Gamma_1(q_1q_2)$, by Lemma \ref{qmfLemma} we have that
    \begin{equation}\label{bruhMoment}\widetilde{S}_{\chi_1,\chi_2,k}(\gamma^{-1})=c^{k-2}\,S_{\chi_1,\chi_2,k}(\gamma^{-1})=c^{k-2}\,h_{\gamma,\chi_1,\chi_2,k}(\gamma^{-1}\infty).\end{equation}
    Note if $h_{\gamma,\chi_1,\chi_2,k}\in\mathcal{P}(k;m,q_1)$ for all $\gamma\in G$, then by Lemma \ref{qmfCrossHom} and Lemma \ref{lem613} (with the fact that $\mathcal{P}(k;m,q_1)$ is closed under addition) we have that $h_{\gamma,\chi_1,\chi_2,k}\in\mathcal{P}(k;m,q_1)$ for all $\gamma\in\Gamma_1(q_1q_2)$. So by (\ref{bruhMoment}) and Lemma \ref{lem614} we have that $\widetilde{S}_{\chi_1,\chi_2,k}(\gamma^{-1})\in m\mathbb{Z}/q_1$ for all $\gamma^{-1}\in\Gamma_1(q_1q_2)$. Since $\gamma\mapsto\gamma^{-1}$ is a bijective map from $\Gamma_1(q_1q_2)$ to itself, it follows that $\widetilde{S}_{\chi_1,\chi_2,k}(\gamma)\in m\mathbb{Z}/q_1$ for all $\gamma\in\Gamma_1(q_1q_2)$ as desired.
\end{proof}

Now we utilize this theorem to prove the containment illustrated in the tables of Section \ref{computation} extends from $G_{50}(q_1q_2)$ to $\Gamma_1(q_1q_2)$. Although we are not able to compute $h_{\gamma,\chi_1,\chi_2,k}$ directly, by Lemma \ref{qmfLemma} we know that $h_{\gamma,\chi_1,\chi_2,k}$ is a degree $k-2$ polynomial; so, by computing $h_{\gamma,\chi_1,\chi_2,k}(\mathfrak{a})$ for $k-1$ distinct values of $\mathfrak{a}$ we can recover $h_{\gamma,\chi_1,\chi_2,k}$ via Lagrange interpolation. Using the finite sum formula of Theorem \ref{mainThm} allows us to make light work of this task with the help of a computer.

As an explicit example of this procedure, we compute these $h_{\gamma,\chi_1,\chi_2,k}$ with $k=4$ and $\chi_1=\chi_2$ being the unique quadratic primitive character modulo $5$. In the table below, we present $h_{\gamma,\chi_1,\chi_2,k}$ for all $\gamma\in G$ where $G$ is a finite generating set of $\Gamma_1(q_1q_2)$ obtained using \code{Gamma1(q1q2).gens() }in \code{SageMath }where \code{q1q2 }is a placeholder variable for the product $q_1q_2$.

\begin{center}
\renewcommand{\arraystretch}{2.1}
\begin{tabular}{c c c}
\begin{tabular}{c|c}
    $\gamma\in G$ & $h_{\gamma,\chi_1,\chi_2,k}(\mathfrak{a})$ \\
    \hline
    $\begin{psmallmatrix}
        1 & 1 \\ 0 & 1
    \end{psmallmatrix}$ & $0$ \\
    $\begin{psmallmatrix}
        51 & -4 \\ 625 & -49
    \end{psmallmatrix}$ & $-5340\mathfrak{a}^2+\dfrac{4176}{5}\mathfrak{a}-\dfrac{816}{25}$ \\
    $\begin{psmallmatrix}
        -74 & 7 \\ -275 & 26
    \end{psmallmatrix}$ & $\dfrac{50244}{5}\mathfrak{a}^2-\dfrac{9576}{5}\mathfrak{a}+\dfrac{456}{5}$ \\
    $\begin{psmallmatrix}
        -149 & 16 \\ -475 & 51
    \end{psmallmatrix}$ & $\dfrac{250206}{5}\mathfrak{a}^2-10752\mathfrak{a}+\dfrac{14442}{25}$ \\
    $\begin{psmallmatrix}
        76 & -9 \\ 625 & -74
    \end{psmallmatrix}$ & $-39240\mathfrak{a}^2+\dfrac{46464}{5}\mathfrak{a}-\dfrac{13752}{25}$ \\
    $\begin{psmallmatrix}
        51 & -7 \\ 175 & -24
    \end{psmallmatrix}$ & $\dfrac{17262}{5}\mathfrak{a}^2-\dfrac{4692}{5}\mathfrak{a}+\dfrac{1596}{25}$ \\
    $\begin{psmallmatrix}
        -274 & 43 \\ -325 & 51
    \end{psmallmatrix}$ & $\dfrac{86262}{5}\mathfrak{a}^2-\dfrac{27036}{5}\mathfrak{a}+\dfrac{10596}{25}$ \\
    $\begin{psmallmatrix}
        201 & -35 \\ 425 & -74
    \end{psmallmatrix}$ & $-\dfrac{21738}{5}\mathfrak{a}^2+\dfrac{7632}{5}\mathfrak{a}-\dfrac{3342}{25}$ \\
    $\begin{psmallmatrix}
        26 & -5 \\ 125 & -24
    \end{psmallmatrix}$ & $16956\mathfrak{a}^2-\dfrac{31908}{5}\mathfrak{a}+\dfrac{3006}{5}$ \\
    $\begin{psmallmatrix}
        101 & -23 \\ 325 & -74
    \end{psmallmatrix}$ & $\dfrac{116838}{5}\mathfrak{a}^2-\dfrac{53136}{5}\mathfrak{a}+\dfrac{30198}{25}$ \\
    $\begin{psmallmatrix}
        226 & -59 \\ 475 & -124
    \end{psmallmatrix}$ & $-\dfrac{34056}{5}\mathfrak{a}^2+\dfrac{17772}{5}\mathfrak{a}-\dfrac{11598}{25}$ \\
    $\begin{psmallmatrix}
        -324 & 77 \\ -425 & 101
    \end{psmallmatrix}$ & $-\dfrac{246912}{5}\mathfrak{a}^2+\dfrac{117408}{5}\mathfrak{a}-\dfrac{69792}{25}$ \\
    $\begin{psmallmatrix}
        176 & -49 \\ 625 & -174
    \end{psmallmatrix}$ & $56910\mathfrak{a}^2-\dfrac{158424}{5}\mathfrak{a}+4410$ \\
    $\begin{psmallmatrix}
        -324 & 103 \\ -475 & 151
    \end{psmallmatrix}$ & $-\dfrac{235464}{5}\mathfrak{a}^2+\dfrac{149712}{5}\mathfrak{a}-\dfrac{118992}{25}$ \\
    $\begin{psmallmatrix}
        -49 & 17 \\ -75 & 26
    \end{psmallmatrix}$ & $-\dfrac{6888}{5}\mathfrak{a}^2+\dfrac{4752}{5}\mathfrak{a}-\dfrac{4104}{25}$ \\
    $\begin{psmallmatrix}
        51 & -20 \\ 125 & -49
    \end{psmallmatrix}$ & $-15294\mathfrak{a}^2+\dfrac{59304}{5}\mathfrak{a}-\dfrac{11502}{5}$ \\
    $\begin{psmallmatrix}
        276 & -121 \\ 625 & -274
    \end{psmallmatrix}$ & $-56190\mathfrak{a}^2+\dfrac{246288}{5}\mathfrak{a}-\dfrac{269874}{25}$ \\
    $\begin{psmallmatrix}
        301 & -144 \\ 625 & -299
    \end{psmallmatrix}$ & $-9090\mathfrak{a}^2+\dfrac{43476}{5}\mathfrak{a}-\dfrac{51984}{25}$
\end{tabular} & \qquad\qquad & \begin{tabular}{c|c}
    $\gamma\in G$ & $h_{\gamma,\chi_1,\chi_2,k}(\mathfrak{a})$ \\
    \hline
    $\begin{psmallmatrix}
        -24 & 1 \\ -25 & 1
    \end{psmallmatrix}$ & $\dfrac{24}{5}\mathfrak{a}^2$ \\
    $\begin{psmallmatrix}
        126 & -11 \\ 275 & -24
    \end{psmallmatrix}$ & $-\dfrac{33444}{5}\mathfrak{a}^2+1164\mathfrak{a}-\dfrac{1266}{25}$ \\
    $\begin{psmallmatrix}
        126 & -13 \\ 475 & -49
    \end{psmallmatrix}$ & $\dfrac{287874}{5}\mathfrak{a}^2-11928\mathfrak{a}+618$ \\
    $\begin{psmallmatrix}
        -199 & 23 \\ -225 & 26
    \end{psmallmatrix}$ & $\dfrac{26136}{5}\mathfrak{a}^2-\dfrac{6048}{5}\mathfrak{a}+\dfrac{1752}{25}$ \\
    $\begin{psmallmatrix}
        176 & -23 \\ 375 & -49
    \end{psmallmatrix}$ & $-1620\mathfrak{a}^2+\dfrac{2184}{5}\mathfrak{a}-\dfrac{732}{25}$ \\
    $\begin{psmallmatrix}
        -74 & 11 \\ -175 & 26
    \end{psmallmatrix}$ & $-\dfrac{46332}{5}\mathfrak{a}^2+\dfrac{13704}{5}\mathfrak{a}-\dfrac{5064}{25}$ \\
    $\begin{psmallmatrix}
        101 & -16 \\ 625 & -99
    \end{psmallmatrix}$ & $-71190\mathfrak{a}^2+\dfrac{112812}{5}\mathfrak{a}-\dfrac{44688}{25}$ \\
    $\begin{psmallmatrix}
        101 & -18 \\ 275 & -49
    \end{psmallmatrix}$ & $\dfrac{125706}{5}\mathfrak{a}^2-\dfrac{44748}{5}\mathfrak{a}+\dfrac{3984}{5}$ \\
    $\begin{psmallmatrix}
        101 & -22 \\ 225 & -49
    \end{psmallmatrix}$ & $-\dfrac{35556}{5}\mathfrak{a}^2+3120\mathfrak{a}-\dfrac{8568}{25}$ \\
    $\begin{psmallmatrix}
        -124 & 29 \\ -325 & 76
    \end{psmallmatrix}$ & $\dfrac{271482}{5}\mathfrak{a}^2-\dfrac{127056}{5}\mathfrak{a}+\dfrac{74322}{25}$ \\
    $\begin{psmallmatrix}
        151 & -36 \\ 625 & -149
    \end{psmallmatrix}$ & $99060\mathfrak{a}^2-\dfrac{236112}{5}\mathfrak{a}+\dfrac{140688}{25}$ \\
    $\begin{psmallmatrix}
        -199 & 55 \\ -275 & 76
    \end{psmallmatrix}$ & $-\dfrac{56826}{5}\mathfrak{a}^2+\dfrac{31416}{5}\mathfrak{a}-\dfrac{21714}{25}$ \\
    $\begin{psmallmatrix}
        151 & -46 \\ 325 & -99
    \end{psmallmatrix}$ & $-\dfrac{58182}{5}\mathfrak{a}^2+\dfrac{35556}{5}\mathfrak{a}-\dfrac{27168}{25}$ \\
    $\begin{psmallmatrix}
        201 & -64 \\ 625 & -199
    \end{psmallmatrix}$ & $80760\mathfrak{a}^2-\dfrac{257136}{5}\mathfrak{a}+\dfrac{204672}{25}$ \\
    $\begin{psmallmatrix}
        226 & -81 \\ 625 & -224
    \end{psmallmatrix}$ & $111060\mathfrak{a}^2-\dfrac{398088}{5}\mathfrak{a}+\dfrac{356724}{25}$ \\
    $\begin{psmallmatrix}
        -99 & 43 \\ -175 & 76
    \end{psmallmatrix}$ & $\dfrac{29238}{5}\mathfrak{a}^2-\dfrac{25464}{5}\mathfrak{a}+\dfrac{27726}{25}$ \\
    $\begin{psmallmatrix}
        -274 & 131 \\ -525 & 251
    \end{psmallmatrix}$ & $\dfrac{37074}{5}\mathfrak{a}^2-7092\mathfrak{a}+\dfrac{42396}{25}$ \\
    $\begin{psmallmatrix}
        351 & -184 \\ 475 & -249
    \end{psmallmatrix}$ & $-\dfrac{234006}{5}\mathfrak{a}^2+\dfrac{245328}{5}\mathfrak{a}-\dfrac{321498}{25}$
\end{tabular}
\end{tabular}
\end{center}

\begin{center}
\renewcommand{\arraystretch}{2.1}
\begin{tabular}{c c c}
\begin{tabular}{c|c}
    $\gamma\in G$ & $h_{\gamma,\chi_1,\chi_2,k}(\mathfrak{a})$ \\
    \hline
    $\begin{psmallmatrix}
        351 & -185 \\ 425 & -224
    \end{psmallmatrix}$ & $\dfrac{259362}{5}\mathfrak{a}^2-\dfrac{273336}{5}\mathfrak{a}+\dfrac{360078}{25}$ \\
    $\begin{psmallmatrix}
        226 & -121 \\ 325 & -174
    \end{psmallmatrix}$ & $-\dfrac{151482}{5}\mathfrak{a}^2+\dfrac{162312}{5}\mathfrak{a}-\dfrac{217398}{25}$ \\
    $\begin{psmallmatrix}
        176 & -97 \\ 225 & -124
    \end{psmallmatrix}$ & $-\dfrac{130956}{5}\mathfrak{a}^2+\dfrac{144456}{5}\mathfrak{a}-\dfrac{199188}{25}$ \\
    $\begin{psmallmatrix}
        76 & -45 \\ 125 & -74
    \end{psmallmatrix}$ & $-17244\mathfrak{a}^2+\dfrac{101436}{5}\mathfrak{a}-\dfrac{29838}{5}$ \\
    $\begin{psmallmatrix}
        226 & -143 \\ 275 & -174
    \end{psmallmatrix}$ & $\dfrac{130806}{5}\mathfrak{a}^2-33096\mathfrak{a}+\dfrac{52338}{5}$ \\
    $\begin{psmallmatrix}
        251 & -173 \\ 325 & -224
    \end{psmallmatrix}$ & $-\dfrac{183162}{5}\mathfrak{a}^2+\dfrac{252552}{5}\mathfrak{a}-\dfrac{435282}{25}$ \\
    $\begin{psmallmatrix}
        -249 & 182 \\ -275 & 201
    \end{psmallmatrix}$ & $-\dfrac{18456}{5}\mathfrak{a}^2+\dfrac{26904}{5}\mathfrak{a}-\dfrac{9804}{5}$ \\
    $\begin{psmallmatrix}
        101 & -80 \\ 125 & -99
    \end{psmallmatrix}$ & $15606\mathfrak{a}^2-\dfrac{122952}{5}\mathfrak{a}+\dfrac{48438}{5}$
\end{tabular} & \qquad\qquad &
\begin{tabular}{c|c}
    $\gamma\in G$ & $h_{\gamma,\chi_1,\chi_2,k}(\mathfrak{a})$ \\
    \hline
    $\begin{psmallmatrix}
        326 & -173 \\ 375 & -199
    \end{psmallmatrix}$ & $19080\mathfrak{a}^2-\dfrac{101184}{5}\mathfrak{a}+\dfrac{134148}{25}$ \\
    $\begin{psmallmatrix}
        251 & -136 \\ 275 & -149
    \end{psmallmatrix}$ & $\dfrac{6}{5}\mathfrak{a}^2-\dfrac{24}{5}\mathfrak{a}+\dfrac{54}{25}$ \\
    $\begin{psmallmatrix}
        -149 & 86 \\ -175 & 101
    \end{psmallmatrix}$ & $\dfrac{15768}{5}\mathfrak{a}^2-\dfrac{18264}{5}\mathfrak{a}+\dfrac{26436}{25}$ \\
    $\begin{psmallmatrix}
        -249 & 154 \\ -325 & 201
    \end{psmallmatrix}$ & $-\dfrac{150918}{5}\mathfrak{a}^2+\dfrac{186588}{5}\mathfrak{a}-\dfrac{288348}{25}$ \\
    $\begin{psmallmatrix}
        -424 & 291 \\ -475 & 326
    \end{psmallmatrix}$ & $\dfrac{145056}{5}\mathfrak{a}^2-39828\mathfrak{a}+\dfrac{341742}{25}$ \\
    $\begin{psmallmatrix}
        151 & -107 \\ 175 & -124
    \end{psmallmatrix}$ & $\dfrac{12012}{5}\mathfrak{a}^2-3396\mathfrak{a}+\dfrac{30006}{25}$ \\
    $\begin{psmallmatrix}
        426 & -313 \\ 475 & -349
    \end{psmallmatrix}$ & $\dfrac{245724}{5}\mathfrak{a}^2-\dfrac{361332}{5}\mathfrak{a}+\dfrac{132834}{5}$ \\
    --- & ---
\end{tabular}
\end{tabular}

% \begin{tabular}{c c c}
% \begin{tabular}{c|c}
%     $\gamma\in G$ & $h_{\gamma,\chi_1,\chi_2,k}(\mathfrak{a})$ \\
%     \hline
    
% \end{tabular} & \qquad\qquad &
% \begin{tabular}{c|c}
%     $\gamma\in G$ & $h_{\gamma,\chi_1,\chi_2,k}(\mathfrak{a})$ \\
%     \hline
    
% \end{tabular}
% \end{tabular}
\end{center}

\noindent Noting that every polynomial in the table is a member of $\mathcal{P}(4;6,5)$, we have that $\widetilde{S}_{\chi_1,\chi_2,4}(\Gamma_1(q_1q_2))\subseteq 6\mathbb{Z}/5$  by Theorem \ref{containmentTheorem}. Using this technique we are able to prove the desired result for all pairs of characters listed in the tables of Section \ref{computation}. To verify this result for every entry in the table, the corresponding polynomials can be found in the Github in Section \ref{codeSec} at the end of this paper.

\section{Reciprocity Identities from the Fricke Involution}\label{frickeSec}
Recall the definition of $\omega$ in Definition \ref{frickeDef}. We note the action of $\omega$ on $E_{\chi_1,\chi_2,k}$ as we will use this to extend the domains of $\widehat{S}_{\chi_1,\chi_2,k}$, $h_{\gamma,\chi_1,\chi_2,k}$, and $\phi_{\chi_1,\chi_2,k}$.
\begin{lemma}\label{weisinger}[\cite{Weisinger}, Proposition $1$]
    The weight $k$ slash operator of Fricke applied to $E_{\chi_1,\chi_2,k}$ obeys the rule
    $$E_{\chi_1,\chi_2,k}|_k\omega(z)=R_{\chi_1,\chi_2,k}\,E_{\chi_2,\chi_1,k}(z)$$
    where $R_{\chi_1,\chi_2,k}=\chi_1(-1)(\tau(\chi_1)/\tau(\chi_2))(q_2/q_1)^{k/2}$.
\end{lemma}
\begin{remark}
    Weisinger defines his $E_{\chi_1,\chi_2,k}$ differently, so care is required in going between his formula and how it is written here. Specifically, he uses the notation $F_{\chi_1,\chi_2,k}$ and this relates to our $E_{\chi_1,\chi_2,k}$ as $E_{\chi_1,\chi_2,k}=2F_{\chi_1,\overline{\chi_2},k}$.
\end{remark}
% \begin{proof}
%     By definition we have that
%     $$E_{\chi_1,\chi_2,k}(z)=\frac{(q_2/(2\pi))^kL(k,\chi_1\chi_2)(k-1)!}{2i^{-k}\tau(\chi_2)}\sum_{(c,d)=1}\frac{\chi_1(c)\chi_2(d)}{(cq_2z+d)^k}.$$
%     Now note that
%     $$E_{\chi_1,\chi_2,k}(\omega z)=\frac{(q_2/(2\pi))^kL(k,\chi_1\chi_2)(k-1)!}{2i^{-k}\tau(\chi_2)}\sum_{(c,d)=1}\frac{\chi_1(c)\chi_2(d)}{(d-c/(q_1z))^k}.$$
%     Multiplying through the top and bottom by $(q_1z)^k$ we have
%     $$E_{\chi_1,\chi_2,k}(\omega z)=(q_2z)^k\,\frac{\tau(\chi_1)}{\tau(\chi_2)}\left(\frac{(q_1/(2\pi))^kL(k,\chi_1\chi_2)(k-1)!}{2i^{-k}\tau(\chi_1)}\right)\sum_{(c,d)=1}\frac{\chi_2(d)\chi_1(c)}{(dq_1z-c)^k}.$$
%     Multiplying by $1=\chi_1(-1)\chi_1(-1)$ and re-indexing the sum we have that
%     \begin{align*}
%         E_{\chi_1,\chi_2,k}(\omega z)&=\chi_1(-1)(q_2z)^k\,\ frac{\tau(\chi_1)}{\tau(\chi_2)}\left(\frac{(q_1/(2\pi))^kL(k,\chi_1\chi_2)(k-1)!}{2i^{-k}\tau(\chi_1)}\right)\sum_{(c,d)=1}\frac{\chi_2(d)\chi_1(c)}{(dq_1z+c)^k} \\
%         &=\chi_1(-1)(q_2z)^k\,\frac{\tau(\chi_1)}{\tau(\chi_2)}\,E_{\chi_2,\chi_1,k}(z).
%     \end{align*}
%     Thus we have that
%     $$E_{\chi_1,\chi_2,k}|_k\omega(z)=(z\sqrt{q_1q_2})^{-k}E_{\chi_1,\chi_2,k}(\omega z)=\chi_1(-1)\,\frac{\tau(\chi_1)}{\tau(\chi_2)}\left(\frac{q_2}{q_1}\right)^{k/2}E_{\chi_2,\chi_1,k}(z).$$
% \end{proof}

Towards the goal of constructing some reciprocity properties of our own, we define extensions of $\phi_{\chi_1,\chi_2,k}(\omega\gamma,X,Y)$, $\widehat{S}_{\chi_1,\chi_2,k}$, and $h_{\gamma,\chi_1,\chi_2,k}$ which are compatible with $\omega$ on the domain $(\Gamma_0(q_1q_2))(\infty)\cup(\omega\Gamma_0(q_1q_2))(\infty)$. To this end, we choose the following natural definitions.
\begin{definition}
    Let
    $$\phi_{\chi_1,\chi_2,k}(\omega\gamma,X,Y)=\int_\infty^{\omega\gamma\infty}E_{\chi_1,\chi_2,k}(z)\,P_{k-2}(z;X,Y)\,dz$$
    where $\gamma\in\Gamma_0(q_1q_2)$. Now we define
    $$\widehat{S}_{\chi_1,\chi_2,k}(\mathfrak{a})=(-1)^k\,\tau(\overline{\chi_1})(k-1)\int_\infty^{\mathfrak{a}}E_{\chi_1,\chi_2,k}(z)\,P_{k-2}(z;1,-\mathfrak{a})\,dz$$
    where $\mathfrak{a}\in(\Gamma_0(q_1q_2))(\infty)\cup(\omega\Gamma_0(q_1q_2))(\infty)$. As before, for $\gamma\in\Gamma_0(q_1q_2)$ we define
    $$h_{\gamma,\chi_1,\chi_2,k}(\mathfrak{a})=\widehat{S}_{\chi_1,\chi_2,k}(\mathfrak{a})-\widehat{S}_{\chi_1,\chi_2,k}|_{2-k}\gamma(\mathfrak{a})\qquad\text{and}\qquad h_{\omega\gamma,\chi_1,\chi_2,k}(\mathfrak{a})=\widehat{S}_{\chi_1,\chi_2,k}(\mathfrak{a})-\widehat{S}_{\chi_1,\chi_2,k}|_{2-k}\omega\gamma(\mathfrak{a})$$
    for $\mathfrak{a}\in((\Gamma_0(q_1q_2))(\infty)\cup(\omega\Gamma_0(q_1q_2))(\infty))\setminus\{\infty\}$.
\end{definition}
\noindent Note that $\phi_{\chi_1,\chi_2,k}(\omega\gamma,X,Y)$ converges for all finite $X$ and $Y$. This is because the endpoints $\infty$ and $\omega\gamma\infty$ are $\Gamma_0(q_1q_2)$ equivalent to $\infty$ and $0$ respectively, and $E_{\chi_1,\chi_2,k}$ has exponential decay at both of these points. We immediately get that $\widehat{S}_{\chi_1,\chi_2,k}(\mathfrak{a})$ converges for all $\mathfrak{a}\in((\Gamma_0(q_1q_2))(\infty)\cup(\omega\Gamma_0(q_1q_2))(\infty))\setminus\{\infty\}$; however, we already know that $\widehat{S}_{\chi_1,\chi_2,k}(\infty)=0$. So $\widehat{S}_{\chi_1,\chi_2,k}$ converges on all of $(\Gamma_0(q_1q_2))(\infty)\cup(\omega\Gamma_0(q_1q_2))(\infty)$ and it follows that $h_{\gamma,\chi_1,\chi_2,k}$ and $h_{\omega\gamma,\chi_1,\chi_2,k}$ converge on all of $((\Gamma_0(q_1q_2))(\infty)\cup(\omega\Gamma_0(q_1q_2))(\infty))\setminus\{\infty\}$. 

We now explicitly compute the value of $\widehat{S}_{\chi_1,\chi_2,k}(0)$, as it relates to special values of $L$-functions and will prove useful in developing our reciprocity identities.
\begin{lemma}\label{Sfricke}
    We have the equality
    $$\widehat{S}_{\chi_1,\chi_2,k}(0)=\frac{\tau(\overline{\chi_1})(k-1)!}{\pi i\,(2\pi i)^{k-2}}\,L(k-1,\chi_1)\,L(0,\overline{\chi_2}).$$
\end{lemma}
\begin{proof}
    We do a similar analysis as in Section \ref{analysisSection}. We deduce an analogue of Lemma \ref{outerSumLimit}. For $k\geq 3$ and $0\leq n<k-2$
    $$\abs{\sum_{1\leq A}\frac{\chi_1(A)}{A^{k-1}}\sum_{1\leq B}\overline{\chi_2}(B)B^{k-n-2}e^{-2\pi ABu}}\ll_{k,\,n,\,q_2}\frac{1}{\sqrt{u}}.$$
    This, of course, allows us to deduce an analogue of Corollary \ref{outerSumLimitCor}. With the same restrictions we have
    $$\lim_{u\to 0^+}u^{k-n-2}\sum_{1\leq A}\frac{\chi_1(A)}{A^{k-1}}\sum_{1\leq B}\overline{\chi_2}(B)B^{k-n-2}e^{-2\pi ABu}=0.$$
    Similarly, we can prove an analogue of Lemma \ref{limitInterchange}
    $$\lim_{u\to 0^+}\sum_{1\leq A}\frac{\chi_1(A)}{A^{k-1}}\sum_{1\leq B}\overline{\chi_2}(B)\,e^{-2\pi ABu}=-\sum_{1\leq A}\sum_{j\bmod q_2}\frac{\chi_1(A)\overline{\chi_2}(j)}{A^{k-1}}\,B_1\left(\frac{j}{q_2}\right)=-L(k-1,\chi_1)\,B_{1,\chi_2}(0).$$
    But by Definition \ref{berndt1}
    $$-B_{1,\chi_2}(0)=\frac{\tau(\overline{\chi_2})}{2\pi i}\sum_{n\neq 0}\frac{\chi_2(n)}{n}=\begin{cases}
        \dfrac{\tau(\overline{\chi_2})}{\pi i}\,L(1,\chi_2) & \chi_2(-1)=-1 \\
        0 & \chi_2(-1)=+1.
    \end{cases}$$
    However, recalling the functional equation of $L(s,\chi)$, we have that
    $$\frac{\pi i^\delta\,L(0,\overline{\chi_2})}{\tau(\overline{\chi_2})}=\sqrt{\pi}\left(\frac{\Gamma((1+\delta)/2)}{\Gamma(\delta/2)}\right)\,L(1,\chi_2)\qquad\text{where}\qquad\sqrt{\pi}\left(\frac{\Gamma((1+\delta)/2)}{\Gamma(\delta/2)}\right)=\begin{cases}
        1 & \delta=0 \\
        0 & \delta=1
    \end{cases}$$
    where $\chi_2(-1)=(-1)^\delta$ with $\delta=0$ or $1$. Thus
    $$\lim_{u\to 0^+}\sum_{1\leq A}\frac{\chi_1(A)}{A^{k-1}}\sum_{ 1\leq B}\overline{\chi_2}(B)\,e^{-2\pi ABu}=L(k-1,\chi_1)\,L(0,\overline{\chi_2}).$$
    Now if we use the analogue of Corollary \ref{outerSumLimitCor}, by following the proof of Theorem \ref{mainThm} in a similar manner we arrive at
    \begin{align*}
        \widehat{S}_{\chi_1,\chi_2,k}(0)=\frac{\tau(\overline{\chi_1})\,(k-1)!}{\pi i}\left(\frac{1}{2\pi i}\right)^{k-2}\lim_{u\to 0^+}\sum_{1\leq A}\frac{\chi_1(A)}{A^{k-1}}\sum_{1\leq B}\overline{\chi_2}(B)\,e^{-2\pi ABu}.
    \end{align*}
    We computed this limit above, and substituing this the desired result follows.
\end{proof}

\noindent Now we relate $\widehat{S}_{\chi_1,\chi_2,k}(\mathfrak{a})$ and $h_{\omega\gamma,\chi_1,\chi_2,k}(\mathfrak{a})$ with $\phi_{\chi_2,\chi_1,k}((\omega\gamma)^{-1},1,-\mathfrak{a})$ in a manner analogous to Lemma \ref{qmfLemma}.
\begin{lemma}\label{qmfLemma2}
    For $\gamma\in\Gamma_0(q_1q_2)$ and $\mathfrak{a}\in((\Gamma_0(q_1q_2))(\infty)\cup(\omega\Gamma_0(q_1q_2))(\infty))\setminus\{\infty\}$ we have
    \begin{align*}\label{equqmfLemma2}
        h_{\omega\gamma,\chi_1,\chi_2,k}(\mathfrak{a})&=(-1)^k\tau(\overline{\chi_1})(k-1)\,R_{\chi_1,\chi_2,k}(\gamma)\,\phi_{\chi_2,\chi_1,k}((\omega\gamma)^{-1},1,-\mathfrak{a}) \\
        &\qquad-R_{\chi_1,\chi_2,k}(\gamma)\,(\tau(\overline{\chi_1})/\tau(\overline{\chi_2}))\,\widehat{S}_{\chi_2,\chi_1,k}(\mathfrak{a})+\widehat{S}_{\chi_1,\chi_2,k}(\mathfrak{a}) \numberthis
    \end{align*}
    where $R_{\chi_1,\chi_2,k}(\gamma)=\overline{\psi}(\gamma)\,R_{\chi_1,\chi_2,k}$.
\end{lemma}
\begin{proof}
    Note that
    \begin{equation}\label{equ8}
        \frac{h_{\omega\gamma,\chi_1,\chi_2,k}(\mathfrak{a})}{(-1)^k\tau(\overline{\chi_1})(k-1)}=\int_\infty^\mathfrak{a}E_{\chi_1,\chi_2,k}(z)\,P_{k-2}(z;1,-\mathfrak{a})\,dz-j(\omega\gamma,\mathfrak{a})^{k-2}\int_\infty^{\omega\gamma\mathfrak{a}}E_{\chi_1,\chi_2,k}(z)\,P_{k-2}(z;1,-\omega\gamma\mathfrak{a})\,dz.
    \end{equation}
    Applying Lemma \ref{integralCOV2} and Lemma \ref{weisinger} we have
    \begin{align*}
        \int_\infty^{\omega\gamma\mathfrak{a}}E_{\chi_1,\chi_2,k}(z)\,P_{k-2}(z;1,-\omega\gamma\mathfrak{a})\,dz&=\frac{1}{j(\omega,\gamma\mathfrak{a})^{k-2}}\int_{\omega^{-1}\infty}^{\gamma\mathfrak{a}}E_{\chi_1,\chi_2,k}|_k\omega(z)\,P_{k-2}(z;1,-\gamma\mathfrak{a})\,dz \\
        &=\frac{R_{\chi_1,\chi_2,k}}{j(\omega,\gamma\mathfrak{a})^{k-2}}\int_{\omega^{-1}\infty}^{\gamma\mathfrak{a}}E_{\chi_2,\chi_1,k}(z)\,P_{k-2}(z;1,-\gamma\mathfrak{a})\,dz. \numberthis \label{equ9}
    \end{align*}
    Applying Lemma \ref{integralCOV2} again and recalling $E_{\chi_1,\chi_2,k}|_k\gamma(z)=\psi(\gamma)\,E_{\chi_1,\chi_2,k}(z)$, we have
    \begin{align*}
        \int_{\omega^{-1}\infty}^{\gamma\mathfrak{a}}E_{\chi_2,\chi_1,k}(z)\,P_{k-2}(z;1,-\gamma\mathfrak{a})\,dz&=\frac{1}{j(\gamma,\mathfrak{a})^{k-2}}\int_{(\omega\gamma)^{-1}\infty}^\mathfrak{a}E_{\chi_2,\chi_1,k}|_k\gamma(z)\,P_{k-2}(z;1,-\mathfrak{a})\,dz \\
        &=\frac{\overline{\psi}(\gamma)}{j(\gamma,\mathfrak{a})^{k-2}}\int_{(\omega\gamma)^{-1}\infty}^\mathfrak{a}E_{\chi_2,\chi_1,k}(z)\,P_{k-2}(z;1,-\mathfrak{a})\,dz. \numberthis \label{equ10}
    \end{align*}
    Putting together (\ref{equ9}) and (\ref{equ10}) we have that
    $$\int_\infty^{\omega\gamma\mathfrak{a}}E_{\chi_1,\chi_2,k}(z)\,P_{k-2}(z;1,-\omega\gamma\mathfrak{a})\,dz=\frac{R_{\chi_1,\chi_2,k}(\gamma)}{j(\omega\gamma,\mathfrak{a})^{k-2}}\int_{(\omega\gamma)^{-1}\infty}^\mathfrak{a}E_{\chi_2,\chi_1,k}(z)\,P_{k-2}(z;1,-\mathfrak{a})\,dz.$$
    Now substituting this into (\ref{equ8}) we get that
    \begin{align*}
        \frac{h_{\omega\gamma,\chi_1,\chi_2,k}(\mathfrak{a})}{(-1)^k\tau(\overline{\chi_1})(k-1)}&=\int_\infty^\mathfrak{a}E_{\chi_1,\chi_2,k}(z)\,P_{k-2}(z;1,-\mathfrak{a})\,dz+R_{\chi_1,\chi_2,k}(\gamma)\int_\mathfrak{a}^{(\omega\gamma)^{-1}\infty}E_{\chi_2,\chi_1,k}(z)\,P_{k-2}(z;1,-\mathfrak{a})\,dz.
    \end{align*}
    Adding and subtracting to the right side we have
    \begin{align*}
        \frac{h_{\omega\gamma,\chi_1,\chi_2,k}(\mathfrak{a})}{(-1)^k\tau(\overline{\chi_1})(k-1)}&=R_{\chi_1,\chi_2,k}(\gamma)\,\phi_{\chi_2,\chi_1,k}((\omega\gamma)^{-1},1,-\mathfrak{a})-R_{\chi_1,\chi_2,k}(\gamma)\frac{\widehat{S}_{\chi_2,\chi_1,k}(\mathfrak{a})}{(-1)^k\tau(\overline{\chi_2})(k-1)}+\frac{\widehat{S}_{\chi_1,\chi_2,k}(\mathfrak{a})}{(-1)^k\tau(\overline{\chi_1})(k-1)}.
    \end{align*}
    Multiplying through by $(-1)^k\tau(\overline{\chi_1})(k-1)$ completes the result.
\end{proof}

Now we can begin proving our reciprocity formulae coming from the Fricke involution. We begin by proving a relation on $h_{\gamma,\chi_1,\chi_2,k}$,
\begin{lemma}\label{threeTerm}
    For $\gamma=\begin{psmallmatrix}
        a & b \\ cq_1q_2 & d
    \end{psmallmatrix}\in\Gamma_0(q_1q_2)$ and $\gamma'=\begin{psmallmatrix}
        d & -c \\ -bq_1q_2 & a
    \end{psmallmatrix}\in\Gamma_0(q_1q_2)$ we have that
    \begin{equation}\label{equThreeTerm}0=h_{\gamma,\chi_1,\chi_2,k}|_{2-k}\omega-h_{\gamma',\chi_1,\chi_2,k}+(h_{\omega,\chi_1,\chi_2,k}-h_{\omega,\chi_1,\chi_2,k}|_{2-k}\gamma').\end{equation}
    for all $\mathfrak{a}\in((\Gamma_0(q_1q_2))(\infty)\cup(\omega\Gamma_0(q_1q_2))(\infty))\setminus\{\infty\}$.
\end{lemma}
\begin{proof}
    Since $\gamma\omega=\omega\gamma'$, we have that
    $h_{\gamma\omega,\chi_1,\chi_2,k}=h_{\omega\gamma',\chi_1,\chi_2,k}$. Note that
    \begin{align*}
        h_{\gamma\omega,\chi_1,\chi_2,k}&=\widehat{S}_{\chi_1,\chi_2,k}-\widehat{S}_{\chi_1,\chi_2,k}|_{2-k}\gamma\omega \\
        &=\left(\widehat{S}_{\chi_1,\chi_2,k}-\widehat{S}_{\chi_1,\chi_2,k}|_{2-k}\omega\right)+\left(\widehat{S}_{\chi_1,\chi_2,k}-\widehat{S}_{\chi_1,\chi_2,k}|_{2-k}\gamma\right)|_{2-k}\omega=h_{\omega,\chi_1,\chi_2,k}+h_{\gamma,\chi_1,\chi_2,k}|_{2-k}\omega.
    \end{align*}
    Likewise,
    \begin{align*}
        h_{\omega\gamma',\chi_1,\chi_2,k}&=\widehat{S}_{\chi_1,\chi_2,k}-\widehat{S}_{\chi_1,\chi_2,k}|_{2-k}\omega\gamma' \\
        &=\left(\widehat{S}_{\chi_1,\chi_2,k}-\widehat{S}_{\chi_1,\chi_2,k}|_{2-k}\gamma'\right)+\left(\widehat{S}_{\chi_1,\chi_2,k}-\widehat{S}_{\chi_1,\chi_2,k}|_{2-k}\omega\right)|_{2-k}\gamma'=h_{\gamma',\chi_1,\chi_2,k}+h_{\omega,\chi_1,\chi_2,k}|_{2-k}\gamma'.
    \end{align*}
    Since $h_{\gamma\omega,\chi_1,\chi_2,k}=h_{\omega\gamma',\chi_1,\chi_2,k}$, the result follows.
\end{proof}

\noindent Now we prove Theorem \ref{reciprocity} by putting together Lemmas \ref{qmfLemma}, \ref{qmfLemma2}, and \ref{threeTerm}.

\begin{proof}
    Substituting (\ref{equqmfLemma}) from Lemma \ref{qmfLemma} and (\ref{equqmfLemma2}) from Lemma \ref{qmfLemma2} into (\ref{equThreeTerm}) from Lemma \ref{threeTerm} we have
    \begin{align*}
        0&=(-1)^k\tau(\overline{\chi_1})(k-1)\,\psi(\gamma)\,\phi_{\chi_1,\chi_2,k}|_{2-k}\omega(\gamma^{-1},1,-\mathfrak{a})+(1-\psi(\gamma))\,\widehat{S}_{\chi_1,\chi_2,k}|_{2-k}\omega(\mathfrak{a}) \\
        &\qquad-(-1)^k\tau(\overline{\chi_1})(k-1)\,\psi(\gamma')\,\phi_{\chi_1,\chi_2,k}(\gamma'^{-1},1,-\mathfrak{a})-(1-\psi(\gamma'))\,\widehat{S}_{\chi_1,\chi_2,k}(\mathfrak{a}) \\
        &+(-1)^k\tau(\overline{\chi_1})(k-1)\,R_{\chi_1,\chi_2,k}\,\phi_{\chi_2,\chi_1,k}(\omega^{-1},1,-\mathfrak{a}) \\
        &\qquad-R_{\chi_1,\chi_2,k}\cdot(\tau(\overline{\chi_1})/\tau(\overline{\chi_2}))\,\widehat{S}_{\chi_2,\chi_1,k}(\mathfrak{a})+\widehat{S}_{\chi_1,\chi_2,k}(\mathfrak{a}) \\
        &-(-1)^k\tau(\overline{\chi_1})(k-1)\,R_{\chi_1,\chi_2,k}\,\phi_{\chi_2,\chi_1,k}|_{2-k}\gamma'(\omega^{-1},1,-\mathfrak{a}) \\
        &\qquad+R_{\chi_1,\chi_2,k}\cdot\,(\tau(\overline{\chi_1})/\tau(\overline{\chi_2}))\,\widehat{S}_{\chi_2,\chi_1,k}|_{2-k}\gamma'(\mathfrak{a})-\widehat{S}_{\chi_1,\chi_2,k}|_{2-k}\gamma'(\mathfrak{a}).
    \end{align*}
    Rearranging this gives
    \begin{align*}
        &\widehat{S}_{\chi_1,\chi_2,k}|_{2-k}\gamma'(\mathfrak{a})-\psi(\gamma')\,\widehat{S}_{\chi_1,\chi_2,k}(\mathfrak{a}) \\
        &\qquad+(-1)^k\tau(\overline{\chi_1})(k-1)\left(\psi(\gamma')\,\phi_{\chi_1,\chi_2,k}(\gamma'^{-1},1,-\mathfrak{a})-\psi(\gamma)\,\phi_{\chi_1,\chi_2,k}|_{2-k}\omega(\gamma^{-1},1,-\mathfrak{a})\right) \\
        &=R_{\chi_1,\chi_2,k}\cdot(\tau(\overline{\chi_1})/\tau(\overline{\chi_2}))\left(\widehat{S}_{\chi_2,\chi_1,k}|_{2-k}\gamma'(\mathfrak{a})-\widehat{S}_{\chi_2,\chi_1,k}(\mathfrak{a})\right) \\
        &\qquad+(-1)^k\tau(\overline{\chi_1})(k-1)\,R_{\chi_1,\chi_2,k}\cdot\left(\phi_{\chi_2,\chi_1,k}(\omega^{-1},1,-\mathfrak{a})-\phi_{\chi_2,\chi_1,k}|_{2-k}\gamma'(\omega^{-1},1,-\mathfrak{a})\right) \\
        &\qquad+(1-\psi(\gamma))\,\widehat{S}_{\chi_1,\chi_2,k}|_{2-k}\omega(\mathfrak{a}).
    \end{align*}
    as desired. 
\end{proof}
\noindent Despite the fact that the above holds for general $k$ over $\mathfrak{a}\in((\Gamma_0(q_1q_2))(\infty)\cup(\omega\Gamma_0(q_1q_2))(\infty))\setminus\{\infty\}$, it turns out the above holds for all $\mathfrak{a}\in(\Gamma_0(q_1q_2))(\infty)\cup(\omega\Gamma_0(q_1q_2))(\infty)$ in the case that $k=2$. Recalling Definition \ref{shatDef} and noting that $\widehat{S}_{\chi_1,\chi_2,k}|_{2-k}\gamma(\mathfrak{a})=\widehat{S}_{\chi_1,\chi_2,k}(\gamma\mathfrak{a})$ when $k=2$, it follows that we can extend $h_{\gamma,\chi_1,\chi_2,2}$ and $h_{\omega\gamma,\chi_1,\chi_2,2}$ to all of $(\Gamma_0(q_1q_2))(\infty)\cup(\omega\Gamma_0(q_1q_2))(\infty)$. So, when $k=2$, this extension allows us to recover the above identity over all $\mathfrak{a}\in(\Gamma_0(q_1q_2))(\infty)\cup(\omega\Gamma_0(q_1q_2))(\infty)$. Using this fact and specializing $k=2$ and $\mathfrak{a}=\infty$ allows us to recover a reciprocity identity derived in \cite{SVY} with a new method of proof.
\begin{corollary}[\cite{SVY}, Theorem $1.3$]\label{SVYFrickeRecovery}
    For $\gamma=\begin{psmallmatrix}
        a & b \\ cq_1q_2 & d
    \end{psmallmatrix}\in\Gamma_0(q_1q_2)$ and $\gamma'=\begin{psmallmatrix}
        d & -c \\ -bq_1q_2 & a
    \end{psmallmatrix}\in\Gamma_0(q_1q_2)$ we have
    $$S_{\chi_1,\chi_2,2}(\gamma)=\chi_1(-1)\,S_{\chi_2,\chi_1,2}(\gamma')+(1-\psi(\gamma))\left(\frac{\tau(\overline{\chi_1})}{\pi i}\right)\,L(1,\chi_1)\,L(0,\overline{\chi_2}).$$
\end{corollary}
\begin{proof}
    When $k=2$ note that $f|_{2-k}\gamma(\mathfrak{a})=f(\gamma\mathfrak{a})$. Recall Theorem \ref{reciprocity}, which for $k=2$ states:
    \begin{align*}
        &\widehat{S}_{\chi_1,\chi_2,2}(\gamma'\mathfrak{a})-\psi(\gamma')\,\widehat{S}_{\chi_1,\chi_2,2}(\mathfrak{a}) \\
        &\qquad+\tau(\overline{\chi_1})\left(\psi(\gamma')\,\phi_{\chi_1,\chi_2,2}(\gamma'^{-1},1,-\mathfrak{a})-\psi(\gamma)\,\phi_{\chi_1,\chi_2,2}(\gamma^{-1},1,-\omega\mathfrak{a})\right) \\
        &=R_{\chi_1,\chi_2,2}\cdot(\tau(\overline{\chi_1})/\tau(\overline{\chi_2}))\left(\widehat{S}_{\chi_2,\chi_1,2}(\gamma'\mathfrak{a})-\widehat{S}_{\chi_2,\chi_1,2}(\mathfrak{a})\right) \\
        &\qquad+\tau(\overline{\chi_1})\,R_{\chi_1,\chi_2,2}\cdot\left(\phi_{\chi_2,\chi_1,2}(\omega^{-1},1,-\mathfrak{a})-\phi_{\chi_2,\chi_1,2}(\omega^{-1},1,-\gamma'\mathfrak{a})\right) \\
        &\qquad+(1-\psi(\gamma))\,\widehat{S}_{\chi_1,\chi_2,2}(\mathfrak{\omega a}).
    \end{align*}
    Since $\phi_{\chi_1,\chi_2,2}$ and $\phi_{\chi_2,\chi_1,2}$ are independent in $\mathfrak{a}$ we have
    $$\phi_{\chi_2,\chi_1,2}(\omega^{-1},1,-\mathfrak{a})-\phi_{\chi_2,\chi_1,2}(\omega^{-1},1,-\gamma'\mathfrak{a})=0,$$
    $\phi_{\chi_1,\chi_2,2}(\gamma'^{-1},1,-\mathfrak{a})=\phi_{\chi_1,\chi_2,2}(\gamma'^{-1},1,-\gamma'^{-1}\mathfrak{a})$, and $\phi_{\chi_1,\chi_2,2}(\gamma^{-1},1,-\omega\mathfrak{a})=\phi_{\chi_1,\chi_2,2}(\gamma^{-1},1,-\gamma^{-1}\mathfrak{a})$. Substituting these three results into the above restatement of Theorem \ref{reciprocity} we have that
    \begin{align*}
        &\widehat{S}_{\chi_1,\chi_2,2}(\gamma'\mathfrak{a})-\psi(\gamma')\,\widehat{S}_{\chi_1,\chi_2,2}(\mathfrak{a})+\tau(\overline{\chi_1})\left(\psi(\gamma')\,\phi_{\chi_1,\chi_2,2}(\gamma'^{-1},1,-\gamma'^{-1}\mathfrak{a})-\psi(\gamma)\,\phi_{\chi_1,\chi_2,2}(\gamma^{-1},1,-\gamma^{-1}\mathfrak{a})\right) \\
        &\qquad=R_{\chi_1,\chi_2,2}\cdot(\tau(\overline{\chi_1})/\tau(\overline{\chi_2}))\left(\widehat{S}_{\chi_2,\chi_1,2}(\gamma'\mathfrak{a})-\widehat{S}_{\chi_2,\chi_1,2}(\mathfrak{a})\right)+(1-\psi(\gamma))\,\widehat{S}_{\chi_1,\chi_2,2}(\omega\mathfrak{a}).
    \end{align*}
    Now specializing $\mathfrak{a}=\infty$ and recalling $\widehat{S}_{\chi_1,\chi_2,2}(\infty)=0$, $\widehat{S}_{\chi_1,\chi_2,2}(\gamma\infty)=S_{\chi_1,\chi_2,2}(\gamma)$, and Definition \ref{Sdef} one gets
    \begin{align*}
        &S_{\chi_1,\chi_2,2}(\gamma')+\psi(\gamma')\,S_{\chi_1,\chi_2,2}(\gamma'^{-1})-\psi(\gamma)\,S_{\chi_1,\chi_2,2}(\gamma^{-1}) \\
        &\qquad=R_{\chi_1,\chi_2,2}\cdot(\tau(\overline{\chi_1})/\tau(\overline{\chi_2}))\,S_{\chi_2,\chi_1,2}(\gamma')+(1-\psi(\gamma))\,\widehat{S}_{\chi_1,\chi_2,2}(0).
    \end{align*}
    Now note that $S_{\chi_1,\chi_2,2}(\gamma')+\psi(\gamma')\,S_{\chi_1,\chi_2,2}(\gamma'^{-1})=S_{\chi_1,\chi_2,2}\begin{psmallmatrix}
        1 & 0 \\ 0 & 1
    \end{psmallmatrix}=0$ by Corollary \ref{crossHomSpecial}.
    Similarly,
    $$-\psi(\gamma)\,S_{\chi_1,\chi_2,2}(\gamma^{-1})=S_{\chi_1,\chi_2,2}(\gamma)-(S_{\chi_1,\chi_2,2}(\gamma)+\psi(\gamma)\,S_{\chi_1,\chi_2,2}(\gamma^{-1}))=S_{\chi_1,\chi_2,2}(\gamma)-S_{\chi_1,\chi_2,2}\begin{psmallmatrix}
        1 & 0 \\ 0 & 1
    \end{psmallmatrix}=S_{\chi_1,\chi_2,2}(\gamma)$$
    by Corollary \ref{crossHomSpecial}. So we have that
    $$S_{\chi_1,\chi_2,2}(\gamma)=R_{\chi_1,\chi_2,2}\cdot(\tau(\overline{\chi_1})/\tau(\overline{\chi_2}))\,S_{\chi_2,\chi_1,2}(\gamma')+(1-\psi(\gamma))\,\widehat{S}_{\chi_1,\chi_2,2}(0).$$
    Noting that $R_{\chi_1,\chi_2,2}\cdot(\tau(\overline{\chi_1})/\tau(\overline{\chi_2}))=\chi_1(-1)$, by substituting Lemma \ref{Sfricke} we get the desired result.
\end{proof}

\section{Code}\label{codeSec}
    The work described in this paper has been implemented using \code{SageMath.}\ The reader can find this code at:
    
    \url{https://github.com/prestontranbarger/NFDSFastComputation}

    \noindent Once the repository is cloned, the reader can navigate to Section $3$ of \code{README.md}\ to gain access to necessary information within the repository which is relevant to the contents of this paper.
    
\section*{Acknowledgements}
    The author would like to thank Dr. Matthew Young for his extensive support and guidance throughout this research. This material is based, in part, upon work supported by the National Science Foundation Graduate Research Fellowship under Grant No. 2233066. Any opinion, findings, and conclusions or recommendations expressed in this material are those of the author and do not necessarily reflect the views of the National Science Foundation.

\bibliographystyle{alpha}
\bibliography{bib}

\end{document}